\documentclass{amsart}

\usepackage{gensymb,epic,eepic,float,fullpage}
\usepackage{longtable}
\usepackage[usenames,dvipsnames]{color}
\usepackage{siunitx}
\usepackage{empheq}
\usepackage[mathscr]{euscript}
\usepackage[T1]{fontenc}
\usepackage{comment}
\usepackage{caption}
\usepackage{enumerate,enumitem}
\usepackage{amsmath,amsfonts,amssymb,amsthm}
\usepackage{mathtools}
\usepackage[colorinlistoftodos]{todonotes}
\usepackage{cancel}
\usepackage{stmaryrd}
\usepackage{url}
\usepackage{tikz}
\usepackage[outline]{contour}
\usetikzlibrary{graphs,patterns,decorations.markings,arrows.meta,matrix}
\usetikzlibrary{calc,decorations.pathmorphing,decorations.pathreplacing,shapes}
\usepackage{graphicx}

\usepackage{ytableau}

\usepackage{hyperref}

\usepackage{bm}
\usepackage{bbm}
\usepackage{mathrsfs}

\colorlet{lgray}{white!85!black}
\colorlet{dgray}{white!45!black}
\colorlet{lred}{white!85!red}
\colorlet{dred}{white!35!red}
\colorlet{lgreen}{white!60!green}
\colorlet{dgreen}{black!30!green}
\colorlet{lpurple}{white!60!purple}
\colorlet{lblue}{white!60!blue}
\definecolor{green}{rgb}{0.1,0.8,0.1}
\definecolor{yellow}{rgb}{1.0,0.85,0.25}
\definecolor{purple}{rgb}{1.0, 0, 1.0}
\definecolor{blue}{rgb}{0, 0, 1.0}

 \renewcommand{\tikz}[2]{
\begin{tikzpicture}[scale=#1,baseline=(current bounding box.center),>=stealth]
#2
\end{tikzpicture}}

\newcommand\restr[2]{{
  \left.\kern-\nulldelimiterspace 
  #1 
  \vphantom{\big|} 
  \right|_{#2} 
 }}

\DeclarePairedDelimiter\floor{\lfloor}{\rfloor}

\tikzstyle{fused}=[lgray, line width=3pt, arrows={-Stealth[scale=1.1,length=10, width=10,dgray]}]
\tikzstyle{fused*}=[lgray, line width=3pt]

\newtheorem{prop}{Proposition}[section]
\newtheorem*{prop*}{Proposition}
\newtheorem{theo}[prop]{Theorem}
\newtheorem*{theo*}{Theorem}
\newtheorem{theoAlph}{Theorem}

\newtheorem{conj}[prop]{Conjecture}
\newtheorem{lem}[prop]{Lemma}
\newtheorem{cor}[prop]{Corollary}
\newtheorem{quest}[prop]{Question}
\theoremstyle{definition}
\newtheorem{defin}[prop]{Definition}

\newtheorem{rem}[prop]{Remark}

\numberwithin{equation}{section}

\renewcommand{\(}{\left (}
\renewcommand{\)}{\right )}
\newcommand{\1}{\mathbbm 1}
\renewcommand{\i}{\mathbf i}

\newcommand{\ba}{\mathbf a}
\newcommand{\bu}{\mathbf u}
\newcommand{\bx}{\mathbf x}
\newcommand{\by}{\mathbf y}
\newcommand{\bz}{\mathbf z}
\newcommand{\bt}{\mathbf t}

\newcommand{\fa}{\mathfrak a}
\newcommand{\fb}{\mathfrak b}

\newcommand{\calM}{\mathcal M}

\newcommand \calA{\mathcal A}
\newcommand \calB{\mathcal B}
\newcommand \calC{\mathcal C}
\newcommand \calP{\mathcal P}
\newcommand \calT{\mathcal T}

\newcommand \F{\mathbb F}
\newcommand \C{\mathbb C}

\newcommand \R{\mathbb R}
\renewcommand \P{\mathbb P}

\newcommand \Z{\mathbb Z}

\newcommand \G{\mathbb{G}}

\newcommand{\eps}{\varepsilon}

\let\Re\relax
\DeclareMathOperator{\Re}{Re}
\let\Im\relax
\DeclareMathOperator{\Im}{Im}

\title{On the Martin boundary for discrete TASEP}

\author{Vadim Gorin}

\address[V.G.]{Departments of Statistics and Mathematics, University of California, Berkeley. vadicgor@gmail.com}

\author{Sergei Korotkikh}

\address[S.K.]{Department of Mathematics, University of California, Berkeley. korotkikh@berkeley.edu}

\begin{document}

\maketitle

\begin{abstract}
 We study a problem with three equivalent formulations: describing Gibbs measures for five-vertex model in quadrant; classifying coherent systems on a $p$--deformation of the Gelfand-Tsetlin graph related to Grothendieck polynomials; finding the Martin boundary for discrete time TASEP with $p$-geometric jumps. We find a wide family of the Gibbs measures, parameterized by certain analytic functions. A subset of our measures have probabilistic interpretation as interacting particle systems with fixed particles speeds. In contrast to previous related boundary problems, we find that admissible speeds are not arbitrary, but must be larger than $\frac{p}{1-p}$. For this subset we further establish Law of Large Numbers and Central Limit Theorem, connecting the fluctuations to families of independent GUE eigenvalues. As a consequence, the measures from the subset are extreme points of the Martin boundary. It remains open whether our list of measures is exhaustive.
\end{abstract}

\setcounter{tocdepth}{2}
\tableofcontents

\section{Introduction}

\subsection{Overview}

This paper is about a classification problem: we want to describe all probability measures on configurations of paths in the quadrant which satisfy a certain Gibbs property depending on a parameter $0\le p \le 1$. Our motivations come from two directions: asymptotic representation theory and 2d statistical mechanics.

When $p=0$, our Gibbs probability measures are in bijection with coherent systems on (positive part of) the Gelfand--Tsetlin graph. In turn, the latter are in correspondence with characters of the infinite-dimensional unitary group $U(\infty)$, its spherical and finite-factor representations, as well as with infinite totally-positive Toeplitz matrices, see  \cite[Section 9]{olshanski2016markov}, \cite[Section 20.3]{Gorin_book}. Because of these connections, the $p=0$ classification problem is very well-studied, see \cite{edrei1953generating,voiculescu1976representations,vershik1982characters,boyer1983infinite,okounkov1998asymptotics,borodin2012boundary,
petrov2014boundary,gorin2015asymptotics} for various approaches. The $p=0$ problem can also be reformulated as a study of all possible limits of normalized Schur polynomials $s_\lambda$ as the number of variables grows to infinity:
\begin{equation}
\label{eq_Schur_limits}
 \lim_{N\to\infty} \frac{s_{\lambda(N)}(x_1,\dots,x_k,1^{N-k})}{s_{\lambda(N)}(1^N)}=?
\end{equation}
Two deformations of \eqref{eq_Schur_limits} were explored in the literature: the first one depends on a real parameter $\theta>0$ and replaces Schur polynomials with Jack polynomials, motivated by connections to spherical representations of Gelfand pairs at $\theta=1/2,1,2$ and to log-gases and $\beta$-ensembles of the random matrix theory. The answer in the Jack-deformed problem turns out to be very similar to the Schur $\theta=1$ case, see \cite{okounkov1998asymptotics}. Another $q$--deformation replaces $1$s in \eqref{eq_Schur_limits} with geometric series with denominator $q>0$ and relates to representation theory of quantum groups \cite{sato2019quantized,sato2021inductive} and to $q^{\text{volume}}$--weighted random plane partitions \cite{gorin2012q}. Here the role of $q$ turned out to be more significant and the answer for $q=1$ is very different from the general $q$ case of \cite{gorin2012q, gorin2016quantization}. The two deformations were subsequently lifted to a common generalization related to principal specializations of Macdonald polynomials in \cite{cuenca2018asymptotic,olshanski2021macdonald}.

From the asymptotic representation theory perspective, in this paper we initiate the study of another deformation of \eqref{eq_Schur_limits}, related to \emph{Grothendieck polynomials} $G_\lambda$, see e.g.\ \cite{fomin1994grothendieck,fomin1996yang,ikeda2013k,motegi2013vertex,Yel16,HJKSS21} among many papers on these interesting polynomials. In contrast to all the previously studied cases from the last paragraph, the denominator in our version of \eqref{eq_Schur_limits} no longer has an explicit fully factorized form, which suggests that less formulas are available for our $p$--deformation and leads us to develop alternative methods.

\medskip

Switching to the statistical mechanics point of view, we recall that configurations of the celebrated \emph{six-vertex model} (see \cite{Lieb_ferroelectric_models,baxter2007exactly,reshetikhin2010lectures,Bleher-Liechty14,Gorin_Nicoletti_lectures} for the reviews) assign to each vertex of a domain $\Omega \subset \mathbb{Z}^2$ one of the six types shown in Figure \ref{Fig_vertex_weights}, in a way that is globally consistent: the result must form non-intersecting, possibly touching paths that connect specified boundary points of $\Omega$. We assign a positive weight $w_i > 0$ to each vertex type $i = 1,\dots,6$, typically denoted by $(a_1, a_2, b_1, b_2, c_1, c_2)$, and say that a probability measure on configurations in $\Omega$ is Gibbs, if for any finite subdomain $\Omega'\subset \Omega$ and any boundary conditions --- configuration of boundary points $\partial \Omega'$ which the paths should connect, the conditional distribution of configurations $\sigma$ inside $\Omega'$ has the form
\begin{equation}
\label{eq_Gibbs}
 \mathrm{Prob}\bigl(\sigma\mid \text{boundary condition on }\partial \Omega'\bigr) = \frac{1}{Z} \prod_{i=1}^6 w_i^{N_i(\sigma)},
\end{equation}
where $N_i(\sigma)$ is the number of type $i$ vertices in configuration $\sigma$ and $Z$ is a normalizing constant. The formula \eqref{eq_Gibbs} silently assumes that the particular boundary condition on $\partial \Omega'$ has a non-zero probability. We remark that the Gibbs property \eqref{eq_Gibbs} depends on two, rather than six parameters due to four conservation laws, see \cite[Lemma 2.1]{Gorin_Nicoletti_lectures}.

\begin{figure}[t]
  \includegraphics[width=0.42\textwidth]{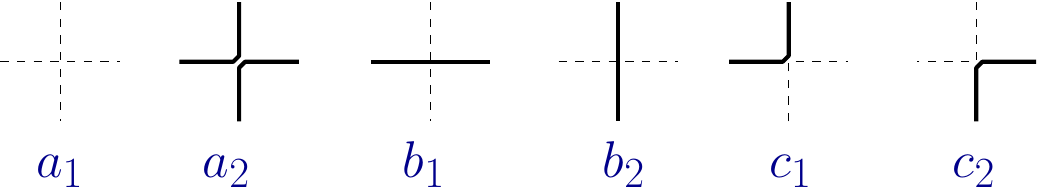}
   \hfill
  \includegraphics[width=0.42\textwidth]{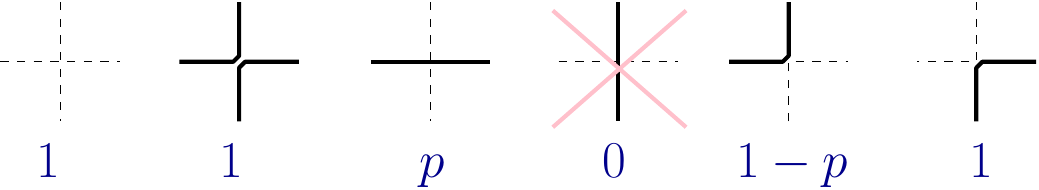}
    \caption{Left panel: six vertex weights. Right panel: stochastic five-vertex weights. \label{Fig_vertex_weights}}
\end{figure}

\begin{figure}[t]
  \includegraphics[width=0.45\textwidth]{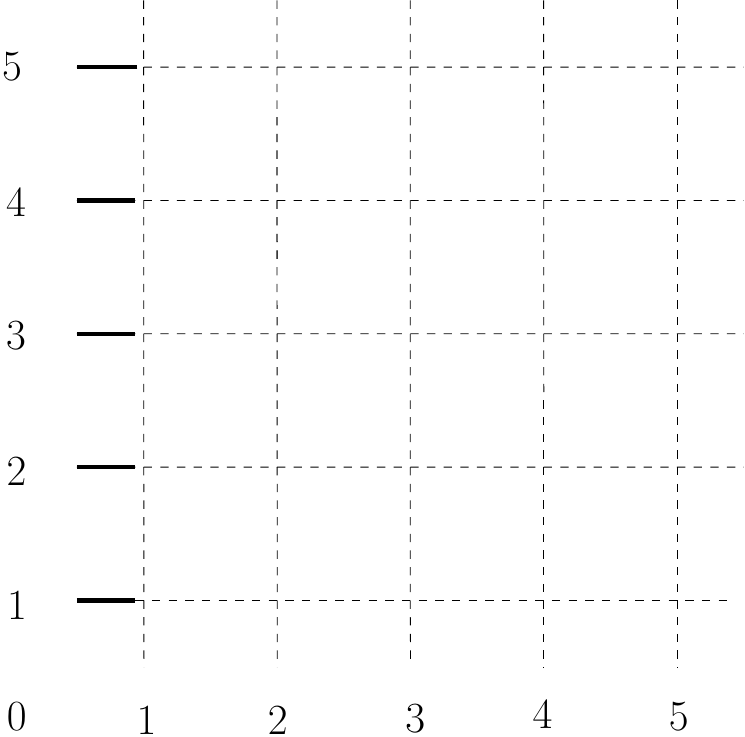 }
   \hfill
  \includegraphics[width=0.45\textwidth]{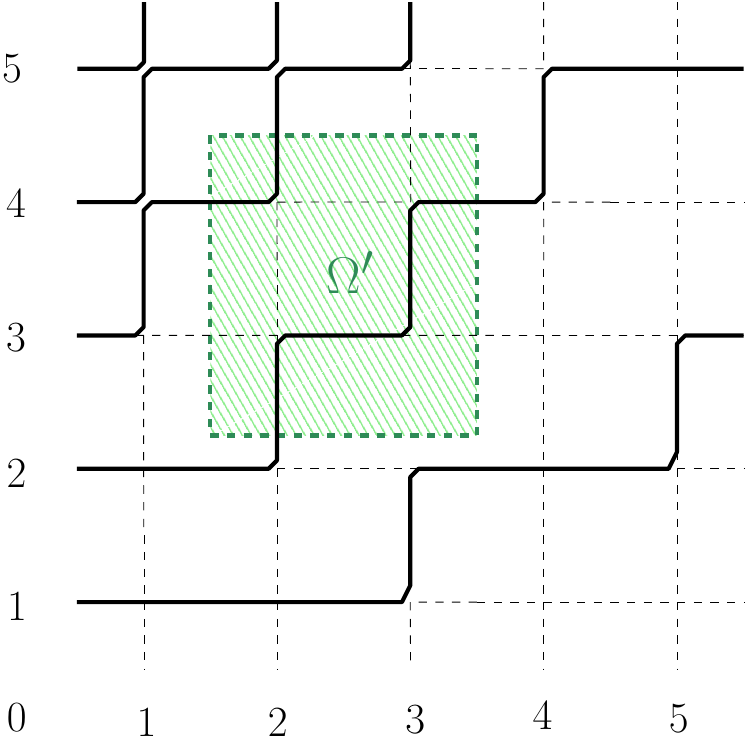 }
    \caption{Left: boundary conditions in quadrant. Right: a configuration and a subdomain $\Omega'$. With given boundary conditions $\Omega'$ has two configurations of conditional probabilities $\frac{c_1^2 c_2^2}{c_1^2 c_2^2+a_1 a_2 b_1 b_2}$ and $\frac{a_1 a_2 b_1 b_2}{c_1^2 c_2^2+a_1 a_2 b_1 b_2}$ --- the second one vanishes for the five-vertex weights. \label{Fig_5v_configurations}}
\end{figure}

We choose $\Omega$ to be the quadrant with step initial condition of Figure \ref{Fig_5v_configurations}: the paths enter on every site from the left and do not enter from below; this is an infinite version of the domain wall boundary conditions.

\smallskip

{\bf Question 6v.} What are possible Gibbs measures for the six-vertex model in the quadrant?

\smallskip

We conjecture that the answer should be very different depending on the value of $\Delta:=\frac{a_1 a_2+b_1 b_2 - c_1 c_2}{2\sqrt{a_1 a_2 b_1 b_2}}$. Namely, for $\Delta<1$, we expect the only measure to be the unit mass on the configurations of all paths going straight to the right. On the other hand, for $\Delta>1$, we expect a rich family of Gibbs measures to exist. The prediction is based on a phase transition at $\Delta=1$ discovered in \cite{Gorin_Liechty_2023}.

In this paper, we deal with a particular case of Question 6v corresponding to $\Delta=+\infty$. Namely, we prohibit the $b_2$ vertex degenerating into the five vertex model, and arrange the weights of the remaining vertices as in the right panel of Figure \ref{Fig_vertex_weights}. This is an $\eps\to 0+$ limit of the stochastic weights (cf.\ \cite{gwa1992six,borodin2016stochastic}) $(1,1,p,\eps,1-p,1-\eps)$ with $\Delta=\frac{p+\eps}{2\sqrt{p\eps(1-p)(1-\eps)}}$. The Gibbs property for this instance of the five-vertex model can be also interpreted through transitional probabilities of discrete time TASEP with geometric jumps, see Section \ref{TASEPsec} for further details. Various questions about such TASEP were previously investigated e.g.\ in \cite{dieker2008determinantal,matetski2023tasep,motegi2013vertex,knizel2019generalizations}. In contrast to the general six-vertex model, our analysis in this situation is simplified by the existence of determinantal formulas for transition probabilities and connection to the Grothendieck polynomials.

\medskip

In line with conjectures about Question 6v, our main result is the construction of a rich class of Gibbs measures in the quadrant for the five vertex model, see Theorem \ref{coherent-thm} for further details. Some of these measures have probabilistic interpretations as time evolutions of TASEP started from the step initial condition and with prescribed asymptotic speeds of individual particles, see Theorem \ref{limittheo}. The phenomenon of asymptotic particle speeds is also known in $p=0$ case of the Gelfand-Tsetlin graph, since \cite{vershik1982characters}, which interpreted the parameters of Gibbs measures as normalized asymptotic lengths (i.e. speeds of growth) of rows and columns of corresponding randomly growing Young diagrams. However, there is a striking difference: in the Gelfand-Tsetlin graph any positive speeds are possible, but in our $p$--deformation we discovered only speeds larger than $\frac{p}{1-p}$ to be admissible.

We remark that the set of all Gibbs measures is convex, and therefore the distinguished role in classification is played by the ergodic measures, which are extreme points of this set. We do not yet know whether all the measures we constructed are ergodic, see Theorem \ref{Theorem_extreme}  for a partial result and Section \ref{Section_open_questions} for further discussion.

\bigskip

Classifications of Gibbs measures and coherent systems have two important consequences in the context of statistical mechanics and integrable probability. First, knowing Gibbs measures in specific geometries leads to predictions for the local limits of models in various domains. Along these lines, \cite{sheffield2005random} classified all translationally invariance Gibbs measures on lozenge tilings in terms of their slopes and subsequently \cite{aggarwal2023universality} (see also \cite{gorin2017bulk}) proved that these measures are the only possible bulk limits for uniformly random tilings of arbitrary domains. \cite{olshanski1996ergodic} classified Gibbs measures on spectra of corners of random matrices; based on that \cite{okounkov2006birth} predicted the universality of GUE-corners process as a limit of statistical mechanics models near boundaries;  \cite{aggarwal2022gaussian} proved this for lozenge tilings (and universality is expected to extend beyond, e.g., to the six-vertex model, cf.\ \cite{Gorin_Nicoletti_lectures,Gorin_Liechty_2023}). Okounkov and Sheffield predicted and \cite{corwin2014brownian,aggarwal2023strong} identified the Airy$_2$ line ensemble as a particular solution to a classification problem involving Brownian Gibbs property for a family of continuous curves, which was subsequently used as a tool for proving convergence towards the Airy$_2$ line ensemble. Similarly, we expect that our Gibbs measures (and, looking further ahead, the eventual answer to Question 6v) will describe possible local limits in the six-vertex model and its five-vertex degeneration.

Second, the Gibbs measures appearing as answers in the classification problems in infinite domains often turn out to be exactly solvable: many more formulas and algebraic structures are available for them, as compared to generic models of 2d statistical mechanics. For instance, the measures for the $p=0$ of our problem corresponding to the Gelfand-Tsetlin graph enjoy connections to determinantal point processes (e.g.~ \cite{borodin2008asymptotics}), Robinson-Schensted-Knuth correspondence
 (e.g. \cite{betea2018perfect} and references therein), to 2+1--dimensional interacting particle systems (e.g.\ \cite{borodin2014anisotropic}). Similarly, we expect that the answers to Question 6v and its five-vertex version are very special and worth further studies.

\bigskip

\subsection{Results and methods} Fix $p\in (0,1)$ and let $G_{\lambda/\mu}^{(0,-p)}$ denote the skew Grothendieck polynomials depending on a parameter $p$. $G_{\lambda/\mu}^{(0,-p)}$ are symmetric functions in variables $x_1,x_2,\dots$, see Section \ref{Gsect} for the definition. To formulate our results we use \emph{coherent systems}. These are collections $\{M_n\}_{n\geq 0}$ of probability measures $M_n$ on partitions $\lambda=(\lambda_1\ge\lambda_2\ge\dots\ge\lambda_n\ge 0)$ of length (number of non-zero parts) at most $n$ which satisfy the coherency relations
$$
\sum_{\lambda}M_{n+1}(\lambda)\frac{G_{\lambda/\mu}^{(0,-p)}(1)G_\mu^{(0,-p)}(1^n)}{G_{\lambda}^{(0,-p)}(1^{n+1})}=M_n(\mu),
$$
for each $n\geq0$ and each partition $\mu$ of length at most $n$. Coherent systems are closely related to limits of normalized Grothendieck polynomials: if for a sequence of partitions $\lambda(N)$ we define $M_k(\lambda)$ by the expansion
$$
 \lim_{N\to\infty} \frac{G^{(0,-p)}_{\lambda(N)}(x_1,\dots,x_k,1^{N-k})}{G^{(0,-p)}_{\lambda(N)}(1^N)}=\sum_{\lambda}M_{k}(\lambda)\frac{G_\lambda^{(0,-p)}(x_1,\dots, x_k)}{G_{\lambda}^{(0,-p)}(1^{k})},
$$
then, under suitable convergence conditions, the resulting measures $\{M_k\}_k$ form a coherent system. Coherency relation in this case follows from the branching rule for Grothendieck polynomials. Further, Remark \ref{5-vertexrem} explains that coherent systems are in correspondence with Gibbs measures for the five-vertex model in the quadrant.

Our first result constructs a family of coherent systems $\{M^{\Phi}_n\}_{n\geq 0}$ parametrized by functions $\Phi\in\mathcal F$. The space $\mathcal F$ consists of functions analytic on the unit disk $|z|\leq 1$, which can be obtained as the limit $\lim_k \Phi_k(z)$ of rational functions
$$
\Phi_k(z)=\prod_{i=1}^{n_k}\frac{1-x_{i,k}(z-1)}{1-y_{i,k}(z-1)},
$$
where $x_{i,k}\in [-1, \frac{p}{1-p}]$, $y_{i,k}\geq 0$, $y_{i,k}\geq x_{i,k}$ for every $i,k$ and the limit is uniform on $|z|\leq 1$. Given a function $\Phi\in\mathcal F$, we define the coherent system $\{M_n^\Phi\}_n$ using the decomposition
$$
\sum_{\lambda: l(\lambda)\leq n} M_n^\Phi(\lambda) \frac{G_{\lambda}^{(0,-p)}(z_1, \dots, z_n)}{G_{\lambda}^{(0,-p)}(1^n)}=\Phi(z_1)\Phi(z_2)\dots\Phi(z_n).
$$
\begin{theoAlph}[{Theorem \ref{coherent-thm} in the text}] \label{Theorem_A} $\{M_n^\Phi\}_n$ is a coherent system for any $\Phi\in\mathcal F$.
\end{theoAlph}
\noindent The difficult part of this result is showing that $M_n^\Phi(\lambda)\geq 0$ for all $n,\lambda$. We do it by using the Cauchy--Littlewood identities for Grothendieck polynomials and combinatorics of dual Grothendieck polynomials. We further explain in Section \ref{Section_independence} that the coherent systems $\{M_n^\Phi\}_n$  are independent: it is impossible to express one of them as a convex linear combination of others. The proof of the independence result is based on a reduction to de Finetti's theorem.

\medskip

Our other results concern the coherent systems $\{M_n^\Phi\}_n$ for the special choice
$$
\Phi(z)=\Phi^{\calA,\calB}(z)=\prod_{i=1}^k\frac{1-\frac{p}{1-p}(z-1)}{1-\alpha_i (z-1)}\prod_{i=1}^l\left(1+\frac{\beta_i-p}{1-p}(z-1)\right)
$$
for a pair of sequences $\mathcal A=(\alpha_1, \dots, \alpha_k), \mathcal B=(\beta_1, \dots, \beta_l)$ satisfying
$$
\alpha_1\geq \alpha_2\geq\dots\geq \alpha_k>\frac{p}{1-p},\qquad 1\geq \beta_1\geq\beta_2\geq\dots \geq \beta_l\geq p.
$$
In this case we use $\{M_n^{\calA,\calB}\}_n$ to denote the resulting coherent system. These systems are distinguished by the following property: the measures $M_n^{\calA,\calB}$ are supported on the partitions whose Young diagrams do not include the box $(k+1,l+1)$. In other words, when working with $M_n^{\calA,\calB}$ we only need to consider partitions contained in the infinite hook-shape with $k$ infinite rows and $l$ infinite columns. We study these systems in more detail and obtain a form of the Law of Large Numbers and Central Limit Theorem for them.

\begin{theoAlph}[{Theorem \ref{limittheo} in the text}] \label{Theorem_B} Let $\lambda(n)$ denote the random partition distributed according to $M^{\calA,\calB}_n$, let $\lambda'(n)$ denote the transpose partition, and $s$ denote the number of $i$ such that $\beta_i=1$. Then $\lambda(n)'_1=\dots=\lambda(n)'_s=n$ almost surely and as $n\to\infty$
$$
\left(\frac{\lambda(n)_1-\alpha_1n}{\sqrt{n\alpha_1(1+\alpha_1)}},\dots, \frac{\lambda(n)_k-\alpha_kn}{\sqrt{n\alpha_k(1+\alpha_k)}} \right) \to \bx^{GUE}_{\calA},
$$
$$
\left(\frac{\lambda(n)'_{s+1}-\beta_{s+1}n}{\sqrt{n\beta_{s+1}(1-\beta_{s+1})}},\dots, \frac{\lambda(n)'_l-\beta_ln}{\sqrt{n\beta_l(1-\beta_l)}} \right) \to \by^{GUE}_{\calB},
$$
where both convergences are in distribution and $\bx^{GUE}_{\calA}, \by^{GUE}_{\calB}$ denote random vectors defined in Section \ref{finitesect} in terms of GUE eigenvalue distributions. In particular, in probability, $\frac{\lambda(n)_i}{n}\to \alpha_i$ for $i\in [1,k]$ and $\frac{\lambda(n)'_i}{n}\to \beta_i$ for $i\in [1,l]$.
\end{theoAlph}
\noindent The proof of Theorem \ref{Theorem_B} is based on precise asymptotic analysis of Grothendieck polynomials $G^{(0,-p)}_{\lambda}$ when either the number of rows or the number of columns of $\lambda$ is fixed. This analysis applies the steepest descent method to two contour integral representations for Grothendieck polynomials developed in Section \ref{Section_contour_integral}.

Combining Theorem \ref{Theorem_B} with very general properties of coherent systems, we arrive at the following strengthening of the independence of $\{M^{\Phi}_n\}_{n\geq 0}$.

\begin{theoAlph}[{Theorem \ref{Theorem_extreme} in the text}]\label{Theorem_C} Let $\calA=(\alpha_1,\dots, \alpha_k)$ and $\calB=(\beta_1,\dots,\beta_l)$ be sequences satisfying
$$
\alpha_1\geq \alpha_2\geq\dots\geq \alpha_k>\frac{p}{1-p},\qquad 1\geq \beta_1\geq\beta_2\geq\dots \geq \beta_l\geq p.
$$
Then the coherent systems $M_n^{\calA,\varnothing}$ and $M_n^{\varnothing,\calB}$ are extreme points in the convex space of all coherent systems.
\end{theoAlph}

Two intriguing follow-up questions remain open: Are all other $\{M_n^\Phi\}_n$ in Theorem \ref{Theorem_A} also extreme? Are there extreme coherent systems not from this list?

\subsection{Structure of the text} In Section \ref{Gsect} we give necessary background about stable Grothendieck polynomials. In Section \ref{Section_Branching_graph} we introduce the branching graph with Grothendieck weights and describe a family of coherent systems on it, proving Theorem \ref{Theorem_A}. In Section \ref{asyman} we provide asymptotic analysis of Grothendieck polynomials which is used in the following section.  In Section \ref{finitesect} we explore the asymptotic behavior for some of the constructed coherent systems (proving Theorem \ref{Theorem_B}) and show that they are extreme (proving Theorem \ref{Theorem_C}). Finally, in Section \ref{Section_open_questions} we discuss open questions and conjectures regarding potential further developments of the subject.

\subsection{Acknowledgements} We thank Alexei Borodin, Alexey Bufetov, and Grigori Olshanski for very helpful discussions.

\section{Grothendieck functions}\label{Gsect}

In this section we describe the necessary facts about (stable) Grothendieck symmetric functions. We are using the version of these functions from \cite{Yel16}, which is our primary reference, together with \cite{HJKSS21}. For this section we let $\fa, \fb$ denote a pair of generic parameters.

\subsection{Basic notation} First we remind the standard notation regarding partitions and symmetric functions, following \cite[Chapter I]{Mac95}. A partition is an integer sequence $\lambda=(\lambda_1,\lambda_2,\dots)$ such that $\lambda_1\geq\lambda_2\geq\dots\geq 0$. We use $\mathbb Y$ to denote the set of all partitions. For a partition $\lambda$ its length $l(\lambda)$ is the number of nonzero parts $\lambda_i$, $m_k(\lambda)$ is the number of $\lambda_i$ equal to $k$ and we also sometimes use $1^{m_1(\lambda)}2^{m_2(\lambda)}\dots$ to denote $\lambda$. We also define the conjugate partition $\lambda'$ where $\lambda'_i$ is the number of $j$ such that $\lambda_j\geq i$. For a pair of partitions $\lambda,\mu$ we write $\mu\subset\lambda$ if $\mu_i\leq\lambda_i$ for all $i$ and we use $\lambda/\mu$ to denote the corresponding skew diagram, i.e.\ a collection of boxes arranged in rows with row $i$ formed by boxes at $(i,\mu_i+1)$, $(i,\mu_i+2)$, \dots, $(i,\lambda_i)$. We set $|\lambda|=\lambda_1+\lambda_2+\dots$ and $|\lambda/\mu|=|\lambda|-|\mu|$. We say that $\mu$ interlaces $\lambda$ and write $\mu\preceq\lambda$ when
$$
\lambda_1\geq\mu_1\geq\lambda_2\geq\mu_2\geq\dots.
$$
This condition is equivalent to the skew diagram $\lambda/\mu$ being a horizontal strip, i.e.\ it has at most one box in each column.

We use $\Lambda_n$ to denote the graded ring of symmetric polynomials over $\Z[\fa,\fb]$ in $n$ variables $x_1, \dots, x_n$, and let $\Lambda$ denote the ring of symmetric functions over $\Z[\fa,\fb]$, which we treat as functions in infinitely many variables $(x_1,x_2,\dots)$. We use $\hat\Lambda_n, \hat\Lambda$ to denote the completions of the corresponding graded rings, which are equivalent to symmetric formal series in $(x_1,\dots, x_n)$ and $(x_1, x_2, \dots)$ respectively. Recall that $\Lambda$ has several graded bases, which are all labeled by the set of partitions. The \emph{complete symmetric functions} are defined by
$$
h_k(x_1,x_2,\dots)=\sum_{i_1\leq i_2\leq\dots\leq i_k} x_{i_1}\dots x_{i_k},\qquad h_\lambda=h_{\lambda_1}h_{\lambda_2}\dots.
$$
The functions $\{h_k\}_{k\geq 1}$ are algebraically independent and generate the algebra $\Lambda$, that is, $\{h_\lambda\}_\lambda$ is a graded basis of $\Lambda$. For finitely many variables, \emph{Schur polynomials} are given by
$$
s_\lambda(x_1,\dots, x_n)=\sum_{\varnothing=\lambda^{(0)}\preceq\lambda^{(1)}\preceq\dots\preceq \lambda^{(n)}=\lambda}x_1^{|\lambda^{(1)}/\lambda^{(0)}|}\dots x_n^{|\lambda^{(n)}/\lambda^{(n-1)}|}=\frac{\det[x_j^{\lambda_i+n-i}]_{1\leq i,j\leq n}}{\det[x_j^{n-i}]_{1\leq i,j\leq n}}.
$$
One can verify that $s_\lambda(x_1,\dots, x_n,0)=s_\lambda(x_1,\dots, x_n)$, which allows to define Schur symmetric functions $s_\lambda(x_1,x_2,\dots)\in\Lambda$. The functions $\{s_\lambda\}_\lambda$ form a graded basis of $\Lambda$.

We can define the Hall scalar product $\langle\cdot,\cdot\rangle$ on $\Lambda$ by setting $\{s_\lambda\}_\lambda$ to be an orthonormal basis. Equivalently, two graded bases $\{f_\lambda\}_\lambda$, $\{g_\lambda\}_\lambda$ are dual to each other with respect to the Hall product if and only if
$$
\sum_{\lambda} f_\lambda(x_1,x_2,\dots) g_\lambda(y_1,y_2,\dots)=\sum_{\lambda} s_\lambda(x_1,x_2,\dots)s_\lambda(y_1,y_2,\dots)=\prod_{i,j\geq 1}\frac{1}{1-x_iy_j}.
$$
The involution $\omega$ on $\hat\Lambda$ is defined by setting $\omega(h_k)=e_k$. Since $\{h_k\}_{k\geq 1}$ generate $\Lambda$, the action of $\omega$ can be extended to any symmetric function. In particular, $\omega(s_\lambda)=s_{\lambda'}$.

\subsection{Grothendieck functions} For a partition $\lambda$ of length at most $n$ the stable Grothendieck polynomial $G_\lambda^{(\fa,\fb)}(x_1, \dots, x_n)$ is defined by
\begin{equation}\label{Gdef}
G_\lambda^{(\fa,\fb)}(x_1, \dots, x_n)=\frac{\displaystyle\det\left[\frac{x_i^{\lambda_j+n-j}(1+\fb x_i)^{j-1}}{(1-\fa x_i)^{\lambda_j}}\right]_{1\leq i,j\leq n}}{\det\left[x_i^{n-j}\right]_{1\leq i,j\leq n}}.
\end{equation}
As a ratio of two skew-symmetric expressions, polynomials $G_\lambda^{(\fa,\fb)}$ are symmetric and we can treat them as elements of $\hat\Lambda_n$. Since
$$
\restr{\det\left[\frac{x_i^{\lambda_j+n+1-j}(1+\fb x_i)^{j-1}}{(1-\fa x_i)^{\lambda_j}}\right]_{1\leq i,j\leq n+1}}{x_{n+1}=0}=x_1x_2\dots x_n\det\left[\frac{x_i^{\lambda_j+n-j}(1+\fb x_i)^{j-1}}{(1-\fa x_i)^{\lambda_j}}\right]_{1\leq i,j\leq n}
$$
we have the stability property
$$
G_\lambda^{(\fa,\fb)}(x_1, \dots, x_n, 0)=G_\lambda^{(\fa,\fb)}(x_1, \dots, x_n),
$$
where for $l(\lambda)>n$ we set $G_\lambda^{(\fa,\fb)}(x_1, \dots, x_n)=0$. This allows us to define $G_\lambda^{(\fa,\fb)}\in\hat\Lambda$ as a symmetric power series in infinitely many variables $x_1, x_2, \dots$. Note that
\begin{equation}\label{expandG}
G_\lambda^{(\fa,\fb)}=s_\lambda+\text{higher degree terms},
\end{equation}
in particular the family $\{G_\lambda^{(\fa,\fb)}\}_{\lambda\in\mathbb Y}$ is linearly independent.

Define the dual Grothendieck functions $\{g_\lambda^{(\fa,\fb)}\}_\lambda$ as the dual family to $\{G_\lambda^{(-\fa,-\fb)}\}_\lambda$ with respect to Hall scalar product. In other words, we set
\begin{equation*}
g_\lambda =\sum_{\mu} a_{\mu\lambda} s_\mu
\end{equation*}
where $(a_{\lambda\mu})_{\lambda,\mu\in\mathbb Y}$ is the inverse of the transition matrix $(\langle G_\lambda^{\fa,\fb}, s_\mu \rangle)_{\lambda,\mu\in \mathbb Y}$. Due to \eqref{expandG} the transition matrix is upper-triangular with respect to an ordering where partitions of $m$ are smaller than partitions of $n$ for $m<n$. So the inverse $(a_{\lambda\mu})_{\lambda,\mu\in\mathbb Y}$ exists and is upper-triangular, implying
$$
g_\lambda^{(\fa,\fb)}=s_\lambda+\text{lower degree terms}.
$$
In particular, $\{g_\lambda^{(\fa,\fb)}\}_\lambda$ is a basis of $\Lambda$. We have the Cauchy identity
$$
\sum_\lambda G_\lambda^{(-\fa,-\fb)}(x_1, x_2,\dots)g_\lambda^{(\fa,\fb)}(y_1, y_2, \dots)=\prod_{i,j\geq 1}\frac{1}{1-x_iy_j}.
$$
We can also define $g_\lambda^{(\fa,\fb)}$ in a fashion similar to \eqref{Gdef}, see \cite[Definition 1.2]{HJKSS21}. We only give the $\fb=0$ case of that definition.

\begin{prop}[{\cite{HJKSS21}}]\label{detg} Let $\lambda$ be a partition of length at most $n$. Then
$$
g^{(\fa,0)}_\lambda(x_1, \dots, x_n)=\frac{\displaystyle\det\left[x_i^{n-j}\phi^{(\fa,0)}_{\lambda_j}(x_i)\right]_{1\leq i,j\leq n}}{\det\left[x_i^{n-j}\right]_{1\leq i,j\leq n}},
$$
where $\phi^{(\fa,0)}_0(x)=1$ and for $k>0$
$$
\phi^{(\fa,0)}_k(x)=x(x+\fa)^{k-1}.
$$
\end{prop}

\subsection{Branching, Jacobi-Trudi and skew Cauchy identities} Both $G_\lambda^{(\fa,\fb)}$ and $g_\lambda^{(\fa,\fb)}$ have branching rules, but to describe them we need additional notation. For a skew partition $\lambda/\mu$ let $r(\lambda/\mu)$ denote the number of non-empty rows of the diagram of $\lambda/\mu$, $c(\lambda/\mu)$ denote the number of non-empty columns, $b(\lambda/\mu)$ be the number of connected components of the diagram and $i(\lambda/\mu):=|\lambda/\mu|-r(\lambda/\mu)-c(\lambda/\mu)+b(\lambda/\mu)$. See Figure \ref{Figskewg} for an example. Additionally, for $\lambda=(\lambda_1, \lambda_2, \dots)$ we set $\overline{\lambda}=(\lambda_2, \lambda_3, \dots)$.

\begin{figure}
\ydiagram{5+2, 2+2, 3, 2}
\caption{\label{Figskewg}} The skew diagram $\lambda/\mu$ for $\lambda=(4,3,2)$ and $\mu=(2)$. In this case we have $r(\lambda/\mu)=4$, $c(\lambda/\mu)=6$, $b(\lambda/\mu)=2$, $i(\lambda/\mu)=1$ and $g^{(\fa,\fb)}_{\lambda/\mu}(x)=\fb^2(\fa+\fb)x^2(x+\fa)^4$.
\end{figure}

\begin{prop}[{\cite[Proposition 8.8]{Yel16}}]\label{Gbranch} We have
$$
G_\lambda^{(\fa,\fb)}(x_1,\dots, x_{n+1})=\sum_{\mu\preceq\lambda}G^{(\fa,\fb)}_{\lambda/\mu}(x_{n+1})G^{(\fa,\fb)}_{\mu}(x_1, \dots, x_n),
$$
where we set
\begin{equation}
\label{Gone}
G^{(\fa,\fb)}_{\lambda/\mu}(x)=\begin{cases}
\left(\frac{x}{1-\fa x}\right)^{|\lambda/\mu|}\left(\frac{1+\fb x}{1-\fa x}\right)^{r(\mu/\overline{\lambda})}\qquad &\text{if}\ \mu\preceq\lambda,\\
0\qquad &\text{otherwise}.
\end{cases}
\end{equation}
\end{prop}

\begin{prop}[{\cite[Theorem 8.6]{Yel16}}] \label{g-branching-prop} We have
$$
g_\lambda^{(\fa,\fb)}(x_1,\dots, x_{n+1})=\sum_{\mu}g^{(\fa,\fb)}_{\lambda/\mu}(x_{n+1})g^{(\fa,\fb)}_{\mu}(x_1, \dots, x_n),
$$
where we set
\begin{equation}
\label{gone}
 g^{(\fa,\fb)}_{\lambda/\mu}(x)=\begin{cases}\fb^{r(\lambda/\mu)-b(\lambda/\mu)}(\fa+\fb)^{i(\lambda/\mu)}x^{b(\lambda/\mu)}(x+\fa)^{c(\lambda/\mu)-b(\lambda/\mu)}\qquad& \text{if}\ \mu\subseteq\lambda,\\ 0\qquad & \text{otherwise.}\end{cases}
\end{equation}
\end{prop}

More generally, we can use branching rules above to define skew functions $G^{(\fa,\fb)}_{\lambda/\mu}, g^{(\fa,\fb)}_{\lambda/\mu}$ for arbitrary number of variables:

\begin{prop}\label{branching-rules} There exist symmetric functions $G^{(\fa,\fb)}_{\lambda/\mu}\in\hat\Lambda$, $g^{(\fa,\fb)}_{\lambda/\mu}\in\Lambda$ satisfying the following properties:
\begin{itemize}
\item One variable specializations $G^{(\fa,\fb)}_{\lambda/\mu}(x), g^{(\fa,\fb)}_{\lambda/\mu}(x)$ coincide with \eqref{Gone}, \eqref{gone};

\item $G^{(\fa,\fb)}_{\lambda/\varnothing}=G^{(\fa,\fb)}_{\lambda}$, $g^{(\fa,\fb)}_{\lambda/\varnothing}=g^{(\fa,\fb)}_{\lambda}$;

\item The following branching rules hold for any sets of variables $\mathbf{x}, \mathbf{y}$
$$
G_{\lambda/\mu}^{(\fa,\fb)}(\mathbf{x}, \mathbf{y})=\sum_{\nu}G^{(\fa,\fb)}_{\lambda/\nu}(\mathbf{x})G^{(\fa,\fb)}_{\nu/\mu}(\mathbf{y}),
$$
$$
g_{\lambda/\mu}^{(\fa,\fb)}(\mathbf{x}, \mathbf{y})=\sum_{\nu}g^{(\fa,\fb)}_{\lambda/\nu}(\mathbf{x})g^{(\fa,\fb)}_{\nu/\mu}(\mathbf{y}),
$$
\end{itemize}
\end{prop}
\begin{proof} This is a standard argument which we give here for $G^{(\fa,\fb)}_{\lambda/\mu}$, the argument for $g_{\lambda/\mu}^{(\fa,\fb)}$ is identical. Let $\mathbf{x}=(x_1, \dots, x_n)$. Define $G_{\lambda/\mu}^{(\fa,\fb)}(\mathbf{x})$ using one-variable functions and the branching rule repeated for $n-1$ times. To show that the resulting functions are symmetric, note that by Proposition \ref{Gbranch} $G_{\lambda/\mu}^{(\fa,\fb)}(\mathbf{x})$ is the coefficient of $G_{\mu}^{(\fa,\fb)}(\mathbf{y})$ in $G_{\lambda}^{(\fa,\fb)}(\mathbf{x}, \mathbf{y})$ and the latter is symmetric. Finally, the stability follows from noticing that $G^{(\fa,\fb)}_{\lambda/\mu}(0)=0$ unless $\lambda=\mu$, so
$$
G_{\lambda/\mu}^{(\fa,\fb)}(\bx, 0)=\sum_{\nu}G^{(\fa,\fb)}_{\lambda/\nu}(\mathbf{x})G^{(\fa,\fb)}_{\nu/\mu}(0)=G^{(\fa,\fb)}_{\lambda/\mu}(\mathbf{x})
$$
and we can define $G^{(\fa,\fb)}_{\lambda/\mu}$ as an element of $\hat\Lambda$.
\end{proof}

We also have a skew version of the Cauchy identity.
\begin{prop}\label{skewCauchy-prop} For any fixed partition $\mu$ we have
$$
\sum_{\lambda} G^{(-\fa,-\fb)}_{\lambda/\mu}(x_1, x_2, \dots)g^{(\fa,\fb)}_{\lambda/\nu}(y_1, y_2, \dots)=\prod_{i,j\geq 1}\frac{1}{1-x_iy_j}\sum_{\lambda}G^{(-\fa,-\fb)}_{\nu/\lambda}(x_1, x_2, \dots)g^{(\fa,\fb)}_{\mu/\lambda}(y_1, y_2, \dots).
$$
\end{prop}
\begin{proof}
First consider the case $\nu=\varnothing$. From the branching rule and the ordinary Cauchy identity we have
\begin{multline*}
\sum_{\lambda,\mu} G^{(-\fa,-\fb)}_{\lambda/\mu}(\mathbf{x})G^{(-\fa,-\fb)}_{\mu}(\mathbf{z})g^{(\fa,\fb)}_{\lambda}(\mathbf{y})=\prod_{i,j\geq 1}\frac{1}{1-x_iy_j}\prod_{i,j\geq 1}\frac{1}{1-z_iy_j}\\
=\sum_\mu g^{(\fa,\fb)}_{\mu}(\mathbf{y})G^{(-\fa,-\fb)}_{\mu}(\mathbf{z})\prod_{i,j\geq 1}\frac{1}{1-x_iy_j}.
\end{multline*}
Applying the scalar product $\langle\cdot, g_\mu^{(\fa,\fb)}\rangle$ to both sides with respect to the variables $\bz$ implies the claim.

For the general case we have
\begin{multline*}
\sum_{\lambda,\nu} G^{(-\fa,-\fb)}_{\lambda/\mu}(\mathbf{x})g^{(\fa,\fb)}_{\lambda/\nu}(\mathbf{y})g^{(\fa,\fb)}_{\nu}(\mathbf{z})=\prod_{i,j\geq 1}\frac{1}{1-x_iy_j}\prod_{i,j\geq 1}\frac{1}{1-x_iz_j}g^{(\fa,\fb)}_{\mu}(\mathbf{y},\mathbf{z})\\
=\prod_{i,j\geq 1}\frac{1}{1-x_iy_j}\sum_{\lambda}\prod_{i,j\geq 1}\frac{1}{1-x_iz_j}g^{(\fa,\fb)}_{\mu/\lambda}(\mathbf{y})g^{(\fa,\fb)}_{\lambda}(\mathbf{z})\\
=\prod_{i,j\geq 1}\frac{1}{1-x_iy_j}\sum_{\lambda,\nu}g^{(\fa,\fb)}_{\mu/\lambda}(\mathbf{y})G^{(-\fa,-\fb)}_{\nu/\lambda}(\bx)g^{(\fa,\fb)}_{\nu}(\mathbf{z}).
\end{multline*}
Taking the coefficient of $g^{(\fa,\fb)}_{\nu}(\mathbf{z})$ on both sides above we get the general skew-Cauchy identity. Note that we can take this coefficient even in the infinite sums above: for each $d\geq 0$ the component of the homogeneous degree $d$ in $\bx$ is computed only finitely many terms.
\end{proof}

For both $G^{(-\fa,-\fb)}_{\lambda/\mu}$ and $g^{(-\fa,-\fb)}_{\lambda/\mu}$ there are various analogues of Jacobi-Trudi identities, see \cite{HJKSS21}. In this work we need only two particular expressions described below. For both statements we use $f(\fa^n)$ to denote the evaluation of $f$ at $n$-copies of $\fa$.

\begin{prop}[{\cite[Theorem 5.2]{HJKSS21}}]\label{GJT} For $\lambda$ with $l(\lambda)\leq n$ we have
$$
G_\lambda^{(\fa,-\fb)}(\mathbf{x})=\prod_{i\geq 1}(1-\fb x_i)^n\det\left[f_{\lambda_i-i+j}^{(\lambda_i, j-i+1)}(\mathbf x)\right]_{1\leq i,j\leq n},
$$
where
$$
f_k^{(m, l)}(\mathbf x)=\begin{cases}\displaystyle\sum_{\substack{a,b\\a-b-c=k}}h_{a}(\mathbf{x})h_b(\fb^l)h_c(\fa^m)\qquad & \text{if}\ l\geq 0,\\\displaystyle\sum_{\substack{a,b\\a-b-c=k}}h_{a}(\mathbf{x})e_b((-\fb)^{-l})h_c(\fa^m)\qquad &\text{if}\ l\leq 0. \end{cases}
$$
\end{prop}
\begin{prop}[{\cite[Theorem 1.3]{HJKSS21}}]\label{JT-g} For $\lambda$ with $l(\lambda)\leq n$ we have
$$
g_\lambda^{(0,\fb)}(\mathbf x)=\det\big[h_{\lambda_i-i+j}[\mathbf x+(i-1)\fb]\big]_{1\leq i,j\leq n},
$$
where we use the plethystic notation
$$
h_n[\mathbf x+k\fb]:=\sum_{a+b=n}h_a(\mathbf x)h_b(\fb^k).
$$
\end{prop}

\subsection{Involution and parameter shifting automorphisms} There are several automorphisms of $\Lambda$, $\hat\Lambda$ whose action on Grothendieck functions can be easily described. We start with the involution $\omega$.

\begin{theo} \label{involution}We have
$$
\omega\left(G_\lambda^{(\fa,\fb)}\right)=G_{\lambda'}^{(\fb,\fa)}\qquad \omega\left(g_\lambda^{(\fa,\fb)}\right)=g_{\lambda'}^{(\fb,\fa)}
$$
\end{theo}
\begin{proof} For $G^{(\fa,\fb)}_\lambda$ this is \cite[Theorem 5.4]{Yel16}. The claim for the dual functions $g_\lambda$ follows since $g^{(\fa,\fb)}_\lambda$ are dual to $G^{(-\fa,-\fb)}_\lambda$ and $\omega$ preserves the scalar product on $\Lambda$.
\end{proof}

Other automorphisms we consider allow us to change parameters $(\fa,\fb)$ in both $G_\lambda$ and $g_\lambda$. First, one can note directly from \eqref{Gdef} that
$$
G_\lambda^{(\fa,\fb)}\left(\frac{x_1}{1-\gamma x_1},\frac{x_2}{1-\gamma x_2},\frac{x_3}{1-\gamma x_3},\dots\right)=G_\lambda^{(\fa+\gamma,\fb-\gamma)}(x_1, x_2, x_3,\dots),
$$
where $\gamma$ is an additional parameter. So we define $\rho_\gamma:\hat\Lambda\to\hat\Lambda$ by
$$
\rho_\gamma f(x_1, x_2, \dots)=f\left(\frac{x_1}{1-\gamma x_1}, \frac{x_2}{1-\gamma x_2}, \dots\right).
$$
Note that $\rho_\gamma$ is a well-defined automorphism of $\hat\Lambda$, since the $n$th degree component of $\rho_\gamma f$ is computed by a finite computation using lower degree terms. Also one can immediately see that
$$
\frac{\frac{x}{1-\gamma_2 x}}{1-\gamma_1\frac{x}{1-\gamma_2 x}}=\frac{x}{1-\gamma_2x-\gamma_1x}=\frac{x}{1-(\gamma_1+\gamma_2)x}
$$
so $\rho_{\gamma_1} \rho_{\gamma_2}=\rho_{\gamma_1+\gamma_2}$.

To do a similar manipulation for the dual functions $g_\lambda$, consider the adjoint morphisms $\rho_\gamma^*$, which are uniquely defined by requiring
$$
\langle f, \rho_\gamma g\rangle=\langle \rho^*_\gamma f,  g\rangle
$$
for any $f,g\in\Lambda$. To give more explicit descriptions, recall the generating function for complete symmetric functions:
$$
H(z)=\sum_{k\geq 0}h_kz^k=\prod_{i\geq 1}\frac{1}{1-zx_i}.
$$
\begin{prop}\label{adjshiftproperties} (1) For any partition $\lambda$ we have $\rho_\gamma^*(g_\lambda^{(\fa,\fb)})=g_\lambda^{(\fa+\gamma,\fb-\gamma)}$. In particular $\rho^*_\gamma$ defines a map $\Lambda\to\Lambda$.

(2) The action of $\rho_\gamma^*$ on complete symmetric functions is described by
$$
\rho_\gamma^*(H(z))=H\left(\frac{z}{1-\gamma z}\right)=\prod_{i}\frac{1-\gamma z}{1-(x_i+\gamma)z}.
$$
(3) We have $\omega \rho_\gamma^*=\rho_{-\gamma}^*\omega$.
\end{prop}
\begin{proof}
(1) For partitions $\lambda,\mu$ we have
$$
\left\langle \rho_\gamma^*\left(g_\lambda^{(\fa,\fb)}\right),G_\mu^{(-\fa-\gamma,-\fb+\gamma)}\right\rangle=\left\langle g_\lambda^{(\fa,\fb)},\rho_\gamma \left(G^{(-\fa-\gamma,-\fb+\gamma)}_\mu\right)\right\rangle=\left\langle g_\lambda^{(\fa,\fb)},G^{(-\fa,-\fb)}_\mu\right\rangle=\delta_{\lambda,\mu}.
$$
This implies $\rho_\gamma^*\left(g_\lambda^{(\fa,\fb)}\right)= g_\lambda^{(\fa+\gamma,\fb-\gamma)}$.

(2) Consider infinite sets of variables $\bx=(x_1, x_2,\dots)$ and $\by=(y_1,y_2,\dots)$. Then
$$
(\rho_\gamma^*)_{\mathbf x}\prod_{i,j}\frac{1}{1-x_iy_j}=\sum_{\lambda,\mu} \langle \rho_\gamma^*s_\lambda \mid s_\mu\rangle s_\mu(\bx)s_\lambda(\by)=\sum_{\lambda,\mu} \langle s_\lambda \mid \rho_\gamma s_\mu\rangle s_\mu(\bx)s_\lambda(\by)=(\rho_\gamma)_{\mathbf y}\prod_{i,j}\frac{1}{1-x_iy_j}
$$
where we use $A_\mathbf{x}, A_\mathbf{y}$ to denote actions of an operator $A$ on the corresponding sets of variables. Hence
$$
(\rho_\gamma^*)_{\mathbf x}\prod_{i,j}\frac{1}{1-x_iy_j}=\prod_{i,j}\frac{1-\gamma y_j}{1-(x_i+\gamma)y_j}
$$
and setting $y_1=z$ and $y_2=y_3=\dots=0$ gives the expression for $\rho_\gamma^*(H(z))$.

(3) Since $g_\lambda^{(\fa,\fb)}$ form a basis of $\Lambda$, it is enough to verify that
$$
\omega \rho_\gamma^*g_\lambda^{(\fa,\fb)}=\rho_{-\gamma}^*\omega g_\lambda^{(\fa,\fb)}
$$
which follows from part (1) and Theorem \ref{involution}.
\end{proof}

\subsection{TASEP with geometric jumps and five-vertex model.}\label{TASEPsec} For $p\in (0,1)$ the functions $G_\lambda^{(0,-p)}(1^n)$ are closely related to two models, which we describe in this subsection.

The first model is a particle system called \emph{TASEP with geometric jumps}. The state space of this system consists of infinite particle configurations on $\mathbb Z$, where particles are located at sites $Y_1> Y_2> Y_3>\dots$ and satisfy $Y_i=-i$ for sufficiently large $i$. In other words, we consider a system with infinitely many particles on the lattice $\Z$, with each site $i\in\Z$ occupied by at most one particle and for sufficiently large $n$ there are $n$ particles weakly to the right of $-n$. 
Note that these states can be identified with partitions by setting $Y_i=\lambda_i-i$, which we equivalently write as $Y=\lambda+\delta$ by setting $\delta=(-1, -2, -3, \dots)$.

TASEP with geometric jumps is a discrete time Markov process $Y(t)$ on the particle configurations, where we use $Y(t)=(Y_1(t),Y_2(t),\dots)$ to describe the configuration at time $t\in\mathbb Z_{\geq 0}$. Evolution of the process is defined as follows:
\begin{itemize}
\item We start with $Y_i(0)=-i$ for all $i$.
\item During the step $Y(t)\to Y(t+1)$ each particle $Y_i$ jumps a random distance forward without overcoming the particle $Y_{i-1}$. All these jumps are independent and applied from left to right, so that $Y_1$ is updated last.

\item For $i> 1$ the jump distance $Y_i(t+1)-Y_i(t)$ has the following distribution:
$$
\P(Y_i(t+1)-Y_i(t)=d)=\begin{cases}(1-p)p^d\qquad & d+Y_i(t)\in \{Y_i(t), Y_i(t)+1,\dots, Y_{i-1}(t)-2\},
\\ p^d \qquad & d+Y_i(t)=Y_{i-1}(t)-1.
\end{cases}
$$
In other words, $Y_i$ tries to jump a geometrically distributed distance forward, but it is stopped by $Y_{i-1}$ which is not yet updated.

\item  The jump distance $Y_1(t+1)-Y_1(t)$ has geometric distribution
$$
\P(Y_1(t+1)-Y_1(t)=d)=(1-p)p^d.
$$
\end{itemize}

See the top part of Figure \ref{modelexample} for an example of one step of this process. It turns out that the distribution of this process can be described using Grothendieck polynomials.

\begin{prop}\label{GTASEP} Let $\lambda,\mu$ be partitions such that $l(\mu)\leq t$, and $n\in\Z_{\geq 0}$. Then
$$
\P(Y(t+n)=\lambda+\delta\mid Y(t)=\mu+\delta)=p^{|\lambda|-|\mu|}(1-p)^nG^{(0,-p)}_{\lambda/\mu}(1^n).
$$
In particular, $\P(Y(n)=\lambda+\delta)=p^{|\lambda|}(1-p)^nG^{(0,-p)}_{\lambda}(1^n)$.
\end{prop}
\begin{proof} First note that $\P(Y(m)=\mu+\delta)>0$ when $l(\mu)\leq m$. Indeed, we can achieve $Y(m)=\mu+\delta$ with positive probability by only moving the particle $Y_i$ to $\mu_i-i$ at step $i$ and keeping it stationary otherwise. So the conditional probability is well-defined.

Now consider the case $n=1$. Fix partitions $\lambda,\mu$. Since the particles $Y_i$ cannot overtake each other we have
$$
\P(Y(t+1)=\lambda+\delta\mid Y(t)=\mu+\delta)=0 \qquad \text{unless\ } \mu\prec \lambda.
$$
The same holds for $G_{\lambda/\mu}^{(0,-p)}$, so from now on assume $\mu\prec \lambda$. From Proposition \ref{Gbranch}
$$
p^{|\lambda|-|\mu|}(1-p)G^{(0,-p)}_{\lambda/\mu}(1)=p^{|\lambda|-|\mu|}(1-p)^{r(\mu\backslash\overline\lambda)+1}=p^{\lambda_1-\mu_1}(1-p)\prod_{i\geq 2} p^{\lambda_i-\mu_i}(1-p)^{\1_{\lambda_i<\mu_{i-1}}},
$$
where $\1_{\lambda_i<\mu_{i-1}}$ is the indicator of the condition ${\lambda_i<\mu_{i-1}}$.  At the same time, if $\lambda+\delta=Y(t+1)$ and $\mu+\delta=Y(t)$ then $\lambda_i-\mu_i$ is exactly the distance jumped by the particle $Y_i$. Also, for $i\geq 2$ the condition $Y_{i}(t+1)<Y_{i-1}(t)-1$ is equivalent to $\lambda_i<\mu_{i-1}$, so
$$
\P(Y(t+1)=\lambda+\delta\mid Y(t)=\mu+\delta)=p^{\lambda_1-\mu_1}(1-p)\prod_{i\geq 2} p^{\lambda_i-\mu_i}(1-p)^{\1_{\lambda_i<\mu_{i-1}}}=p^{|\lambda|-|\mu|}(1-p)G^{(0,-p)}_{\lambda/\mu}(1).
$$

For general $n$ note that $Y(t)$ satisfies the Markov property
\begin{multline*}
\P(Y(t+n)=\lambda+\delta\mid Y(t)=\mu+\delta)\\
=\sum_{\nu\preceq\lambda}\P(Y(t+n)=\lambda+\delta\mid Y(t+{n-1})=\nu+\delta)\P(Y(t+n-1)=\nu+\delta\mid Y(t)=\mu+\delta),
\end{multline*}
so the proof is finished using induction on $n$ and the branching from Proposition \ref{branching-rules}.
\end{proof}

\begin{figure}
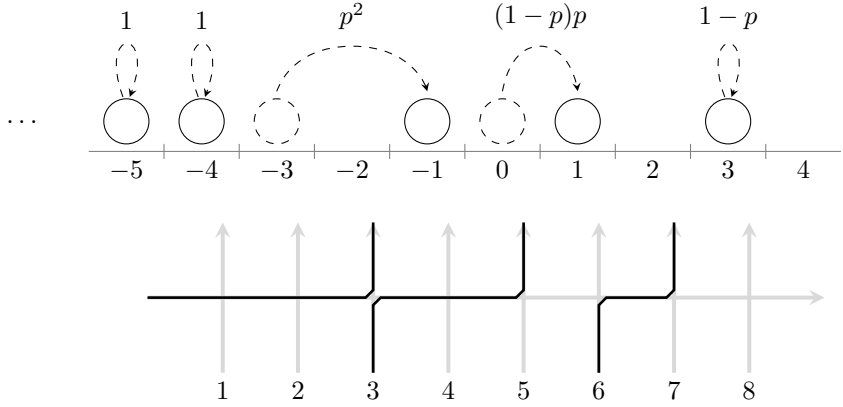


\tikzset{every loop/.style={min distance=10mm}}

\tikz{1}{

    \foreach \x in {-4, ..., 4} {
    	\draw[gray] (\x,-0.1)--(\x,0.1);
	\node[below]  at (\x+0.5, 0)  {$\x$};
	}
	\draw[gray] (-5,0) -- (5,0);
	
	\node[below]  at (-5+0.5, 0)  {$-5$};
		\node[left]  at (-5.5, 0.4)  {$\large\dots$};

    \draw[dashed] (0.5,0.4) circle (0.3cm);

    \draw[dashed] (-2.5,0.4) circle (0.3cm);

    \draw (3.5,0.4) circle (0.3cm);

    \draw(1.5,0.4) circle (0.3cm);

    \draw(-0.5,0.4) circle (0.3cm);

    \draw (-3.5,0.4) circle (0.3cm);

    \draw (-4.5,0.4) circle (0.3cm);
	\node at (3.5, 0.6) {} edge [in=70, out=110, dashed, ->, loop] ();
	\node at (-3.5, 0.6) {} edge [in=70, out=110, dashed, ->, loop] ();
	\node at (-4.5, 0.6) {} edge [in=70, out=110, dashed, ->, loop] ();
	\draw[dashed, ->] plot [smooth, tension=1.5] coordinates { (0.5,0.8) (1,1.4) (1.5,0.8)};
	\draw[dashed, ->] plot [smooth, tension=1.5] coordinates { (-2.5,0.8) (-1.5,1.4) (-0.5,0.8)};
    \node[above] at (3.5, 1.5) {$1-p$};
     \node[above] at (1, 1.5) {$(1-p)p$};
     \node[above] at (-1.5, 1.5) {$p^2$};
          \node[above] at (-3.5, 1.5) {$1$};
          \node[above] at (-4.5, 1.5) {$1$};

}
\newline
\vspace{1em}

\tikz{1}{
	\draw[lgray,line width=1.5pt,->] (0,0) -- (9,0);
	
	\foreach \x in {1, ..., 8} {
		\draw[lgray,line width=1.5pt,->] (\x,-1) -- (\x,1);
		\node[below] at (\x, -1) {$\x$};
	}
	\draw[black, line width=1pt] (0,0) -- (2.9,0) -- (3, 0.1) -- (3,1);
	\draw[black, line width=1pt] (3,-1) -- (3, -0.1) -- (3.1,0) -- (4.9, 0) -- (5, 0.1) -- (5,1);
	\draw[black, line width=1pt] (6,-1) -- (6, -0.1) -- (6.1,0) -- (6.9, 0) -- (7, 0.1) -- (7,1);
}

\caption{Top half: A possible step of TASEP with geometric jumps from $Y(2)$ to $Y(3)$, where $\mu=(4,2)$ and $\lambda=(4,3,2)$. The probability of this step is $(1-p)^2p^3$.\\ Bottom half: the partition function $Z_{3, \lambda/\mu}$ of the five-vertex model, with the same $\lambda,\mu$ as in the top half. The value of this partition function is $(1-p)^2p^3$.\label{modelexample} }
\end{figure}

\begin{figure}
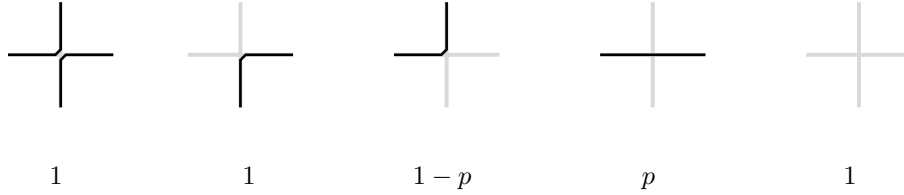

\begin{tabular}{ccccc}
\qquad
\tikz{0.7}{
	\draw[lgray,line width=1.5pt] (-1,0) -- (1,0);
	\draw[lgray,line width=1.5pt] (0,-1) -- (0,1);
	\draw[black, line width=1pt] (-1,0) -- (-0.1, 0) -- (0, 0.1) -- (0,1);
	\draw[black, line width=1pt] (0,-1) -- (0,-0.1) -- (0.1,0) -- (1,0);
}
\qquad
&
\quad
\tikz{0.7}{
	\draw[lgray,line width=1.5pt] (-1,0) -- (1,0);
	\draw[lgray,line width=1.5pt] (0,-1) -- (0,1);
	\draw[black, line width=1pt] (0,-1) -- (0,-0.1) -- (0.1,0) -- (1,0);
}
\qquad
&
\qquad
\tikz{0.7}{
	\draw[lgray,line width=1.5pt] (-1,0) -- (1,0);
	\draw[lgray,line width=1.5pt] (0,-1) -- (0,1);
	\draw[black, line width=1pt] (-1,0) -- (-0.1, 0) -- (0, 0.1) -- (0,1);
}
\qquad
&
\qquad
\tikz{0.7}{
	\draw[lgray,line width=1.5pt] (-1,0) -- (1,0);
	\draw[lgray,line width=1.5pt] (0,-1) -- (0,1);
	\draw[black, line width=1pt] (-1,0) -- (1,0);
}
\qquad
&
\qquad
\tikz{0.7}{
	\draw[lgray,line width=1.5pt] (-1,0) -- (1,0);
	\draw[lgray,line width=1.5pt] (0,-1) -- (0,1);
}
\qquad
\\[1.3cm]
\qquad
$1$
\qquad
&
\qquad
$1$
\qquad
&
\qquad
$1-p$
\qquad
&
\qquad
$p$
\qquad
&
\qquad
$1$
\qquad
\end{tabular}
\caption{Weights of possible local vertex configurations in the five-vertex model.\label{five-vertex-fig}}
\end{figure}

The second model is the \emph{stochastic five-vertex model}. Consider the infinite two-dimensional lattice $\Omega=\mathbb Z^2_{\geq 1}$ with vertices of the form $(i,j)\in\Z^2_{\geq 1}$. These vertices are connected by (oriented) edges $(i,j) \to (i, j+1)$ and $(i,j) \to (i+1, j)$ for $i,j\geq 1$, which we call internal edges. We also consider edges $(i,0) \to (i, 1)$ and $(0,j) \to (1, j)$ for $i,j\geq 1$, which we call boundary edges. A configuration of the five-vertex model is an assignment of "filled" or "empty" to each edge of the lattice such that the four edges around each vertex have one of the five possible local configurations depicted in Figure \ref{five-vertex-fig}. Note that all valid configurations have the following conservation law: the number of filled incoming edges is equal to the number of filled outgoing edges. A configuration is said to satisfy the \emph{domain-wall boundary condition} if all horizontal boundary edges $(0,j)\to(1,j)$ are filled while all vertical boundary edges $(i,0)\to(i,1)$ are empty. See Figure \ref{Fig_5v_configurations} for an example. More generally, we can consider configurations of the five-vertex model on an arbitrary domain $\Omega'\subset \Omega$, where we only fill in the edges containing at least one vertex from $\Omega'$. In this case edges connecting a vertex of $\Omega'$ with a vertex outside of $\Omega'$ are called the boundary edges of $\Omega'$.

 To each of the five possible local configurations around a vertex we assign a weight depending on $p\in (0,1)$, see Figure \ref{five-vertex-fig}. For a configuration $P$ on a finite domain $\Omega'$ we define its weight $w_{\Omega'}(P)$ by taking the product of the weights of all vertices in $\Omega'$. Given a configuration $b$ on the boundary edges of $\Omega'$ we define the partition function $Z_{\Omega',b}=\sum_{P}w_{\Omega'}(P)$, where the sum is over configurations on $\Omega'$ such that the states of boundary edges coincide with $b$. We distinguish two particular cases of this construction. Let $n\in\mathbb Z_{>0}$ and $\lambda,\mu$ be partitions such that $l(\lambda)\leq n, l(\mu)\leq n-1$.  Consider the $n$th row of $\mathbb Z^2_{\geq 1}$ with the following boundary condition:
 \begin{itemize}
 \item The left boundary edge $(0,n)\to(1,n)$ is filled;
 \item On the bottom boundary the edges of the form $(\mu_i+n-i, n-1)\to(\mu_i+n-i, n)$ for $i=1,\dots, n-1$ are filled;
 \item On the top boundary the edges of the form $(\lambda_i+n+1-i, n-1)\to(\lambda_i+n+1-i, n)$ for $i=1,\dots, n$ are filled;
 \item All other boundary edges are empty.
 \end{itemize}
 We use $Z_{n,\lambda/\mu}$ to denote the partition function of $n$th row with the boundary conditions above. Note that this infinite partition function is well-defined, since all edges to the right of column $\lambda_1+n$ must be empty and the corresponding vertices have weight $1$, so $Z_{n,\lambda/\mu}$ can be computed using only the finite rectangle $[1, \lambda_1+n]\times\{n\}$. See the bottom part of Figure \ref{modelexample} for an example. Similarly, we use $Z_{[1,n],\lambda}$ to denote the partition function of the bottom $n$ rows with the boundary conditions
 \begin{itemize}
 \item All left boundary edges are filled;
 \item On the top boundary the edges of the form $(\lambda_i+n+1-i, n-1)\to(\lambda_i+n+1-i, n)$ for $i=1,\dots, n$ are filled;
 \item All other boundary edges are empty.
 \end{itemize}

 \begin{prop}\label{partitionG} We have the following expression for the row partition functions:
 $$
 Z_{n,\lambda/\mu}=(1-p)p^{|\lambda|-|\mu|}G^{(0,-p)}_{\lambda/\mu}(1),\qquad  Z_{[1;n],\lambda}=(1-p)^np^{|\lambda|}G^{(0,-p)}_{\lambda}(1^n).
 $$
 \end{prop}
 \begin{proof}
 For $Z_{n,\lambda/\mu}$ note that there exists at most one configuration satisfying the boundary conditions of $Z_{n,\lambda/\mu}$ because the state of each horizontal edge can be determined using the conservation law. This unique configuration can be visualized as a collection of $n-1$ up-right paths which connect $(\mu_i+n-i,n-1)$ and $(\lambda_{i}+n+1-i,n+1)$ for $i=1,\dots, n-1$ and a path entering from the left and leaving at column $\lambda_n+1$, see the bottom part of Figure \ref{modelexample}. These paths cannot be completely vertical since this would lead to a prohibited vertex configuration, hence $\lambda_{i}\geq \mu_i$. Since each edge can contain at most one path, the path exiting at column $\lambda_{i+1}+n-i$ should be weakly to the left of the next path entering at column $\mu_{i}+n-i$, leading to $\mu_i\geq \lambda_{i+1}$. So $Z_{n,\lambda/\mu}=0$ unless $\mu\preceq\lambda$. When $\mu\preceq\lambda$ all vertices strictly before column $\lambda_n+1$ and all vertices strictly between columns $\mu_i+n-i$ and $\lambda_{i}+n+1-i$ must have weight $p$, leading to $|\lambda|-|\mu|$ vertices with weight $p$. At the same time, vertices with weight $1-p$ are the right-most non-empty vertex $(\lambda_1+n, n)$ and vertices of the form $(\lambda_{i+1}+n-i,n)$ when $\lambda_{i+1}+n-i<\mu_i+n-i$. This leads to $r(\mu/\overline{\lambda})+1$ vertices with weight $1-p$. Comparing with Proposition \ref{Gbranch}, we get $Z_{n,\lambda/\mu}=(1-p)p^{|\lambda|-|\mu|}G^{(0,-p)}_{\lambda/\mu}(1)$.

 For $Z_{[1;n],\lambda}$ note that we have the following branching rule
 \begin{equation}\label{branchingZ}
 Z_{[1;n],\lambda}=\sum_{\mu:l(\mu)\leq n-1} Z_{[1;n-1],\mu} Z_{n,\lambda/\mu}.
 \end{equation}
Indeed, each configuration $P$ computing the partition function $Z_{[1;n],\lambda}$ can be sliced into a pair of configurations $P_1$ and $P_2$ such that $P_1$ is a configuration on the bottom $n-1$ rows, $P_2$ is a configuration on the $n$th row and the top boundary of $P_1$ coincides with the bottom boundary of $P_2$. Due to the conservation law, there must be exactly $n-1$ filled edges between rows $n$ and $n-1$, so the configuration between these rows can be encoded using a partition $\mu$ of length $\leq n-1$ in the same way as in $Z_{[1;n-1],\mu}$. This gives a weight preserving correspondence between configurations computing $Z_{[1;n],\lambda}$ and pairs of configurations computing $Z_{[1;n-1],\mu} Z_{n,\lambda/\mu}$, leading to \eqref{branchingZ}. Then the claim follows from the expression for $Z_{n,\lambda/\mu}$ and the branching rule of Proposition \ref{Gbranch}.
 \end{proof}


\section{Branching graph with Grothendieck weights} \label{Section_Branching_graph} In this section we describe the branching graph $\G^p$, which can be defined using Grothendieck polynomials $G_{\lambda/\mu}^{(0,-p)}$. We also remind the general notation used in the context  of graded graphs and construct a wide class of coherent systems on $\G^p$.

\subsection{Definition of the graph} Let $p\in [0,1)$. We consider an oriented weighted graded graph $\G^p$ with the vertex set $\bigsqcup_{n\geq 0}\G^p_n$, where the $n$th degree component $\G^p_n$ consists of partitions $\lambda$ of length $\leq n$, i.e. $n$-tuples of non-negative integers $\lambda_1\geq\lambda_2\geq\dots\geq\lambda_n\geq0$. In particular $\G^p_0$ is a singleton $\{\varnothing\}$. The edges of $\G^p$ are defined by requiring all edges to increase the degree by $1$ and $\mu\in\G^p_n$ goes to $\lambda\in\G^p_{n+1}$ if and only if $\mu\preceq\lambda$. To each edge $\mu\to\lambda$ we assign the weight (c.f. \eqref{Gone})
$$
w^p(\mu,\lambda)=G_{\lambda/\mu}^{(0,-p)}(1)=(1-p)^{r(\mu/\overline{\lambda})}.
$$
Note that all these weights are non-zero.

For a pair of vertices $\mu\in\G^p_n$, $\lambda\in\G^p_{n+k}$ with $n,k\geq 0$ we define
$$
\dim^{p}_{n,n+k}(\mu,\lambda)=\sum_{\mu=x_0\to x_1\to\dots\to x_{k}=\lambda} w^p(x_0,x_1)\dots w^p(x_{k-1}, x_k),
$$
where the sum is over all directed paths $x_0\to x_1\to\dots\to x_{k}$ connecting $\mu$ to $\lambda$ in $\G^p$. When $k=0$ we set $\dim^p_{n,n}(\mu,\lambda)=\delta_{\mu=\lambda}$ and when $n=0$ we write $\dim_n^p(\lambda):=\dim_{0,n}^p(\varnothing,\lambda)$, which is the sum of weights of all paths connecting $\varnothing$ to $\lambda$. Note that due to the branching rule of Proposition \ref{Gbranch} we have
$$
\dim^{p}_{n,n+k}(\mu,\lambda)=G^{(0,-p)}_{\lambda/\mu}(1^k).
$$
Define \emph{cotransition probabilities} $p^{\downarrow}_{n+1,n}(\lambda,\mu)$ for a pair of vertices $\mu\in\G^p_n, \lambda\in \G^p_{n+1}$ by
$$
p^{\downarrow}_{n+1,n}(\lambda,\mu):=\frac{w^p(\mu,\lambda)\dim_{n}^p(\mu)}{\dim_{n+1}^p(\lambda)}=\frac{G_{\lambda/\mu}^{(0,-p)}(1)G_\mu^{(0,-p)}(1^n)}{G_{\lambda}^{(0,-p)}(1^{n+1})}.
$$
Note that $p^{\downarrow}_{n+1,n}(\lambda,\mu)\geq 0$ for any $\mu\in\G^p_n, \lambda\in\G^p_{n+1}$ and for fixed $\lambda\in\G^p_{n+1}$ we have $\sum_{\mu\in\G^p_n}p^{\downarrow}_{n+1,n}(\lambda,\mu)=1$. So cotransition probabilities define Markov kernels $p^{\downarrow}_{n+1,n}:\G^p_{n+1}\dashrightarrow\G^p_n$ for $n\geq 0$. More generally, for $k\geq 0$ and $\mu\in\G^p_n, \lambda\in \G^p_{n+k}$ we define
\begin{equation}\label{downpdef}
p^{\downarrow}_{n+k,n}(\lambda,\mu)=\frac{\dim^p_{n,n+k}(\mu,\lambda)\dim_{n}^p(\mu)}{\dim_{n+k}^p(\lambda)}
\end{equation}

When $p>0$ the cotransition probabilities of $\G^p$ can also be described in terms of the TASEP with geometric jumps described in Section \ref{TASEPsec}. Namely, from Proposition \ref{GTASEP}  we have
$$
p^{\downarrow}_{n+1,n}(\lambda,\mu)=\dfrac{\P(Y(n+1)=\lambda+\delta\mid Y(n)=\mu+\delta)\ \P(Y(n)=\mu+\delta)}{\P(Y(n+1)=\lambda+\delta)}=\P(Y(n)=\mu+\delta\mid Y(n+1)=\lambda+\delta).
$$
In other words, the down-transition probabilities of $\G^p$ are exactly the cotransition probabilities of the process $Y(t)$ when $p\in (0,1)$.

Let $\calM(\G^p_n)$ denote the space of Borel probability measures on $\G^p_n$, where the latter is equipped with the discrete topology. Note that $\calM(\G^p_n)$ is a convex subset of the vector space $\mathbb R^{|\G^p_n|}$ of signed measures on $\Gamma_n$. Then the Markov kernels $p^{\downarrow}_{n+1,n}$ naturally induce the chain
$$
\calM(\G^p_0)\leftarrow\calM(\G^p_1)\leftarrow\dots \leftarrow\calM(\G^p_n)\leftarrow\dots
$$
of affine maps of convex sets.

\begin{defin} A \emph{coherent system} on $\G^p$ is an element of the projective limit $\varprojlim \calM(\G^p_n)$, that is, a family of measures $\{M_n\in\calM(\G^p_n)\}_{n\geq 0}$ such that for every $\mu\in\G^p_n$ the following coherency relation holds
\begin{equation}
\label{coherency-eq}
M_n(\mu)=\sum_{\lambda\in\G^p_{n+1}}p^{\downarrow}_{n+1,n}(\lambda,\mu) M_{n+1}(\lambda).
\end{equation}
Using the embedding $\varprojlim \calM(\G^p_n)\to\prod_{n\geq 0}\calM(\G^p_n)$ we define the Borel structure on the set of coherent systems $\varprojlim \calM(\G^p_n)$.
\end{defin}

Let $\mathcal T$ denote the set of paths on the graph $\G^p_n$, each such path is represented by a sequence of partitions $t=(t_n)$ where $t_n\in\G^p_n$ and $t_{n}\preceq t_{n+1}$. Clearly we have the embedding $\calT\subset \prod_{n\geq 0}\G^p_n$. Using the product topology on $\prod_{n\geq 0}\G^p_n$, $\calT$ is a closed subset of this product space and we equip $\calT$ with the induced topology. Let $\calT_{\leq N}$ denote the space of finite paths $\tau=(\tau_1,\dots, \tau_N)$ where $\tau_i\in\G_i^p$. We have continuous projection $\calT\to\calT_{\leq N}$ obtained by remembering only the first $N$ steps of a path. For $\tau\in \calT_{\leq N}$ define the \emph{cylinder set} $C_\tau\subset\calT$ as the preimage of $\{\tau\}$ under this projection, that is
$$
C_\tau=\{t\in\calT: t_1=\tau_1, \dots, t_N=\tau_N\}.
$$
For $\tau\in\calT_{\leq N}$ let $w^p(\tau)$ denote the product of all weights of edges along $\tau$.

\begin{defin}\label{centraldef} A Borel probability measure $M$ on $\calT$ is called \emph{central} if for every finite path $\tau\in \calT_{\leq N}$ the value $\frac{M(C_\tau)}{w^p(\tau)}$ depends only on the last vertex $\tau_{N}$. We use $M_n$ to denote the measure on $\G^p_n$ induced along the projection $\calT\to \G^p_n$.
\end{defin}
\begin{prop}\label{centrlacoheq} The assignment $M\mapsto \{M_n\}_{n\geq 0}$ is a one-to-one correspondence between central measures $M$ on $\calT$ and coherent systems $\{M_n\}_{n\geq 0}$ on $\G^p$. The reverse correspondence is determined by setting
$$
M(C_\tau)=\frac{w^p(\tau)}{\dim_n^p(\tau_n)} M_n(\tau_n)
$$
for $\tau\in\calT_{\leq n}$.
\end{prop}
\begin{proof} Straightforward modification of \cite[Proposition 10.3]{Olsh01}.
\end{proof}

In this work we consider the problem of classification of all coherent systems on $\G^p$. Note that $\varprojlim \calM(\G^p_n)$ forms a convex subset of $\prod_{n\geq 0}\calM(\G^p_n)$. Theorem \ref{extremeredtheo} below reduces our classification problem to describing extreme points of $\varprojlim \calM(\G^p_n)$.

\begin{defin} The set of extreme points of the convex set $\varprojlim \calM(\G^p_n)$ is called the \emph{boundary} of $\G^p$. We denote it by $\Omega^p$, and for $\omega\in\Omega^p$ the corresponding coherent system is denoted by $M^{\omega}_n$.
\end{defin}

\begin{theo}[\cite{Olsh01}]\label{extremeredtheo} $\Omega^p$ is a Borel subset of $\varprojlim \calM(\G^p_n)$. For every coherent system $\{M_n\}_n$ on $\G^p$ there exists the unique Borel measure $m$ on $\Omega^p$ such that
$$
M_n(\lambda)=\int_{\Omega^p} M_n^\omega(\lambda)m(d\omega)\qquad\forall \lambda\in\G^p_n.
$$
Conversely, for each Borel measure $m$ on $\Omega^p$ the relation above defines a coherent system on $\Omega^p$. This gives a one-to-one correspondence between Borel measures on $\Omega^p$ and coherent systems on $\G^p$.
\end{theo}

\begin{rem}\label{GT-rem} When $p=0$ the graph $\G^p$ degenerates to a half of the extensively studied Gelfand-Tsetlin graph. In this case the boundary is described as follows. $\Omega=\Omega^0$ is the set of triples $(\alpha,\beta,\delta)$, where $\alpha,\beta$ are infinite real sequences
$$
\alpha_1\geq \alpha_2\geq\alpha_3\geq\dots\geq 0, \qquad 1\geq\beta_1\geq\beta_2\geq\beta_3\geq\dots\geq 0
$$
and $\delta$ is a real number such that $\sum_i\alpha_i+\sum_i\beta_i\leq \delta$. Embedding $\alpha,\beta$ into the infinite countable product $\mathbb R_{\geq 0}^{\infty}$, we define topology $\Omega$ by inducing from the product topology of $\mathbb R_{\geq 0}^{\infty}\times \mathbb R_{\geq 0}^{\infty}\times \mathbb R_{\geq 0}$. For a point $\omega=(\alpha,\beta,\delta)\in\Omega$ we also set $\gamma=\delta-\sum_i\alpha_i-\sum_i\beta_i$. Define
$$
\Phi^\omega(z):=e^{\gamma(z-1)}\prod_{i=1}^\infty\frac{1+\beta_i(z-1)}{1-\alpha_i(z-1)}.
$$
When treated as a function in $z$ and with fixed $\omega\in\Omega$ this is a meromorphic function on $\mathbb C$, analytic in a neighborhood of the unit disk. For $\lambda\in\G_n^0$ define $M_n(\lambda)$ by taking a finite collection of variables $z_1,\dots, z_k$ in a unit disk and writing its Taylor expansion
$$
\sum_{\lambda}M^{\omega}_n(\lambda)\frac{s_\lambda(z_1,\dots,z_k)}{s_\lambda(1^k)}=\Phi^\omega(z_1)\dots\Phi^\omega(z_k).
$$
Then $\{M^\omega_n\}_n$ is an extreme coherent system on $\G^0$ and $\Omega$ and every extreme coherent system is described in this way. See \cite{okounkov1998asymptotics,borodin2012boundary,petrov2014boundary,gorin2015asymptotics}.
\end{rem}

\begin{rem}\label{5-vertexrem} When $p\neq 0$ the coherent systems on $\G^p$ are equivalent to Gibbs measures for the five-vertex model from Section \ref{TASEPsec}. To see it, note that by Proposition \ref{centrlacoheq} coherent system are equivalent to central measures on $\mathcal T$. Now, given a path $t=(t^{(n)})\in\mathcal T$\footnote{Here we use a superscript to allow the notation $t_i^{(n)}$ for $i$th part of the partition $t^{(n)}\in\G^p_n$.} we can construct a five-vertex model configuration $\sigma(t)$ on $\mathbb Z_{\geq 1}^2$ as follows:  for each $i\geq 1$ consider the up-right path starting at $(0,i)$ and with vertical steps $(t_i^{(j)}+j+1-i, j)\to (t_i^{(j)}+j+1-i, j+1)$ for $j\geq i$. Then $\sigma(t)$ is the configuration obtained by filling only the edges contained in these up-right paths, see Figure \ref{example-central-gibbs} for an example.

The configuration $\sigma(t)$ satisfies \emph{domain-wall boundary conditions}: all boundary edges on the left boundary are filled while the boundary edges on the bottom boundary are empty. Moreover, $\sigma(t)$ has another restriction: for every $n$ there are exactly $n$ filled vertical edges between rows $n$ and $n+1$. Let $\Sigma$ denote the set of configurations of the five-vertex model satisfying these two conditions. Embedding $\Sigma$ into the product $\prod_{i,j\geq 1}S$, where $S$ is the five element set of possible local vertex configurations with the discrete topology, we equip $\Sigma$ with the product topology. Then, using the map $t\mapsto\sigma(t)$, a probability measure $M$ on $\mathcal T$ induces a probability measure on $\Sigma$, which we denote by $\P_M$.

Note that the correspondence $t\mapsto \sigma(t)$ preserves weights up to a scaling. Namely, for $n>0$ let $t^{(\leq n)}$ denote the restriction of a path $t$ to its first $n$ steps and let $\sigma_{\leq n}$ denote the restriction of a five-vertex configuration $\sigma$ to the first $n$ rows $\Z_{\geq 1}\times [1;n]$. Then the path weight $w^p(t^{(\leq n)})$ is equal to $G^{(0,-p)}_{t^{(n)}/t^{(n-1)}}(1)\dots G^{(0,-p)}_{t^{(1)}/t^{(0)}}(1)$, while the weight of the corresponding configuration $\sigma(t)_{\leq n}$ is given by $Z_{n,t^{(n)}/t^{(n-1)}}\dots Z_{1,t^{(1)}/t^{(0)}}$. By Proposition \ref{partitionG} we get
$$
w_{ \Z_{\geq 1}\times [1;n]}(\sigma(t)_{\leq n})=p^{|t^{(n)}|}(1-p)^n w^p(t^{(\leq n)}).
$$
Then the defining property of a central measure $M$ from Definition \ref{centraldef} is equivalent to the following Gibbs property for $\P_M$. Let $\rho$ be a five-vertex model configuration on $\Z_{\geq 1}\times [1;n]$ such that on the left and bottom boundaries $\rho$ satisfies the domain-wall boundary conditions while on the top boundary only the edges at columns $\lambda_i+n+1-i$ are filled for a partition $\lambda\in\G_n^p$. Then
$$
\P_M\left(\restr{\sigma}{\Z_{\geq 1}\times[1;n]}=\rho\right)=\frac{w_{ \Z_{\geq 1}\times [1;n]}(\rho)}{Z_{[1;n],\lambda}}.
$$
In other words, if $\rho$ is the restriction to the first $n$ rows of a $\P_M$-random configuration $\sigma$ and we have a fixed configuration of edges between rows $n$ and $n+1$ with $n$ filled edges, the conditional distribution satisfies the Gibbs property
$$
\P(\rho\mid \text{the\ configuration\ of\ edges\ between\ rows\ }n\ \text{and\ } n+1)\sim w_{\Z_{\geq 1}\times [1;n]}(\rho).
$$
This construction can also be reversed, producing a central measure $M$ on $\G^p$ from a Gibbs measure on $\Sigma$.
\end{rem}

\begin{figure}
\tikz{1.2}{
	\foreach \x in {1, ..., 10} {
		\draw[dashed,line width=0.5pt] (\x,0.5) -- (\x,5.5);
		\node[below] at (\x, 0) {$\x$};
	}
	\foreach \y in {1, ..., 5} {
		\draw[dashed,line width=0.5pt] (0.5, \y) -- (10.5, \y);
		\node[left] at (0, \y) {$\y$};
	}
	
	\draw[black, line width=1.5pt] (0.5,1) -- (2.9,1) -- (3, 1.1) -- (3,1.9) -- (3.1, 2) -- (4.9, 2) -- (5, 2.1) -- (5, 2.9) -- (5.1, 3) -- (7.9, 3) -- (8, 3.1) -- (8, 3.9) -- (8.1, 4) -- (8.9, 4) -- (9, 4.1) -- (9, 4.9) -- (9.1 ,5) -- (9.9, 5) -- (10, 5.1) -- (10, 5.5);
	\draw[black, line width=1.5pt] (0.5,2) -- (1.9,2) -- (2, 2.1) -- (2, 2.9) -- (2.1, 3) -- (2.9, 3) -- (3, 3.1) -- (3, 3.9) -- (3.1, 4) -- (3.9, 4) -- (4, 4.1) -- (4, 4.9) -- (4.1 ,5) -- (6.9, 5) -- (7, 5.1) -- (7, 5.5);
	\draw[black, line width=1.5pt] (0.5,3) -- (0.9, 3) -- (1, 3.1) -- (1, 3.9) -- (1.1, 4) -- (1.9, 4) -- (2, 4.1) -- (2, 4.9) -- (2.1 ,5) -- (2.9, 5) -- (3, 5.1) -- (3, 5.5);
	\draw[black, line width=1.5pt] (0.5,4) -- (0.9, 4) -- (1, 4.1) -- (1, 4.9) -- (1.1 ,5) -- (1.9, 5) -- (2, 5.1) -- (2, 5.5);
	\draw[black, line width=1.5pt] (0.5,5) -- (0.9, 5) -- (1, 5.1) -- (1, 5.5);
}
\caption{The five-vertex model configuration corresponding to a path starting from $\varnothing\to(2)\to (3,1)\to (5,1,0)\to (5,1,0,0)\to(5,3,0,0,0)\to\dots$.\label{example-central-gibbs}}
\end{figure}

\subsection{Coherent systems $M^\Phi_n$} Here we construct a family of coherent systems $M^\Phi_n$ on $\G^p$ using specializations of Grothendieck polynomials. From now on we fix $p\in [0,1)$.

Let $\overline{\mathbb D}$ denote the closed unit disk $\{z\in\mathbb C: |z|\leq 1\}$. Let $\mathcal F$ be the space of complex functions $\Phi(z)$ on $\overline{\mathbb D}$ such that there exists a sequence $\Phi_k(z)$ satisfying
\begin{itemize}
\item $\Phi_k(z)$ are complex functions converging to $\Phi(z)$ uniformly on $\overline{\mathbb D}$;
\item $\Phi_k(z)$ are of the form
\begin{equation}\label{Phik}
\Phi_k(z)=\prod_{i=1}^{n_k}\frac{1-x_{i,k}(z-1)}{1-y_{i,k}(z-1)},
\end{equation}
where $x_{i,k}\in [-1, \frac{p}{1-p}]$, $y_{i,k}\geq 0$, $y_{i,k}\geq x_{i,k}$ for every $i,k$.
\end{itemize}
Note that $\Phi(z)\in\mathcal{F}$ is analytic on the unit disk (the poles of $\Phi_k$ are at $\frac{1+y_{i,k}}{y_{i,k}}>1$) and $\Phi(1)=1$. We equip $\mathcal F$ with the topology of uniform convergence on $\overline{\mathbb{D}}$, which makes $\mathcal F$ a complete metric space.

For $n>0$, $\lambda\in\G^p_n, \Phi\in\mathcal F$ we define $M_n^{\Phi}(\lambda)$ by the following identity:
\begin{equation}
\label{defofM}
\sum_{\lambda\in\G^p_n} M_n^\Phi(\lambda) \frac{G_{\lambda}^{(0,-p)}(z_1, \dots, z_n)}{G_{\lambda}^{(0,-p)}(1^n)}=\Phi(z_1)\Phi(z_2)\dots\Phi(z_n).
\end{equation}
Also set $M^{\Phi}_0(\varnothing):=1$. We postpone the convergence questions until later. For now, we treat the identity above as an identity of formal power series in $z_1,\dots, z_n$, where $\Phi(z)\in\mathcal F$ is identified with its Taylor expansion at $0$. We also should explain why $M_n^\Phi(\lambda)$ are well-defined, which is done in Proposition \ref{basics-wd}. More generally, for $\lambda,\mu\in \G_n^p$ we define $M_n^{\Phi}(\lambda/\mu)$ by
\begin{equation}
\label{defofM-skew}
\sum_{\lambda\in\G^p_n} M_n^\Phi(\lambda/\mu) \frac{G_{\lambda}^{(0,-p)}(z_1, \dots, z_n)}{G_{\lambda}^{(0,-p)}(1^n)}=\Phi(z_1)\Phi(z_2)\dots\Phi(z_n)\frac{G_{\mu}^{(0,-p)}(z_1, \dots, z_n)}{G_{\mu}^{(0,-p)}(1^n)}.
\end{equation}
Our goal is to show that $\{M_n^\Phi\}_n$ form a coherent system on $\G^p$.

\begin{prop}\label{topologicalprop} Let $\Phi\in\mathcal F$. Then the Taylor expansion of $\Phi$ at $0$ has non-negative coefficients and converges uniformly on $\overline{\mathbb D}$.
\end{prop}
\begin{proof} First note that for $x\in [-1, \frac{p}{1-p}]$, $y\geq 0$, $y\geq x$ and $z\in\overline{\mathbb D}$ we have
$$
\frac{1-x(z-1)}{1-y(z-1)}=\frac{1+x}{1+y}+\sum_{n\geq 1} \frac{(y-x)y^{n-1}}{(1+y)^{n+1}}z^n,
$$
where all the coefficients $\frac{(y-x)y^{n-1}}{(1+y)^{n+1}}$ are non-negative. Hence for the function $\Phi_k(z)$ from \eqref{Phik} the Taylor expansion at $0$ is of the form 
$$
\Phi_k(z)=\sum_{n\geq 0} a_{n,k} z^n,
$$
where $a_{n,k}\geq 0$ and the series converges for $|z|\leq 1$. Since $\Phi_k(1)=1$ we also know that $\sum_{n\geq 0} a_{n,k}=1$.

Let $\Phi\in\mathcal F$. Then there is a sequence $(\Phi_k)_k$ of the form \eqref{Phik} converging uniformly to $\Phi$. By Cauchy's integral formula the Taylor expansions
$$
\Phi_k(z)=\sum_{n\geq 0} a_{n} z^n, \qquad \Phi_k(z)=\sum_{n\geq 0} a_{n,k} z^n
$$ 
satisfy $a_n=\lim_{k\to\infty} a_{n,k}\geq 0$ and $\sum_{n\geq 0} a_n\leq \lim_{k\to\infty} \sum_{n\geq 0} a_{n,k}=1$. This implies that $\sum_{n\geq 0} a_{n} z^n$ converges uniformly on $\overline{\mathbb D}$.
%
\end{proof}

\begin{prop}\label{basics-wd} The identities \eqref{defofM}, \eqref{defofM-skew} yield well-defined numbers $M_n^\Phi(\lambda), M_n^\Phi(\lambda/\mu)$  respectively. Moreover, $M_n^\Phi(\lambda/\mu)=0$ unless $|\lambda|\geq |\mu|$.
\end{prop}
\begin{proof}
Fix $\mu\in\G^p_n$. Note that $\Phi(z_1)\Phi(z_2)\dots\Phi(z_n)G_{\mu}^{(0,-p)}(z_1, \dots, z_n)$ is a symmetric power series in $z_1, \dots, z_n$. Since the Schur polynomials $\{s_\lambda(z_1,\dots, z_n)\}_{\lambda\in\G^p_n}$ form a basis of symmetric polynomials in $z_1,\dots, z_n$, we have a decomposition of the form
$$
\Phi(z_1)\Phi(z_2)\dots\Phi(z_n)G_{\mu}^{(0,-p)}(z_1, \dots, z_n)=\sum_{\lambda\in\G^p_n}c^\Phi_{\lambda\mu} s_\lambda(z_1,\dots,z_n).
$$
Since $G_\lambda^{(0,-p)}=s_\lambda+\text{higher\ order\ terms}$, we can invert the upper-triangular transition matrix getting $s_\lambda=G^{(0,-p)}_\lambda+\sum_{\nu:|\nu|>|\lambda|}a_{\lambda\nu}G_\nu^{(0,-p)}$. Then each $M_n^{\Phi}(\lambda/\mu)$ is given by the finite sum
$$
M_n^{\Phi}(\lambda/\mu)=G_\lambda^{(0,-p)}(1^n)\sum_{\nu:|\nu|\leq|\lambda|} c^\Phi_{\nu\mu} a_{\nu\lambda}.
$$

For the last statement note that the smallest degree of $\Phi(z_1)\dots\Phi(z_n)G_{\mu}^{(0,-p)}(z_1, \dots, z_n)$ is $|\mu|$, so $c^\Phi_{\lambda,\mu}=0$ unless $|\lambda|\geq |\mu|$. Thus, in the sum above computing $M_n^{\Phi}(\lambda/\mu)$ all terms vanish when $|\lambda|<|\mu|$.
\end{proof}

\begin{prop}\label{branchingMPhi} Let $\Phi_1,\Phi_2\in\mathcal{F}$. Then for any $n\geq 0$ and $\lambda,\nu\in \G^p_n$ we have
$$
M_n^{\Phi_1\Phi_2}(\lambda/\nu)=\sum_{\mu\in \G_n^p}M_n^{\Phi_1}(\lambda/\mu)M_n^{\Phi_2}(\mu/\nu).
$$
\end{prop}
\begin{proof} From the definition of $M^{\Phi_1\Phi_2}_n$ we have
$$
\sum_{\lambda\in\G^p_n} M_n^{\Phi_1\Phi_2}(\lambda/\nu) \frac{G_{\lambda}^{(0,-p)}(z_1, \dots, z_n)}{G_{\lambda}^{(0,-p)}(1^n)}=\frac{G_{\nu}^{(0,-p)}(z_1, \dots, z_n)}{G_{\nu}^{(0,-p)}(1^n)}\prod_{i=1}^n\Phi_1(z_i)\Phi_2(z_i).
$$
On the other hand, we can apply the defining relations for $M^{\Phi_1}_n,M^{\Phi_2}_n$ sequentially to get
\begin{multline*}
\frac{G_{\nu}^{(0,-p)}(z_1, \dots, z_n)}{G_{\nu}^{(0,-p)}(1^n)}\prod_{i=1}^n\Phi_1(z_i)\Phi_2(z_i)=\prod_{i=1}^n\Phi_1(z_i)\sum_{\mu\in\G^p_n} M_n^{\Phi_2}(\mu/\nu) \frac{G_{\mu}^{(0,-p)}(z_1, \dots, z_n)}{G_{\mu}^{(0,-p)}(1^n)}\\
=\sum_{\mu,\lambda\in\G^p_n} M_n^{\Phi_1}(\lambda/\mu)M_n^{\Phi_2}(\mu/\nu) \frac{G_{\lambda}^{(0,-p)}(z_1, \dots, z_n)}{G_{\lambda}^{(0,-p)}(1^n)}.
\end{multline*}
Comparing identities above, the claim follows from the linear independence of $G_{\lambda}^{(0,-p)}(z_1, \dots, z_n)$.
\end{proof}

\begin{cor}\label{shiftcor} Let $\Phi\in \mathcal F$. Then
$$
M^{z\Phi}_n(\lambda)=\begin{cases}M_n^{\Phi}(\lambda-1^n)\qquad& \text{if}\ \lambda_n>0,\\
0\qquad& \text{if}\ \lambda_n=0,
\end{cases}
$$
where $\lambda-1^n=(\lambda_1-1,\lambda_2-1,\dots,\lambda_n-1)$.

\end{cor}
\begin{proof} From Proposition \ref{branchingMPhi} we have
$$
M^{z\Phi}_n(\lambda)=\sum_{\mu\in\G^p_n} M^{z}_n(\lambda/\mu)M^{\Phi}_n(\mu).
$$
However, from the definition \eqref{Gdef} it is clear that
$$
G^{(0,-p)}_{\mu}(z_1,\dots z_n)\prod_{i=1}^nz_i=G^{(0,-p)}_{\mu+1^n}(z_1,\dots z_n),
$$
so $M^{z}_n(\lambda/\mu)=1$ when $\lambda=\mu+1^n$ and $M^{z}_n(\lambda/\mu)=0$ otherwise.
\end{proof}

For the next property assume that $\Phi(0)\neq 0$. Define an algebra homomorphism $\iota_\Phi:\Lambda\to\C$ by its values on the generators $h_k$:
$$
\sum_{k\geq 0}\iota_\Phi(h_k)z^k=\Phi(z)/\Phi(0).
$$
To make expression shorter we use the notation $f(\Phi):=\iota_\Phi(f)$.

\begin{prop}\label{specialization} Let $\Phi\in\mathcal F$ and $\Phi(0)\neq 0$. Then for $\lambda,\mu\in\G_n^p$ we have
$$
M_n^{\Phi}(\lambda/\mu)=g_{\lambda/\mu}^{(0,p)}(\Phi)\Phi(0)^{n} \frac{G^{(0,-p)}_\lambda(1^n)}{G_\mu^{(0,-p)}(1^n)}.
$$
\end{prop}
\begin{proof} We show that the expression above satisfies the definition of $M^\Phi_n(\lambda/\mu)$. Consider the skew Cauchy identity from Proposition \ref{skewCauchy-prop}:
$$
\sum_{\lambda} g^{(0,p)}_{\lambda/\mu}(\bx) G^{(0,-p)}_{\lambda}(z_1,\dots, z_n)=G^{(0,-p)}_{\mu}(z_1,\dots, z_n)\prod_{i=1}^n\left(\sum_{k\geq 0}h_k(\bx)z_i^k\right).
$$
Applying $\iota_\Phi$ to the functions in $\bx$ we get
$$
\sum_{\lambda} g^{(0,p)}_{\lambda/\mu}(\Phi) G^{(0,-p)}_{\lambda}(z_1,\dots, z_n)=G^{(0,-p)}_{\mu}(z_1,\dots, z_n)\prod_{i=1}^n\frac{\Phi(z_i)}{\Phi(0)}.
$$
Rearranging the terms we get
$$
\sum_{\lambda} g^{(0,p)}_{\lambda/\mu}(\Phi)\Phi(0)^{n}\frac{G^{(0,-p)}_{\lambda}(1^n)}{G^{(0,-p)}_{\mu}(1^n)} \frac{G^{(0,-p)}_{\lambda}(z_1,\dots, z_n)}{G^{(0,-p)}_{\lambda}(1^n)}=\frac{G^{(0,-p)}_{\mu}(z_1,\dots, z_n)}{G^{(0,-p)}_{\mu}(1^n)}\prod_{i=1}^n\Phi(z_i).
$$
\end{proof}

\begin{prop}\label{positivityM} Let $\Phi\in\mathcal F$. Then $M_n^{\Phi}(\lambda/\mu)\geq 0$ for any $\lambda,\mu\in\G_n^p$.
\end{prop}
\begin{proof} Using Corollary \ref{shiftcor} we can divide $\Phi(z)$ by $z$ as long as $\Phi(0)=0$. Since $\Phi(z)$ is analytic and $\Phi(1)\neq 0$, its zero at $z=0$ must be of a finite order. So we only need to consider the case $\Phi(0)\neq 0$.

First assume that $\Phi(z)=\frac{1-x(z-1)}{1-y(z-1)}$ with $x\leq y$, $x\in (-1,\frac{p}{1-p}]$, $y\geq 0 $. Then we have
$$
\Phi(z)/\Phi(0)=\frac{1-\frac{x}{1+x}z}{1-\frac{y}{1+y}z}=\psi_u\rho^*_\gamma(H(z)),
$$
where $\gamma=\frac{x}{1+x}$, $\rho^*_\gamma$ is the automorphism from Proposition \ref{adjshiftproperties}, and $\psi_u$ denotes the single-variable specialization $f\mapsto f(u)$ with $u=\frac{y}{1+y}-\frac{x}{1+x}$. Hence $\iota_\Phi$ is the composition $\psi_u\circ\rho_\gamma$ and
$$
g_{\lambda/\mu}^{(0,p)}(\Phi)=g^{(\gamma,p-\gamma)}_{\lambda/\mu}(u).
$$
Since $u, u+\gamma, p-\gamma\geq 0$, Proposition \ref{g-branching-prop} implies $g^{(\gamma,p-\gamma)}_{\lambda/\mu}(u)\geq 0$. Then $M^{\Phi}_n(\lambda/\mu)\geq 0$ by Proposition \ref{specialization}.

For general $\Phi\in\mathcal F$ such that $\Phi(0)\neq 0$, write $\Phi$ as the uniform limit of functions $\Phi_k(z)$ of the form
$$
\Phi_k(z)=\prod_{i=1}^{n_k}\frac{1-x_{i,k}(z-1)}{1-y_{i,k}(z-1)},
$$
where $x_{i,k}\leq y_{i,k}$, $x_{i,k}\in (-1,\frac{p}{1-p}]$, $y_{i,k}\geq 0$. Since the limit is uniform, $\lim_{k\to\infty}h_i(\Phi_k)=h_i(\Phi)$ for any $i$ and $\Phi_k(0)\to \Phi(0)$, thus $g_{\lambda}^{(0,p)}(\Phi_k)\to g_{\lambda}^{(0,p)}(\Phi)$ by Proposition \ref{JT-g}. At the same time, since $\Phi_k$ is the product of $\frac{1-x (z-1)}{1-y(z-1)}$ from the first part of the proof, $g_{\lambda}^{(0,p)}(\Phi_k)\geq 0$ by Proposition \ref{branchingMPhi}. This proves the claim.
\end{proof}

\begin{prop}\label{pre-coh} Let $\Phi\in\mathcal F$, $\Phi(0)\neq 0$. Then for any $\mu\in\G^p_n$ we have
\begin{equation}\label{pre-coheq}
\sum_{\lambda\in\G^p_{n+1}}g^{(0,p)}_{\lambda}(\Phi) G^{(0,-p)}_{\lambda/\mu}(z)=g^{(0,p)}_{\lambda}(\Phi)\frac{\Phi(z)}{\Phi(0)},
\end{equation}
where the sum converges uniformly on $\overline{\mathbb D}$.
\end{prop}
\begin{proof} Using the skew Cauchy identity from Proposition \ref{skewCauchy-prop} we have
$$
\sum_{\lambda} G^{(0,-p)}_{\lambda/\mu}(z)g^{(0,p)}_{\lambda}(x_1, x_2, \dots)=g^{(0,p)}_{\mu}(x_1, x_2, \dots)\prod_{i\geq 1}\frac{1}{1-x_iz}.
$$
Applying the map $\iota_\Phi$ in variables $x_1,x_2,\dots$ we get
$$
\sum_{\lambda\in\G_{n+1}^p} G^{(0,-p)}_{\lambda/\mu}(z)g^{(0,p)}_{\lambda}(\Phi)=g^{(0,p)}_{\mu}(\Phi)\frac{\Phi(z)}{\Phi(0)}.
$$
However, so far the identity above is proved for formal parameter $z$. To demonstrate the convergence we first recall that $G^{(0,-p)}_{\lambda/\mu}(z)$ vanishes unless $\mu\preceq \lambda$, so for fixed $\mu\in\G^p_n$ the sum above is over all $\lambda$ such that $\lambda_i\in [\mu_{i}, \mu_{i-1}]$. In particular, there are finitely many choices for $\lambda_2, \dots, \lambda_{n+1}$ and so it is enough to establish the convergence of
$$
\sum_{\lambda_1\geq \mu_1}G^{(0,-p)}_{\lambda/\mu}(z)g^{(0,p)}_{\lambda}(\Phi)
$$
with fixed $\lambda_2, \dots, \lambda_{n+1}$. Now applying Jacobi-Trudi identity Proposition \ref{JT-g} we have:
$$
g^{(0,p)}_{\lambda}(\Phi)=\det\left[\sum_{k=0}^{i-1}h_{\lambda_i-i+j-k}(\Phi)\binom{i-1}{k}p^k\right]_{1\leq i,j\leq n}=\sum_{k=-n+1}^{n-1} c_k h_{\lambda_1+k}(\Phi),
$$
where $c_k$ are constants which do not depend on $\lambda_1$. Hence, it is enough to verify the convergence of
$$
\sum_{\lambda_1\geq \mu_1}G^{(0,-p)}_{\lambda/\mu}(z)h_{\lambda_1+k}(\Phi).
$$
From \eqref{Gone} we have $|G^{(0,-p)}_{\lambda/\mu}(z)|\leq z^{|\lambda|-|\mu|}$ when $|z|\leq1<p^{-1}$, so the convergence of \eqref{pre-coheq} reduces to the convergence of
$$
\sum_{\lambda_1\geq \mu_1}z^{|\lambda|-|\mu|}h_{\lambda_1+k}(\Phi)=z^{\lambda_2+\dots+\lambda_{n+1}-|\mu|-k}\sum_{i\geq \mu_1+k}h_i(\Phi)z^i.
$$
This last series converges uniformly and absolutely on $\overline{\mathbb D}$ by Proposition \ref{topologicalprop}. So, both sides of \eqref{pre-coheq} are analytic functions on the unit disk with the same Taylor expansions at $0$ due to the formal identities.
\end{proof}

\begin{theo}\label{coherent-thm} $\{M_n^\Phi\}_n$ is a coherent system on $\G^p$ for any $\Phi\in\mathcal F$.
\end{theo}
\begin{proof} We need to check that $M_n^\Phi(\lambda)$ define probability measures on $\G^p$ satisfying the coherency relation \eqref{coherency-eq}. From Proposition \ref{positivityM} we know that $M_n^\Phi(\lambda)\geq 0$ and $\sum_{\lambda\in \G^p_n} M^\Phi_n(\lambda)=1$ follows from coherency and $M_0^{\Phi}(\varnothing)=1$:
$$
\sum_{\lambda\in\G^p_{n+1}} M_{n+1}^\Phi(\lambda)=\sum_{\mu\in\G^p_{n}} M_{n}^\Phi(\mu)=\dots=M_0^\Phi(\varepsilon)=1.
$$
So, it is enough to verify \eqref{coherency-eq}.

First we show that if $M^{\Phi}_n$ satisfy the coherency relation then so does $M^{z\Phi}_n$. Indeed, from Corollary \ref{shiftcor} $M^{z\Phi}_n$ is supported on $\lambda$ with $\lambda_n>0$ and $M^{z\Phi}_n(\lambda)=M^{\Phi}_n(\lambda-1^n)$. So for every $\mu\in\G^p_{n-1}$ we need to establish
$$
M^{z\Phi}_{n-1}(\mu)=\sum_{\lambda\in\G^p_n: \lambda_n>0}p^{\downarrow}_n(\lambda,\mu) M^{\Phi}_{n}(\lambda-1^n).
$$
If $\mu_{n-1}=0$ then both sides vanish. If $\mu_{n-1}>0$ we can use $G_{\lambda+1^n}^{(0,-p)}(1^n)=G_{\lambda}^{(0,-p)}(1^n)$ and $G_{\lambda+1^n/\mu+1^{n-1}}^{(0,-p)}(1)=G_{\lambda/\mu}^{(0,-p)}(1)$ to get
$$
p^{\downarrow}_n(\lambda,\mu)=p^{\downarrow}_n(\lambda-1^n,\mu-1^{n-1})
$$
when $\lambda_n>0$. Hence the coherency of $M_n^{z\Phi}$ follows from the coherency of $M_n^{\Phi}$. So, dividing $\Phi$ by $z$ while $\Phi(0)=0$, we can reduce to the case $\Phi(0)\neq 0$.

Now consider $\Phi\in\mathcal F$ such that $\Phi(0)\neq 0$. By Proposition \ref{pre-coh}
$$
\sum_{\lambda\in\G^p_{n+1}}g^{(0,p)}_{\lambda}(\Phi) G^{(0,-p)}_{\lambda/\mu}(1)=\frac{g^{(0,p)}_{\lambda}(\Phi)}{\Phi(0)}.
$$
Then
\begin{multline*}
\sum_{\lambda\in\G^p_{n+1}}M_{n+1}^\Phi(\lambda)\frac{G^{(0,-p)}_{\lambda/\mu}(1)G^{(0,-p)}_{\mu}(1^n)}{G^{(0,-p)}_{\lambda}(1^{n+1})}=\Phi(0)^{n+1}G^{(0,-p)}_{\mu}(1^n)\sum_{\lambda\in\G^p_{n+1}}g^{(0,p)}_{\lambda}(\Phi)G^{(0,-p)}_{\lambda/\mu}(1)\\
=\Phi(0)^{n}G^{(0,-p)}_{\mu}(1^n)g^{(0,p)}_{\mu}(\Phi)=M_n^\Phi(\mu). \qedhere
\end{multline*}
\end{proof}

\subsection{Independence of $M^{\Phi}_n$} \label{Section_independence} Here we show that the coherent systems $\{M_n^\Phi\}_n$ constructed above cannot be expressed in terms of each other.

First note that $G_{(n)}^{(0,-p)}(z)=z^n$, so from \eqref{defofM} we get
\begin{equation}\label{TaylorM}
\sum_{(n)\in\G^p_1}M^\Phi_1((n)) z^n=\Phi(z).
\end{equation}
In other words, $M^\Phi_1((n))$ are Taylor coefficients of $\Phi$ and the coherent systems $\{M^\Phi_n\}_n$ are different for different $\Phi\in\mathcal F$.

\begin{prop} $\Phi\mapsto\{M_n^\Phi\}_n$ defines a closed embedding $\Upsilon:\mathcal F\hookrightarrow\varprojlim \calM(\G^p_n)$.
\end{prop}
\begin{proof}
First we are going to check that $M_n^{\Phi}(\lambda)$ is a continuous function of $\Phi\in\mathcal F$ for each fixed $n\geq0, \lambda\in\G_n^p$. Recall from the proof of Proposition \ref{basics-wd} that $M_n^{\Phi}(\lambda)$ depends polynomially on $c^\Phi_{\lambda,\varnothing}$, which in turn depend polynomially on the Taylor coefficients of $\Phi$. These Taylor coefficients depend continuously on $\Phi\in\mathcal F$ by Cauchy's integral formula, so we get the desired continuity. Since the topology on $\varprojlim \calM(\G^p_n)$ is the minimal topology such that $M_n(\lambda)$ are continuous, we get that $\Upsilon$ is continuous. 

Now assume that $\Phi_k\in\mathcal F$ is a sequence of functions such that $\Upsilon(\Phi_k)$ converge to a coherent system $\{M_n\}_n$. In other words, $\lim_{k\to\infty} M_n^{\Phi_k}(\lambda)=M_n(\lambda)$ for each $n\geq0, \lambda\in\G_n^p$. Consider
$$
\Phi(z)=\sum_{n\geq 0} M_1((1))z^n.
$$
We claim that $\Phi_k$ converge to $\Phi$ uniformly on $\overline{\mathbb D}$ and $M_n(\lambda)=M_n^\Phi(\lambda)$. This would imply that $\Upsilon(\mathcal F)\subset\varprojlim \calM(\G^p_n)$ is closed and $\Upsilon^{-1}:\Upsilon(\mathcal F)\to\mathcal F$ is continuous, finishing the proof. 

For the rest of the argument let $a_{n,k}=M^{\Phi_k}_1((n))$ and $a_{n}=M_1((n))$ denote the Taylor coefficients of $\Phi_k$ and $\Phi$ respectively. Then  $\lim_{k\to\infty}a_{n,k}=a_n$. For an arbitrary $\varepsilon>0$ fix sufficiently large $N$ such that $\sum_{n>N}a_n<\varepsilon$. Then for sufficiently large $k$ we have $\sum_{n=0}^N |a_n-a_{n,k}|<\varepsilon$, which implies
$$
\sum_{n>N} a_{N,k}=1-\sum_{n=0}^N a_{N,k}<1-\sum_{n=0}^N a_{N}+\varepsilon=\sum_{n>N}^N a_{N}+\varepsilon<2\varepsilon.
$$
Hence for sufficiently large $k$ we get $|\Phi(z)-\Phi_k(z)|<4\varepsilon$ for all $|z|\leq1$. This implies the uniform convergence $\Phi_k\rightrightarrows\Phi$ on $\overline{\mathbb D}$, which in turn leads to $\Phi\in\mathcal F$. Since $M_n^\Phi(\lambda)$ is continuous in $\Phi$, we get $M_n^{\Phi}(\lambda)=\lim_{k\to\infty}M_n^{\Phi_k}(\lambda)=M_n(\lambda)$.
\end{proof}

\begin{prop}\label{avargecoh} Let $\nu$ be a Borel probability measure on $\mathcal F$. Then
\begin{equation}\label{avaragecoh-eq}
M_n(\lambda)=\int_{\mathcal F} M^\Phi_n(\lambda)\nu(d\Phi)
\end{equation}
defines a coherent system on $\G^p$.
\end{prop}
\begin{proof} Since $M^\Phi_n(\lambda)\geq 0$ for any $\Phi\in\mathcal F$, we have $M_n(\lambda)\geq0$. By the monotone convergence theorem we get
$$
\sum_{\lambda\in\G^p_n} M_n(\lambda)=\int_{\mathcal F}\sum_{\lambda\in\G^p_n}  M^\Phi_n(\lambda)\nu(d\Phi)=\int_{\mathcal F}1\ \nu(d\Phi)=1,
$$
$$
\sum_{\lambda\in\G^p_{n+1}} M_{n+1}(\lambda)p^{\downarrow}_{n+1,n}(\lambda,\mu)=\int_{\mathcal F}\sum_{\lambda\in\G^p_{n+1}}  M^\Phi_{n+1}(\lambda)p^{\downarrow}_{n+1,n}(\lambda,\mu)\nu(d\Phi)=\int_{\mathcal F}M^\Phi_n(\mu) \nu(d\Phi)=M_n(\mu).
$$
This implies that $\{M_n\}$ are probability measures on $\G_n^p$ satisfying the coherency relation \eqref{coherency-eq}.
\end{proof}

To prove independence we use the setting of de Finetti's theorem. Let $\mathbb Z_{\geq 0}^\infty$ be the space of infinite non-negative integer sequences $\bx=(x_1, x_2, \dots)$. Treating $\mathbb Z_{\geq 0}$ as a discrete topological space, equip $\mathbb Z_{\geq 0}^\infty$ with the product topology. Then any probability Borel measure $m$ on $\mathbb Z_{\geq 0}^\infty$ is determined by its values on the cylinder sets
$$
C_{n}(\ba)=\{\bx\in\mathbb Z^{\infty}_{\geq 0}: x_1=a_1,\dots, x_n=a_n\},
$$
where $n\geq 0$ and $\ba=(a_1,\dots, a_n)\in\Z_{\geq0}^n$. Conversely, given non-negative numbers $\{c_n(\ba)\}_{n\geq 0, \ba\in\Z_\geq0^n}$ satisfying
$$
\sum_{a_{n+1}\geq 0} c_{n+1}(a_1,\dots, a_{n+1})=c_n(a_1,\dots, a_n),\qquad c_0(\varnothing)=1,
$$
we can construct a unique Borel probability measure $m$ satisfying $m(C_{n}(\ba))=c_n(\ba)$ for all $n\geq 0$, $\ba=(a_1,\dots, a_n)\in\Z_{\geq0}^n$. A Borel measure $m$ on $\mathbb Z_{\geq 0}^\infty$ is called symmetric if for any permutation $\sigma\in\mathfrak S_n$ and any $\ba\in\Z^n_{\geq 0}$ we have
$$
m(C_{n}(a_1,\dots,a_n))=m(C_{n}(a_{\sigma(1)},\dots,a_{\sigma(n)})).
$$
Symmetric Borel probability measures on $\mathbb Z_{\geq 0}^\infty$ form a convex set denoted by $\mathcal M_{sym}(\mathbb Z_{\geq 0}^\infty)$. An important example of symmetric measures is given by the product measures $\mu^\infty$, where $\mu$ is a Borel measure on $\Z_{\geq 0}$ and
$$
\mu^\infty(C_{n}(a_1,\dots,a_n))=\mu(\{a_1\})\dots\mu(\{a_n\}).
$$
De Finetti's theorem implies that the set of extreme points of $\mathcal M_{sym}(\mathbb Z_{\geq 0}^\infty)$ consists of the product measures, see \cite[Theorem 5.3]{deFin}.

\begin{theo} Assume that for a Borel probability measure $\nu$ on $\mathcal F$ and $\Psi\in\mathcal F$ we have
$$
M^\Psi_n(\lambda)=\int_{\mathcal F} M^\Phi_n(\lambda)\nu(d\Phi)
$$
for all $n\geq0, \lambda\in\G^p_n$. Then $\nu$ is the Dirac measure $\delta_\Psi$ concentrated at $\Psi$.
\end{theo}
\begin{proof}
Given a coherent system $\{M_n\}_n$ define $c_n(\mathbf a)$ by the decomposition
\begin{equation}\label{exchangedef}
\sum_{a_1,\dots, a_n\geq 0} c_n(\mathbf a) z_1^{a_1}\dots z_n^{a_n}=\sum_{\lambda\in\G^p_n}M_n(\lambda)\frac{G_\lambda^{(0,-p)}(z_1,\dots, z_n)}{G^{(0,-p)}_\lambda(1^n)}.
\end{equation}
In other words,
$$
c_n(\mathbf a)=\sum_{\lambda\in\G^p_n} b_{\lambda}(\ba)M_n(\lambda),
$$
where for $\ba\in\Z_{\geq0}^n$, $\lambda\in\G^p_n$ we define $b_{\lambda}(\ba)$ as the coefficient of $z_1^{a_1}\dots z_n^{a_n}$ in $\frac{G_\lambda^{(0,-p)}(z_1,\dots, z_n)}{G^{(0,-p)}_\lambda(1^n)}$. Note that the sum above is finite since $b_{\lambda}(\ba)=0$ unless $|\lambda|\leq a_1+\dots +a_n$. In particular, $c_n(\mathbf a)$ can be treated as a continuous function on $\varprojlim \calM(\G^p_n)$. Also, since $\{G_\lambda^{(0,-p)}(z_1,\dots, z_n)\}_{\lambda\in\G^p_n}$ are linearly independent, different coherent systems $\{M_n\}_n$ result in different collections $\{c_n(\mathbf a)\}_{n\geq 0, \ba\in\Z_{\geq0}^n}$.

For $\Phi\in\mathcal F$ let $c_n^{\Phi}(\mathbf a)$ denote the numbers $c_n(\mathbf a)$ corresponding to the coherent system $\{M^{\Phi}_n\}_n$.  Then \eqref{exchangedef} implies
\begin{equation*}
\sum_{a_1,\dots, a_n\geq 0} c_n^{\Phi}(\mathbf a) z_1^{a_1}\dots z_n^{a_n}=\Phi(z_1)\dots\Phi(z_n).
\end{equation*}
Taking the product of $n$ copies of \eqref{TaylorM} we get $c_n^{\Phi}(\mathbf a)=M_1^{\Phi}((a_1))\dots M_1^{\Phi}((a_n))\geq 0$. Hence $c_n^{\Phi}(\ba)$ define a product measure on $\Z_{\geq0}^\infty$, which we denote by $m^{\Phi}$. More generally, when $\{M_n\}_n$ is the coherent system from Proposition \ref{avargecoh} we have
$$
c_n(\ba)=\int_{\mathcal F} c_n^{\Phi}(\ba)\nu(d\Phi).
$$
Taking the integral of $m^{\Phi}$ over $\nu$ we get a symmetric measure $m$ with $m(C_n(\ba))=c_n(\ba)$ for $n\geq0$, $\ba\in\Z_{\geq0}^n$.

To sum it up, for each coherent system form Proposition \ref{avargecoh} we can construct a unique symmetric measure on $\Z_{\geq0}^\infty$, with systems $\{M^\Phi_n\}_n$ corresponding to product measures. Now take $\nu$ and $\Psi$ as in the statement of the theorem and assume that $\nu$ is not a Dirac measure. Since $\varprojlim \calM(\G^p_n)$ is metrizable, $\mathcal F$ is metrizable as well and we can find a Borel decomposition $\mathcal F=A\sqcup B$ such that $\nu(A),\nu(B)>0$. Restricting $\nu$ to $A$ and $B$ and using Proposition \eqref{avargecoh}, we get two coherent systems $\{M^A_n\}_n, \{M^B_n\}_n$ of the form \eqref{avaragecoh-eq} such that $M_n^\Psi=\nu(A) M^A_n+\nu(B)M^B_n$ for $n\geq 0$. However, if $m^A$, $m^B$ denote the symmetric measures on $\Z_{\geq 0}^\infty$ corresponding to $\{M^A_n\}_n$ and $\{M^B_n\}_n$, we get $\nu(A)m^A+\nu(B)m^B=m^\Psi$. This contradicts de Finetti's theorem since $m^\Psi$ is an extreme point of $\mathcal M_{sym}(\Z_{\geq0}^\infty)$.
\end{proof}

\subsection{Examples of $M^{\Phi}_n$} Here we describe several notable examples of $\{M^{\Phi}_n\}_n$.

\emph{Finite GT-type systems:} Let $\mathcal A=(\alpha_1, \dots, \alpha_k), \mathcal B=(\beta_1, \dots, \beta_l)$ be a pair of sequences such that
\begin{equation}\label{ineq}
\alpha_1\geq \alpha_2\geq\dots\geq \alpha_k>\frac{p}{1-p},\qquad 1\geq \beta_1\geq\beta_2\geq\dots \geq \beta_l\geq p.
\end{equation}
Set
\begin{equation}\label{Mab-def}
\Phi(z)=\Phi^{\calA,\calB}(z)=\prod_{i=1}^k\frac{1-\frac{p}{1-p}(z-1)}{1-\alpha_i (z-1)}\prod_{i=1}^l\left(1+\frac{\beta_i-p}{1-p}(z-1)\right).
\end{equation}
We say that the resulting system $M_n^\Phi$ is a \emph{finite system of GT-type} and we denote it by $M^{\calA,\calB}_n$.  Note that in the case $p=0$ these systems degenerate to a part of the boundary of Gelfand-Tsetlin graph in Remark \ref{GT-rem}, hence the name. In Section \ref{finitesect} we describe the limit law of $M^{\calA,\calB}_n$ as $n\to \infty$ and partially prove that these systems are extreme. Here we give a more explicit description for these systems.

\begin{prop}\label{hook} Let $\mathcal A=(\alpha_1, \dots, \alpha_k), \mathcal B=(\beta_1, \dots, \beta_l)$ be sequences satisfying \eqref{ineq}, $s$ denote the number of $i$ such that $\beta_i=1$ and $\tilde{\Phi}^{\calA,\calB}(z)=z^{-s}\Phi^{\calA,\calB}(z)$. We have $M_n^{\calA,\calB}(\lambda)=0$ unless $\lambda_n\geq s$ and
$$
\frac{M_n^\Phi(\lambda+s^n)}{\tilde\Phi^{\calA,\calB}(0)^{n}G_{\lambda+s^n}^{(0,-p)}(1^n)}=\sum_{\mu} g_{\mu}^{(p,0)}(x_1, \dots, x_k)g_{\lambda'/\mu'}^{(p,0)}(y_{s+1}, \dots, y_l)=\sum_{\mu} g_{\lambda/\mu}^{(p,0)}(x_1, \dots, x_k)g_{\mu'}^{(p,0)}(y_{s+1}, \dots, y_l),
$$
where $x_i=\frac{\alpha_i-p(1+\alpha_i)}{1+\alpha_i}$, and $y_i=\frac{\beta_i-p}{1-\beta_i}$.

 In particular, $M_n^\Phi(\lambda)=0$ unless $\lambda_i\leq l$ for every $i>k$, that is, $M_n^\Phi$ is supported by partitions inside the infinite "hook" with $k$ infinite rows and $l$ infinite columns.
\end{prop}
\begin{proof}
Note that $s$ is the degree of zero of $\Phi$ at $z=0$, so using Corollary \ref{shiftcor} we can reduce the claim to the case $s=0$.

Assuming that $s=0$ and $\Phi(0)\neq 0$, we can use Propositions \ref{branchingMPhi}, \ref{specialization} to get
$$
\frac{M_n^\Phi(\lambda)}{\Phi(0)^nG_{\lambda}^{(0,-p)}(1^n)}=\sum_{\mu} g_{\mu}^{(0,p)}(\Phi_1)g_{\lambda/\mu}^{(0,p)}(\Phi_2)=\sum_{\mu} g_{\lambda/\mu}^{(0,p)}(\Phi_1)g_{\mu}^{(0,p)}(\Phi_2),
$$
where
$$
\Phi_1(z)=\prod_{i=1}^k\frac{1-\frac{p}{1-p}(z-1)}{1-\alpha_i (z-1)},\qquad \Phi_2(z)=\prod_{i=1}^l\left(1+\frac{\beta_i-p}{1-p}(z-1)\right).
$$
From Proposition \ref{adjshiftproperties} we see that
$$
\frac{\Phi_1(z)}{\Phi_1(0)}=\prod_{i=1}^k\frac{1-pz}{1-(x_i+p)z}=\pi_{\bx}\ \rho^*_pH(z),\qquad \frac{\Phi_2(z)}{\Phi_2(0)}=\prod_{i=1}^l(1+y_iz)=\pi_{\by}\ \omega H(z),
$$
where $\pi_{\bx}, \pi_{\by}$ are specializations $f\mapsto f(x_1,\dots, x_k), f\mapsto f(y_1,\dots, y_l)$ respectively. Hence $g_{\lambda/\mu}^{(0,p)}(\Phi_1)=[\rho^*_p g^{(0,p)}_{\lambda/\mu}](x_1,\dots, x_k)=g^{(p,0)}_{\lambda/\mu}(x_1,\dots, x_k)$ and $g_{\lambda/\mu}^{(0,p)}(\Phi_2)=[\omega g^{(0,p)}_{\lambda/\mu}](y_1,\dots, y_l)=g^{(p,0)}_{\lambda'/\mu'}(y_1,\dots, y_k)$. This proves the first part of the statement.

The second part follows from the branching rules in Propositions \ref{g-branching-prop},\ref{branching-rules}. Namely, the single variable function $g^{(p,0)}_{\lambda/\mu}(z)$ vanishes unless $\mu\preceq\lambda$, hence in the sum
$$
\sum_{\mu} g_{\mu}^{(p,0)}(x_1, \dots, x_k)g_{\lambda'/\mu'}^{(p,0)}(y_{1}, \dots, y_l)
$$
nonzero terms correspond to $\mu$ such that $l(\mu)\leq k$ and $\lambda/\mu$ is a union of $l$ vertical strips. This forces the condition $\lambda_i\leq l$ for all $i>k$.
\end{proof}

\emph{GT-type systems:} We can extend the previous example to infinite sequences as follows. Let $\mathcal A=(\alpha_1, \alpha_2, \dots), \mathcal B=(\beta_1, \beta_2, \dots)$ be infinite sequences such that
\begin{equation*}
\alpha_1\geq \alpha_2\geq\dots >\frac{p}{1-p},\qquad 1\geq \beta_1\geq\beta_2\geq\dots > p,
\end{equation*}
and
$$
\sum_{i} \left(\alpha_i-\frac{p}{1-p}\right)+\sum_{i} \left(\beta_i-p\right)<\infty.
$$
Moreover, choose two additional parameters $\gamma_{\mathcal A}, \gamma_{\mathcal B}>0$. Then
$$
\Phi(z)=e^{\gamma_{\mathcal A}(z-1)+\gamma_{\calB}\frac{(1-p)(z-1)}{1-pz}}\prod_{i=1}^\infty\frac{1-\frac{p}{1-p}(z-1)}{1-\alpha_i (z-1)}\prod_{i=1}^\infty\left(1+\frac{\beta_i-p}{1-p}(z-1)\right)
$$
is a uniform limit of the functions $\Phi$ from the previous example, so $\Phi\in\mathcal F$ (see the last example in this section for the construction of the exponential factors). When $p=0$ systems $M^{\Phi}_n$ with $\Phi(z)$ as above describe the full boundary of $\G^0$, as follows from the references in Remark \ref{GT-rem}

\emph{Plancherel-type systems:} For our next example first take $t\in [0,\frac{p}{1-p}), c\geq 0$ and consider
$$
\Phi(z)=\lim_{n\to\infty}\left(\frac{1-t(z-1)}{1-(t+cn^{-1})(z-1)}\right)^n=\lim_{n\to\infty}\left(1-\frac{c(z-1)}{n(1-t(z-1))}\right)^{-n}=\exp\left(c\frac{z-1}{1-t(z-1)}\right).
$$
The limit converges uniformly on a neighborhood of the unit disk, so $\Phi(z)\in\mathcal F$. Since $\mathcal F$ is closed under multiplication we get
$$
\exp\left(\sum_{i=1}^kc_i\frac{z-1}{1-t_i(z-1)}\right)\in\mathcal F
$$
for any $t_1, \dots, t_k\in [0,\frac{p}{1-p})$ $c_1,\dots, c_n\geq 0$. Now consider the discrete measure $\nu=\sum_{i=1}^kc_i\delta_{x_i}$ on $[0,\frac{p}{1-p}]$ and rewrite $\Phi(z)$ as
$$
\exp\left(\int_0^{\frac{p}{1-p}}\frac{z-1}{1-t(z-1)}\nu(dt)\right)\in\mathcal F.
$$
Since any finite Borel measure on $[0,\frac{p}{1-p}]$ can be approximated in distribution by discrete measures, we get
$$
\exp\left(\int_0^{\frac{p}{1-p}}\frac{z-1}{1-t(z-1)}\nu(dt)\right)\in\mathcal F.
$$
for any finite Borel measure on $[0,\frac{p}{1-p}]$.

\emph{TASEP with geometric jumps:} Let $\tilde p\in (0,p]$ and consider
$$
\Phi(z)=\left(1-\frac{\tilde{p}}{1-\tilde{p}}(z-1)\right)^{-1}=\frac{1-\tilde p}{1-\tilde pz}.
$$
Then $\frac{\Phi(z)}{\Phi(0)}=(1-\tilde{p}z)^{-1}$ and using Proposition \ref{specialization} we get
$$
M^\Phi_n(\lambda)=g_\lambda^{(0,p)}(\tilde{p})(1-\tilde{p})^nG_\lambda^{(0,-p)}(1^n).
$$
From Proposition \ref{g-branching-prop} we have $g_\lambda^{(0,p)}(\tilde{p})=p^{|\lambda|-\lambda_1}\tilde{p}^{\lambda_1}$ for any $\lambda$, so by Proposition \ref{GTASEP} we get
\begin{equation}\label{TASEPdegen}
M^\Phi_n(\lambda)=\frac{\tilde p^{\lambda_1} (1-\tilde{p})^n }{p^{\lambda_1} (1-p)^n}\P(Y(n)=\lambda+\delta)
\end{equation}
where $Y(n)$ is the TASEP with geometric jumps. When $\tilde{p}=p$, $M^\Phi_n$ are distributions of the TASEP with geometric jumps at time $n$. When $\tilde{p}<p$ the measures $M^\Phi_n$ are time $n$ distributions of the TASEP with geometric jumps, where the jumps of the first particle $Y_1$ have geometric distributions with rate $\tilde{p}$ instead of $p$.

\begin{rem} When $p=0$ the first three examples degenerate to extreme coherent systems on the positive part of Gelfand-Tsetlin graph and they fully describe its boundary, see \cite{Olsh01}. In this case the coherent systems are related to characters of the infinite unitary group, and the name for the third example comes from the connection to Plancherel measures on the unitary group. However, the coherent system from the last example does not have any analogue in the $p=0$ case: $\Phi(z)=\frac{1-p}{1-pz}$ degenerates to $1$ when $p=0$.
\end{rem}

  \section{Asymptotic analysis}\label{asyman}

 This section is dedicated to the asymptotic analysis of $G^{(0,-p)}_\lambda(1^N)$ and $g^{(p,0)}_\lambda(\chi_1,\dots, \chi_k)$ when $N\to\infty$ and $\lambda$ grows linearly with finitely many nonzero rows or columns. This analysis will be used later in Section \ref{finitesect} to study the systems $\{M^{\calA,\calB}_{n}\}_n$.

\subsection{Integral formulae for $G^{(0,\fb)}_\lambda$} \label{Section_contour_integral} Our asymptotic analysis is based on two contour integral formulae. The first formula comes from the fact that the functions $G_\lambda^{(0,\fb)}$ are degenerations of \emph{spin $q$-Whittaker functions} $\mathbb F_{\lambda/\mu}$, introduced in \cite{BW17}. Here we define these functions following the notation from \cite{BK21}. Let $s,\xi,q$ be a triple of complex parameters. For a single variable $x$ and a pair of partitions $\lambda,\mu$ set
$$
  \mathbb{F}_{\lambda/\mu}(x\mid\xi,s)=\begin{cases}
 (-x/\xi)^{|\lambda|-|\mu|}{\displaystyle \prod_{i\geq 1}}\dfrac{(x^{-1}s\xi;q)_{\lambda_i-\mu_i}(xs/\xi;q)_{\mu_i-\lambda_{i+1}}(q;q)_{\lambda_i-\lambda_{i+1}}}{(q;q)_{\lambda_i-\mu_i}(q;q)_{\mu_i-\lambda_{i+1}}(s^2;q)_{\lambda_i-\lambda_{i+1}}}\qquad& \text{if\ } \mu\preceq\lambda,\\
 0\qquad&\text{otherwise}.
 \end{cases}
$$
Here we use the $q$-Pochhammer symbol $(x;q)_n=\prod_{i=1}^n(1-xq^{i-1})$. Then for larger sets of variables $x_1, \dots, x_n$ we define $\mathbb F_{\lambda/\mu}(x_1,\dots, x_n\mid\xi,s)$ inductively using the expression above and the branching rule
 \begin{equation}
 \label{branchingsqW}
 \mathbb{F}_{\lambda/\mu}(x_1,\dots, x_n\mid\xi,s)=\sum_{\nu} \F_{\lambda/\nu}(x_1,\dots, x_{n-1}\mid\xi,s)\F_{\nu/\mu}(x_{n}\mid \xi, s).
 \end{equation}

\begin{theo}[{\cite[Theorem 8.1]{BW17},\cite[Theorem 5.13]{BK21}}]\label{integralsqW} Let $\lambda$ be a partition and $k$ be an integer such that $\lambda_1\leq k$. Assume that there exists a complex positively-oriented simple contour $\mathcal C$  such that
\begin{itemize}
\item $0$ and $s/\xi$ are inside the contour $\mathcal C$;
\item all points $(s\xi)^{-1}$ is outside of the contour $\mathcal C$;
\item the image $q\mathcal C$ of the contour $\mathcal C$ under the multiplication by $q$ is inside $\mathcal C$.
\end{itemize}
\begin{equation*}
\mathbb{F}_{\lambda}(x_1,\dots, x_n\mid\xi,s)\\
=\xi^{-|\lambda|}\oint_{\mathcal C}\frac{du_1}{2\pi\i u_1}\dots\oint_{\mathcal C}\frac{du_k}{2\pi\i u_k}\prod_{i<j}\frac{u_i-u_j}{u_i-qu_j}\prod_{i=1}^k\(\frac{(1-s\xi u_i)^{\lambda_i'-1}}{(u_i-\xi^{-1}s)^{\lambda'_i}}\prod_{j=1}^n\frac{1-u_ix_j}{1-u_i\xi s}\).
\end{equation*}
\end{theo}

Now we consider the degeneration where $s,\xi,q\to0$ in a way such that the ratio $s/\xi=-\fb$ remains constant. One can see that
$$
\lim_{\varepsilon\to0} \restr{\left( (-\xi)^{|\lambda|-|\mu|}\mathbb{F}_{\lambda/\mu}(x\mid\xi,s)\right)}{\substack{q=\xi=\varepsilon\\s=-\fb\varepsilon}}=\begin{cases}
 x^{|\lambda|-|\mu|}(1+x\fb)^{r(\mu/\overline{\lambda})}\qquad& \text{if}\ \mu\preceq\lambda,\\
 0\qquad&\text{otherwise}.
 \end{cases}
$$
Comparing with Proposition \ref{Gbranch} and using the branching rule we get
 $$
\lim_{\varepsilon\to0}  \restr{\left( (-\xi)^{|\lambda|-|\mu|}\mathbb{F}_{\lambda/\mu}(x_1, \dots, x_n\mid\xi,s)\right)}{\substack{q=\xi=\varepsilon\\s=-\fb\varepsilon}}=G^{(0,\fb)}_{\lambda/\mu}(x_1, \dots, x_n).
 $$
 Now we apply this degeneration to Theorem \ref{integralsqW}.
 \begin{cor} \label{Gint2} Let $\lambda$ be a partition and $k$ be an integer such that $\lambda_1\leq k$. Let $\mathcal C$ be a positively-oriented simple complex contour encircling both $0$ and $-\fb$. Then
 \begin{equation*}
G^{(0,\fb)}_{\lambda}(x_1,\dots, x_n)\\
=(-1)^{-|\lambda|}\oint_{\mathcal C}\frac{du_1}{2\pi\i u_1}\dots\oint_{\mathcal C}\frac{du_k}{2\pi\i u_k}\prod_{i<j}\frac{u_i-u_j}{u_i}\prod_{i=1}^k\(\frac{1}{(u_i+\fb)^{\lambda'_i}}\prod_{j=1}^n(1-u_ix_j)\).
\end{equation*}
 \end{cor}

Note that the complexity of the expression from Corollary \ref{Gint2} increases with the number of columns in $\lambda$. Our second contour integral expression behaves in the opposite way, with the number of integrals depending on the number of rows.

\begin{prop}\label{Gintl} Assume that $x_1, \dots, x_n, \fa, \fb$ are complex parameters satisfying $\max\{|\fa|,|\fb|\}<\min\{|x_i|^{-1}\}_i$. Let $\lambda$ be a partition, $k$ an integer such that $l(\lambda)\leq k$, and $\mathcal C$ be a positively-oriented simple contour encircling $0,\fa$ and $\fb$ and leaving $x_i^{-1}$ outside. Then
$$
G_\lambda^{(\fa,-\fb)}(x_1, \dots, x_n)=\oint_{\mathcal{C}}\cdots\oint_{\mathcal{C}} \frac{du_1}{2\pi \i}\dots \frac{du_k}{2\pi \i}\prod_{i<j}(u_i-u_j)\prod_{i=1}^k\left( \frac{1}{(u_i-\fa)^{\lambda_i}(u_i-\fb)^{k-i+1}}\prod_{j=1}^n\frac{1-\fb x_j}{1-u_ix_j}\right).
$$
\end{prop}
\begin{proof} We use the Jacobi-Trudi formula for $G^{(\fa,-\fb)}_\lambda$. First, note that the functions $f^{(m,l)}_k$ from Proposition \ref{GJT} are defined by converging sums when $\max(|\fa|,|\fb|)<\min_i |x_i|^{-1}$. Note that the generating function
$$
\sum_{k\in\mathbb Z} f_k^{(m,l)}(\mathbf x) u^k=\sum_k\sum_{\substack{a,b,c:\\a-b-c=k}}h_a(\mathbf x)h_b(\fb^l)h_c(\fa^m)u^{a-b-c}=\frac{1}{(1-\fa u^{-1})^m(1-\fb u^{-1})^l}\prod_{i=1}^n\frac{1}{1-ux_i}
$$
is analytic in the annulus $\max(|\fa|,|\fb|)<|u|<\min_i |x_i^{-1}|$. Hence, we get
$$
f_k^{(m,l)}(\mathbf x) =\oint_{C} \frac{du}{2\pi \i u} \frac{u^{-k}}{(1-\fa u^{-1})^m(1-\fb u^{-1})^l}\prod_{i=1}^n\frac{1}{1-ux_i},
$$
where $C$ is a circle around $0$ containing $\fa,\fb$ and leaving $x_i^{-1}$ outside.

Applying Proposition \ref{GJT} leads to
$$
G_\lambda^{(\fa,-\fb)}(\mathbf{x})=\prod_{i\geq 1}(1-\fb x_i)^k\det\left[\oint_{C} \frac{du}{2\pi \i u} \frac{u^{-\lambda_a+a-b}}{(1-\fa u^{-1})^{\lambda_a}(1-\fb u^{-1})^{b-a+1}}\prod_{i=1}^n\frac{1}{1-ux_i}\right]_{1\leq a,b\leq k}.
$$
We can rewrite it as
\begin{multline*}
\sum_{\sigma\in \mathfrak{S}_k}(-1)^\sigma\prod_{i\geq 1}(1-\fb x_i)^n\oint_{C}\cdots\oint_{C} \frac{du_1}{2\pi \i u_1}\dots \frac{du_k}{2\pi \i u_k}\prod_{a=1}^k\left( \frac{u_a^{-\lambda_a+a-\sigma(a)}}{(1-\fa u_a^{-1})^{\lambda_a}(1-\fb u_a^{-1})^{\sigma(a)-a+1}}\prod_{i=1}^n\frac{1}{1-u_ax_i}\right)\\
=\sum_{\sigma\in \mathfrak{S}_k}(-1)^\sigma\oint_{C}\cdots\oint_{C} \frac{du_1}{2\pi \i}\dots \frac{du_k}{2\pi \i}\prod_{a=1}^k\left( \frac{1}{(u_a-\fa)^{\lambda_a}(u_a-\fb)^{\sigma(a)-a+1}}\prod_{i=1}^n\frac{1-\fb x_i}{1-u_ax_i}\right)\\
=\oint_{C}\cdots\oint_{C} \frac{du_1}{2\pi \i}\dots \frac{du_k}{2\pi \i}\det\left[(u_a-\fb)^{k-b}\right]_{1\leq a,b\leq b}\prod_{i=1}^k\left( \frac{1}{(u_i-\fa)^{\lambda_i}(u_i-\fb)^{k-i+1}}\prod_{j=1}^n\frac{1-\fb x_j}{1-u_ix_j}\right).
\end{multline*}
The proof is finished by rewriting the determinant in the last expression using the Vandermonde determinant and deforming the integration contour to $\mathcal C$.
\end{proof}

 \begin{rem} The functions $G_\lambda^{(\fa,0)}$ can also be obtained as a degeneration of \emph{spin Hall-Littlewood} functions $\mathsf{G}_\lambda(x_1,\dots, x_n\mid \Xi, \mathsf{S})$ from \cite{BP16}. Comparing \cite[Theorem 4.14]{BP16} with \eqref{Gdef} one can show that
 $$
 \lim_{\varepsilon\to0}\varepsilon^{|\lambda|}\restr{\mathsf{G}_\lambda(x_1,\dots, x_n\mid \Xi, \mathsf{S})}{\substack{q=0,\ \xi_0=\xi_1=\dots=\varepsilon^{-1}\\s_0=0,\ s_1=s_2=\dots=\varepsilon\fa}}=G^{(\fa, 0)}_\lambda(x_1,\dots, x_n),
 $$
 where $l(\lambda)\leq n$ and in the left-hand side $\lambda$ is treated as a signature of length $n$. Then the integral formula for $G^{(\fa, 0)}_\lambda$ obtained in Proposition \ref{Gintl} can be seen as a degeneration of the integral formula for $\mathsf{G}_\lambda(x_1,\dots, x_n\mid \Xi, \mathsf{S})$ from \cite[Corollary 7.16]{BP16}.
  \end{rem}

 \subsection{$G^{(0,-p)}_\lambda(1^N)$ with $k$ rows} Fix $k\in\mathbb Z_{>0}$ and a sequence $\alpha_1\geq\alpha_2 \dots\geq \alpha_k> 0$. For our analysis we consider sequences of partition $\lambda(N)$ such that $l(\lambda(N))\leq k$ and
$$
N^{-\frac12}\frac{\lambda(N)_i-\alpha_i N}{\sqrt{\alpha_i(1+\alpha_i)}}\to x_i,
$$
where $\bx=(x_1, \dots, x_k)\in C$ for a fixed compact compact subset $C\in\R^k$, but we allow $\bx$ to vary for different choices of sequences $\lambda(N)$. Let $\chi_i=\frac{\alpha_i}{1+\alpha_i}-p$ and $\tilde \chi_i=\max(0, \chi_i)$. We assume that $\chi_i\neq 0$ for all $i$ and we use $\tilde k$ to denote the number of $i$ such that $\alpha_i>\frac{p}{1-p}$. For $\alpha>0$ let $m(\alpha)$ denote the number of times $\alpha$ appears in $\calA=(\alpha_1, \dots, \alpha_k)$, for $\chi>0$ let $m(\chi)$ denote the multiplicity of $\chi$ in $(\chi_1,\dots, \chi_k)$ and set
$$
d=\tilde k+\sum_{\alpha>\frac{p}{1-p}}\binom{m(\alpha)}{2}=\sum_{\alpha>\frac{p}{1-p}}\binom{m(\alpha)+1}{2}=\sum_{\chi>0}\binom{m(\chi)+1}{2}.
$$
Finally, fix a vector $\bz=(z_1, \dots, z_n)\in\C^n$ satisfying $|z_i|<1$ for all $i$. We allow $\bz$ to be empty when $n=0$.

\begin{theo}\label{steeptheo1} With the notation above:

(1) The limit
$$
\lim_{N\to \infty}G^{(0,-p)}_{\lambda(N)}(\bz, 1^{N-n})\ N^{\frac{d}{2}}\prod_{i=1}^k\left(\frac{1-\tilde \chi_i-p}{1-p}\right)^{N} (\tilde\chi_i+p)^{\lambda(N)_i}
$$
does not depend on $\lambda(N)_i$ for $i$ such that $\alpha_i<\frac{p}{1-p}$. When $\alpha_i<\frac{p}{1-p}$ for all $i$ the limit above is equal to $1$.

(2) When $\alpha_1,\dots, \alpha_k>\frac{p}{1-p}$:
\begin{multline}\label{steeptheoeq1}
G^{(0,-p)}_\lambda(\bz, 1^{N-n})\ N^{\frac{d}{2}}\prod_{i=1}^k\left(\frac{1-\tilde \chi_i-p}{1-p}\right)^{N}(\tilde \chi_i+p)^{\lambda(N)_i}\\
\to Z (2\pi)^{-\frac{k}{2}}\prod_{i=1}^{k}\frac{e^{-\frac{x_i^2}{2}}}{\sqrt{\alpha_i(1+\alpha_i)}}\prod_{\substack{i<j\\\alpha_i=\alpha_j}}\frac{x_i-x_j}{\sqrt{\alpha_i(1+\alpha_i)}}\prod_{i=1}^n\Phi^{\calA,\varnothing}(z_i),
\end{multline}
where
$$
 Z=\frac{\prod_{\chi>0} (\chi+p)^{\binom{m(\chi)+1}{2}}\prod_{\substack {i<j:\\ \chi_i>\chi_j}}(\chi_i-\chi_j) }{\prod_{i=1}^k\chi_i^{k-i+1}},\qquad \Phi^{\calA,\varnothing}(z)=\prod_{i=1}^k\frac{1-\frac{p}{1-p}(z-1)}{1-\alpha_i(z-1)}.
$$
Moreover, this convergence is uniform over choices of $\lambda(N)$ such that $\bx$ vary over the compact subset $C\subset\R^k$ and $\{\lambda(N)_i-N\alpha_i-\sqrt{N\alpha_i(1+\alpha_i)}x_i\}_{i,N}$ are uniformly bounded. In other words, for a compact subset $C\subset \R^k$ and a constant $M$ there exists a positive real sequence $(\Lambda_n)$ such that $\lim_{n\to\infty}\Lambda_n = 0$ and
\begin{multline*}
\Big|G^{(0,-p)}_\lambda(\bz, 1^{N-n})\ N^{\frac{d}{2}}\prod_{i=1}^k\left(\frac{1-\tilde \chi_i-p}{1-p}\right)^{N}(\tilde \chi_i+p)^{\lambda(N)_i}\\
- Z (2\pi)^{-\frac{k}{2}}\prod_{i=1}^{k}\frac{e^{-\frac{x_i^2}{2}}}{\sqrt{\alpha_i(1+\alpha_i)}}\prod_{\substack{i<j\\\alpha_i=\alpha_j}}\frac{x_i-x_j}{\sqrt{\alpha_i(1+\alpha_i)}}\prod_{i=1}^n\Phi^{\calA,\varnothing}(z_i)\Big|<\Lambda_n
\end{multline*}
for every sequence $\lambda(N)$ such that $|\lambda(N)_i-N\alpha_i-\sqrt{N\alpha_i(1+\alpha_i)}x_i|< M$ for some vector $\bx\in C$ and all $i, N$.
\end{theo}

We use a standard steepest descent argument to establish this asymptotic behavior. The idea is to localize the integral to a small neighborhood of a critical point of the integrand, see e.g. \cite{Copson} or \cite{Erd56} for the basics of the method. Our first step is to rewrite the integral expression from Proposition \ref{Gintl} in an exponential form. Fix the analytic branch of $\ln(u)$ on $\C\backslash\R_{\leq0}$ with real values on $\R_{>0}$ and set
$$
h_{t}(u)=-\ln(1-u)-t\ln(u),\qquad R_k(u_1, \dots, u_k)=\prod_{i<j}(u_i-u_j)\prod_{i=1}^k\frac{1}{(u_i-p)^{k-i+1}},
$$
$$
f(u)=\prod_{i=1}^n\frac{1-u}{1-p}\frac{1-pz_i}{1-uz_i}.
$$
To make our computations more compact we also set $r_i=\frac{\alpha_i}{1+\alpha_i}=\chi_i+p$ and $\tilde r_i=\max(r_i, p)$. Then by Proposition \ref{Gintl} we have
\begin{multline}
\label{int1}
\prod_{i=1}^k\left(\frac{1-\tilde \chi_i-p}{1-p}\right)^{N}(\tilde \chi_i+p)^{\lambda(N)_i}G_{\lambda(N)}^{(0,-p)}(\bz, 1^{N-n})\\
=\oint_{\mathcal{C}}\frac{du_1}{2\pi \i}\cdots\oint_{\mathcal{C}}\frac{du_k}{2\pi \i}R_k(\bu)\prod_{i=1}^k\exp\left(N \left(h_{\lambda(N)_i/N}(u_i)-h_{\lambda(N)_i/N}(\tilde \chi_i+p)\right)\right) f(u_i),
\end{multline}
where $\mathcal{C}$ is a simple positively-oriented contour encircling $0, p$ and leaving $1, \{z_i^{-1}\}_i$ outside.

The next step is to find suitable descent contours. Let $C_{\alpha}$ denote the positively-oriented circle centered at $0$ of radius $\frac{\alpha}{1+\alpha}$.

\begin{lem} \label{descent-lem1} Let $\alpha,\beta\in\R_{>0}$. Then $C_\alpha$ is a steep descent contour for $\Re[h_\beta(z)]$, namely, $\Re\left[h_\beta\left(\frac{\alpha}{1+\alpha}e^{i\theta}\right)\right]$ decreases for $\theta\in (0,\pi)$ and increases for $\theta\in(-\pi, 0)$.
\end{lem}
\begin{proof}
Set $r=\frac{\alpha}{1+\alpha}\in(0,1)$. It is enough to check that the derivative $\frac{d}{d\theta}\Re\left[h_\beta(re^{i\theta})\right]$ has the same sign as $-\sin(\theta)$. Note that $\Re[h_\beta(u)]=-\frac12\ln |1-u|^2-\beta\ln |u|$. Plugging $u=re^{i\theta}$, the derivative is
\begin{equation*}
\frac{d}{d\theta}\Re[h_\beta(re^{i\theta})]=-\frac{d}{d\theta}\frac12\ln |1-re^{i\theta}|^2=-\frac{d}{d\theta}\frac12\ln (1+r^2-2r\cos(\theta))=\frac{-\sin(\theta)}{1+r^2-2r\cos(\theta)}. \qedhere
\end{equation*}
\end{proof}

We also need several properties of $h_\alpha(u)$.

\begin{lem}\label{Taylor1} (1) Assume that $\alpha\in \R_{>0}$. We have $h'_\alpha(u^{crt})=0$ where $u^{crt}=\frac{\alpha}{1+\alpha}$.

(2)  Assume that $\alpha\in \R_{>0}$. Then $h_\alpha(\frac{\alpha}{1+\alpha})$ is the unique minimum of $h_\alpha(x)$ on $(0,1)$.

(3) Let $D$ be a convex compact set avoiding $0,1$. Then there exists a constant $K>0$ such that for any $z\in D$ and $\alpha\in \R_{>0}$ satisfying $u^{crt}=\frac{\alpha}{1+\alpha}\in D$ we have
\begin{equation}
\label{Taylor1eq}
\left|h_\alpha(u)-h_\alpha(u^{crt})-\frac{(1+\alpha)^3}{2\alpha}(u-u^{crt})^2\right|<K|u-u^{crt}|^3
\end{equation}
\end{lem}
\begin{proof}
The first two parts follow from the direct computation of $h'_\alpha$ resulting in
$$
h'_\alpha(u)=\frac{1+\alpha}{u(1-u)}\left(u-\frac{\alpha}{1+\alpha}\right).
$$
The last part follows from Taylor's theorem.
\end{proof}

\begin{proof}[Proof of Theorem \ref{steeptheo1}] Let
$$
I^{(N)}(\bu):=R_k(\bu)\prod_{i=1}^k \exp\left(N h_{\lambda(N)_i/N}(u_i)-N h_{\lambda(N)_i/N}(\tilde \chi_i+p)\right)f(u_i),
$$
so the left-hand side of \eqref{steeptheoeq1} is
\begin{equation}\label{proof1eq1}
N^{\sum_{\alpha>\frac{p}{1-p}}\binom{m(\alpha)+1}{2}}\oint_{\calC}\frac{du_1}{2\pi \i}\cdots\int_{\calC}\frac{du_k}{2\pi \i}I^{(N)}(\bu).
\end{equation}
Our proof is structured as follows: In Step 1 we show that the limit in \eqref{steeptheoeq1} does not depend on $\alpha_i<\frac{p}{1-p}$, so we can assume that $\alpha_i>\frac{p}{1-p}$ for all $i$. In Steps 2-3 we show that, up to an exponentially decaying error, the limit \eqref{steeptheoeq1} is given by
$$
N^{\sum_{\alpha>\frac{p}{1-p}}\binom{m(\alpha)+1}{2}}\int_{\chi_1+p-\i\varepsilon}^{\chi_1+p+\i\varepsilon}\frac{du_1}{2\pi \i}\cdots\int^{\chi_k+p+\i\varepsilon}_{\chi_k+p-\i\varepsilon}\frac{du_k}{2\pi \i}I^{(N)}(\bu),
$$
in other words, the only asymptotically meaningful part comes from the neighborhood of the critical points $\chi_i+p$ of $h_{\alpha_i}(u_i)$. In Step 4 we perform a change of variables which allows to take the limit as $N\to\infty$.

The uniform convergence part of the theorem will follow from the fact that all convergences and bounds in the proof will be uniform over all valid choices of $\lambda(N)_i$.

\emph{Step 1:} Our first goal is to get rid of all $\alpha_i\in (0, \frac{p}{1-p})$. We do it by induction: assume that $\alpha_k<\frac{p}{1-p}$ and, equivalently, $\chi_k<0$. Deform the contours in \eqref{int1} in the following way: if $\alpha_i>\frac{p}{1-p}$ let $C_i$ be $C_{\alpha_i}$, if $\alpha_i<\frac{p}{1-p}$ and $i\neq k$ set $C_i$ to be the circle $|z|=p+\delta$ for a small $\delta>0$ and  $C_k=C_{\alpha_k}$. Then
\begin{multline}\label{rescancel1}
\oint_{\calC}\frac{du_1}{2\pi \i}\cdots\int_{\calC}\frac{du_k}{2\pi \i}I^{(N)}(\bu)\\
=F+\oint_{{C}_1}\frac{du_1}{2\pi \i}\cdots\oint_{{C}_k}\frac{du_k}{2\pi \i}R_k(\bu)\prod_{i=1}^k\exp\left(N h_{\lambda{(N)}_i/N}(u_i)-N h_{\lambda{(N)}_i/N}(\tilde \chi_i+p)\right) f(u_i),
\end{multline}
where $F$ is the term coming from moving the contour of $u_k$ through the pole of $R_k(\bu)$ at $u_k=p$. Note that $R_k(\bu)$ and $\Phi^{\calA,\varnothing}(u_i)$ are uniformly bounded for $u_i\in C_i$ and $i=1,\dots, k$. By Lemma \ref{descent-lem1}
\begin{align*}
\left|\exp\left(h_{\lambda{(N)}_i/N}(u_i)-h_{\lambda{(N)}_i/N}(\tilde \chi_i+p)\right)\right|&\leq 1&\quad &\text{if}\ \alpha_i>\frac{p}{1-p};\\
\left|\exp\left(h_{\lambda{(N)}_i/N}(u_i)-h_{\lambda{(N)}_i/N}(\tilde \chi_i+p)\right)\right|&\leq \exp\left(h_{\lambda{(N)}_i/N}(p+\delta)-h_{\lambda{(N)}_i/N}(p)\right)&\quad &\text{if}\ \alpha_i<\frac{p}{1-p}, i\neq k;\\
\left|\exp\left(h_{\lambda{(N)}_k/N}(u_k)-h_{\lambda{(N)}_k/N}(\tilde \chi_k+p)\right)\right|&\leq \exp\left(h_{\lambda{(N)}_k/N}(r_k)-h_{\lambda{(N)}_k/N}(p)\right).&\quad &\\
\end{align*}
Note that $h_{\lambda{(N)}_i/N}(u)$ converges uniformly to $h_{\alpha_i}(u)$ on any compact avoiding $0$ and $1$ since $h_{\lambda{(N)}_i/N}(u)-h_{\alpha_i}(u)=(\alpha_i-\frac{\lambda{(N)_i}}{N})\ln(1-z)$. So
$$
h_{\lambda{(N)}_i/N}(p+\delta)-h_{\lambda{(N)}_i/N}(p)\to h_{\alpha_i}(p+\delta)-h_{\alpha_i}(p),
$$
$$
h_{\lambda{(N)}_k/N}(\chi_k+p)-h_{\lambda{(N)}_k/N}(p)\to h_{\alpha_k}(\chi_k+p)-h_{\alpha_k}(p).
$$
By second part of Lemma \ref{Taylor1} $h_{\alpha_k}(\chi_k+p)-h_{\alpha_k}(p)<0$, and picking $\delta>0$ small enough we can achieve
$$
h_{\alpha_k}(\chi_k+p)-h_{\alpha_k}(p)+\sum_{i:\chi_i<0, i<k} h_{\alpha_i}(p+\delta)-h_{\alpha_i}(p)<0.
$$
Hence, for sufficiently large $n$ and small enough $\delta>0$ we can find $d>0$ such that
$$
\left|\exp\left(N h_{\lambda{(N)}_i/N}(u_i)-N h_{\lambda{(N)}_i/N}(\tilde \chi_i+p)\right)\right|<e^{-dN}.
$$
So, in the right-hand side of \eqref{rescancel1} the integral over $C_1,\dots, C_k$ decays exponentially and the only non-vanishing summand is $F$. Computing the residue at $u_k=p$ and using
$$
f(p)=1,\qquad  \qquad \restr{(u_k-p)R_k(\bu)}{u_k=p}=R_{k-1}(u_1, \dots, u_{k-1}),
$$
we get
$$
F=\oint_{C_1}\frac{du_1}{2\pi \i}\cdots\oint_{{C}_{k-1}}\frac{du_{k-1}}{2\pi \i}R_{k-1}(u_1,\dots, u_{k-1})\prod_{i=1}^{k-1}\exp\left(N h_{\lambda{(N)}_i/N}(u_i)-N h_{\lambda{(N)}_i/N}(\tilde \chi_i+p)\right) f(u_i),
$$
which is exactly the expression from the theorem for $\tilde\lambda{(N)}=(\lambda{(N)}_1,\dots,\lambda{(N)}_{k-1})$. So, by induction on $k$, we can remove all rows with $\alpha_i<\frac{p}{1-p}$.

\emph{Step 2:} From now on we can assume that $\alpha_1, \dots, \alpha_k>\frac{p}{1-p}$. Then $\chi_1+p, \dots, \chi_k+p\in (p,1)$ and we can deform the contours in \eqref{int1} to $C_{\alpha_1}, \dots, C_{\alpha_n}$ respectively without picking up any residues. Let $C^\varepsilon_\alpha$ be the arc of $C_\alpha$ contained in the $\varepsilon$-neighborhood of $\frac{\alpha}{1+\alpha}$. Our next goal is to show that for sufficiently small $\varepsilon>0$ we can find $d_1, K_1>0$ such that
\begin{equation}
\label{concentrate1}
\left|\oint_{\mathcal{C}_{\alpha_1}}\frac{du_1}{2\pi \i}\cdots\oint_{\mathcal{C}_{\alpha_k}}\frac{du_k}{2\pi \i}I^{(N)}(\bu)-\int_{\mathcal{C}^{\varepsilon}_{\alpha_1}}\frac{du_1}{2\pi \i}\cdots\int_{\mathcal{C}^{\varepsilon}_{\alpha_k}}\frac{du_k}{2\pi \i}I^{(N)}(\bu)\right|<K_1e^{-Nd_1}.
\end{equation}
Note that $R_k(\bu)$ and $f(u_i)$ can be uniformly bounded when $u_i\in C_{\alpha_i}$, so we only need to bound
$$
\exp\left(\sum_{i=1}^N h_{\lambda{(N)}_i/N}(u_i)- h_{\lambda{(N)}_i/N}(\chi_i+p)\right).
$$
Let $u^{\varepsilon}_{i}, u^{-\varepsilon}_{i}$ denote the endpoints of $C^{\varepsilon}_{\alpha_i}$ with the positive and negative imaginary part respectively. Consider the minimum $d=\min_i\Re[h_{\alpha_i}(\chi_i+p)-h_{\alpha_i}(u^{\varepsilon}_i)]$. By Lemma \ref{descent-lem1} we have $d<0$. Then for $u\in C_{\alpha_i}\backslash C_{\alpha_i}^{\varepsilon}$ we have
$$
\Re\left[h_{\alpha_i}(u)-h_{\alpha_i}(\chi_i+p)\right]\leq -d.
$$
At the same time, $h_{\lambda{(N)}_i/N}(u)-h_{\lambda{(N)}_i/N}(\chi_i+p)$ converges uniformly to $h_{\alpha_i}(u)-h_{\alpha_i}(\chi_i+p)$ for $u\in C_{\alpha_i}$.  So for sufficiently large $N$ we have
$$
\left|(h_{\lambda{(N)}_i/N}(u)-h_{\lambda{(N)}_i/N}(\chi_i+p)) - (h_{\alpha_i}(u)-h_{\alpha_i}(\chi_i+p))\right|<d/2
$$
uniformly for all $i$ and $u\in C_{\alpha_i}$. In particular, for $u\in C_{\alpha_i}\backslash C_{\alpha_i}^{\varepsilon}$
$$
\Re\left[h_{\lambda{(N)}_i/N}(u)-h_{\lambda{(N)}_i/N}(\chi_i+p))\right]<\Re\left[h_{\alpha_i}(u)-h_{\alpha_i}(\chi_i+p))\right]+\frac{d}2\leq -\frac d2.
$$
Finally, for any $u\in C_{\alpha_i}$ we have $\Re\left[h_{\lambda{(N)}_i/N}(u)-h_{\lambda{(N)}_i/N}(\chi_i+p))\right]\leq 0$ by Lemma  \ref{descent-lem1}, so
$$
\exp\left(\sum_{i=1}^k N h_{\lambda{(N)}_i/N}(u_i)-N h_{\lambda{(N)}_i/N}(\chi_i+p)\right)<\exp\left(-\frac{Nd}2\right)
$$
as long as $\bu\in \mathcal{C}_{\alpha_1}\times\dots\times \mathcal{C}_{\alpha_k}\backslash \mathcal{C}^{\varepsilon}_{\alpha_1}\times\dots\times \mathcal{C}^{\varepsilon}_{\alpha_k}$. This implies \eqref{concentrate1}.

\emph{Step 3:} Now we want to show that, with sufficiently small $\varepsilon>0$, we have
\begin{equation}
\label{straighten1}
\left|\int_{\mathcal{C}^{\varepsilon}_{\alpha_1}}\frac{du_1}{2\pi \i}\cdots\int_{\mathcal{C}^{\varepsilon}_{\alpha_k}}\frac{du_k}{2\pi \i}I^{(N)}(\bu)-\int_{\chi_1+p-\i\varepsilon}^{\chi_1+p+\i\varepsilon}\frac{dz_1}{2\pi \i}\cdots\int^{\chi_k+p+\i\varepsilon}_{\chi_k+p-\i\varepsilon}\frac{dz_k}{2\pi \i}I^{(N)}(\bu)\right|<K_2e^{-Nd_2},
\end{equation}
where the second integral is taken over the corresponding vertical line segments. To prove it note that $I^{(N)}(\bu)$ is analytic on $|u_i-\chi_i-p|<\varepsilon$ for sufficiently small $\varepsilon$. So the segment $[\chi_i+p-\i\varepsilon, \chi_i+p+\i\varepsilon]$ can be deformed to the contour $D_i=[\chi_i+p-\i\varepsilon, u_i^{-\varepsilon}]\cup C_{\alpha_i}^\varepsilon\cup [u_i^{\varepsilon},\chi_i+p+\i\varepsilon]$ without changing the integral. Then, similarly to Step 2, it is enough to show exponential decay when $u_i\in [\chi_i+p\pm\i\varepsilon, u_i^{\pm\varepsilon}]$ for at at least one $i$. Since $R_k(\bu)$ and $f(u_i)$ are uniformly bounded for all considered values of $\bu$ and $\left|\exp\left(N h_{\lambda{(N)}_i/N}(u_i)-N h_{\lambda{(N)}_i/N}(\chi_i+p)\right)\right|\leq 1$ when $u_i\in C_{\alpha_i}$, it is enough to show exponential decay for $\exp\left(N h_{\lambda{(N)}_i/N}(u_i)-N h_{\lambda{(N)}_i/N}(\chi_i+p)\right)$ when $u_i\in [\chi_i+p\pm\i\varepsilon, u_i^{\pm\varepsilon}]$.

Reducing $\varepsilon$ if necessary, we apply Lemma \ref{Taylor1} to find $K$ such that \eqref{Taylor1eq} holds uniformly in $\varepsilon$-neighborhood of $\chi_i+p$ for every $i$. Note that $\mathrm{Arg}(u^{\pm\varepsilon}_i-\chi_i-p)\to\pm\pi/2$ as $\varepsilon\to 0$, so for sufficiently small $\varepsilon$ we have
$$
|\Re[u^{\pm\varepsilon}_i-\chi_i-p]|<\varepsilon/10, \qquad |\Im[u^{\pm\varepsilon}_i-\chi_i-p]|>\frac{9\varepsilon}{10}.
$$
Then all points $u\in[u^{\pm\varepsilon}_i, \chi_i+p\pm\varepsilon\i]$ satisfy
$$
\Re[u]\in(\chi_i+p-\varepsilon/5, \chi_i+p+\varepsilon/5),\qquad |\Im[u]|>\frac{9\varepsilon}{10}.
$$
From \eqref{Taylor1eq} we have
$$
\Re\left[h_{\alpha_i}(u)- h_{\alpha_i}(\chi_i+p)\right]<\Re\left[\frac{(1+\alpha_i)^3}{2\alpha_i}(u-\chi_i-p)^2+K(u-\chi_i-p)^3\right].
$$
From our assumptions on $u$ we get
$$
\Re (u-\chi_i-p)^2<-\frac{81}{100}\varepsilon^2+\frac{1}{25}\varepsilon^2<-\varepsilon^2/2,
$$
$$
\Re (u-\chi_i-p)^3\leq \varepsilon^3.
$$
Combining these bounds and using the uniform convergence $h_{\lambda{(N)}_i/N}\rightrightarrows  h_{\alpha_i}$, we can decrease $\varepsilon$ further to get
 $$
\Re\left[h_{\lambda{(N)}_i/N}(u)- h_{\lambda{(N)}_i/N}(\chi_i+p)\right]<-d_2
 $$
for some $d_2>0$ and sufficiently large $N$. This implies \eqref{straighten1}

\emph{Step 4:} From now fix $\varepsilon>0$ sufficiently small so that Steps 2,3 hold, as well as Lemma \ref{Taylor1}. We perform the change of variables $u_i=\chi_i+p+\frac{v_i}{\sqrt{N}}$ in the integral
$$
\int_{\chi_1+p-\i\varepsilon}^{\chi_1+p+\i\varepsilon}\frac{du_1}{2\pi \i}\cdots\int^{\chi_k+p+\i\varepsilon}_{\chi_k+p-\i\varepsilon}\frac{du_k}{2\pi \i}\exp\left(N\sum_{i=1}^n  h_{\lambda{(N)}_i/N}(u_i)- h_{\lambda{(N)}_i/N}(\chi_i+p)\right)R_k(\bu)\prod_{i=1}^k f(u_i).
$$
Let us see what happens with each term in the integral. First, $f(u_i)=\prod_{j=1}^m\frac{1-pz_j}{1-p}\frac{1-\chi_i-p-\sqrt{N}^{-1}v_i}{1-(\chi_i+p)z_j-\sqrt{N}^{-1}v_iz_j}$ converges to $f(\chi_i+p)$. Moreover, this convergence is uniform when $v_a$ are allowed to vary over compact subsets of $\i\R$. For $R_k(\bu)$ clearly
$$
\prod_{i=1}^k\frac{1}{(\chi_i+\sqrt{N}^{-1}v_i)^{k-i+1}}\to \prod_{i=1}^k\frac{1}{\chi_i^{k-i+1}},
$$
and
$$
\sqrt{N}^{\sum_{\alpha}\binom{m(\alpha)}{2}} \prod_{i<j}(u_a-u_b) \to \prod_{\substack{i<j:\\ \alpha_i=\alpha_j}}(v_a-v_b)\prod_{\substack{i<j:\\\alpha_i>\alpha_j}}(\chi_i-\chi_j).
$$
Both convergences are again uniform when $v_a$ are allowed to vary over compact subsets of $\i\R$.

Finally, using Taylor expansions of $\ln(u), \ln(1-u)$ in the $\varepsilon$-neighborhood of $\chi_i+p$ we have
\begin{multline*}
h_{\lambda{(N)}_i/N}(u_i)- h_{\lambda{(N)}_i/N}(\chi_i+p)=\ln(1-r_i)-\ln(1-u_i)+\frac{\lambda{(N)}_i}{N}\left(\ln(r_i)-\ln(u_i)\right)\\
=\left(\frac{1}{1-r_i}-\frac{\lambda(N)_i}{Nr_i}\right)(u_i-r_i)+\left(\frac{1}{2(1-r_i)^2}+\frac{\lambda(N)_i}{2Nr_i^2}\right)(u_i-r_i)^2+\Omega(u_i),
\end{multline*}
where $|\Omega(u_i)|<K|u_i-\chi_i-p|^3$ and we use $r_i=\chi_i+p$ to make expressions shorter. By direct computations
$$
N\left(\frac{1}{1-r_i}-\frac{\lambda({N})_i}{Nr_i}\right)(u_i-r_i)=\frac{1+\alpha_i}{\alpha_i}\frac{\alpha_i N-\lambda{(N)}_i}{\sqrt{N}}v_i\to-\sqrt{\frac{(1+\alpha_i)\alpha_i}{(\chi_i+p)^2}}x_iv_i,
$$
$$
N\left(\frac{1}{2(1-r_i)^2}+\frac{\lambda(N)_i}{2Nr_i^2}\right)(u_i-r_i)^2 =\left(\frac{1}{(1-r_i)^2}+\frac{\lambda{(N)}_i}{Nr_i^2}\right)\frac{v_i^2}{2}\to \frac{(1+\alpha_i)\alpha_i}{(\chi_i+p)^2}\frac{v_i^2}{2},
$$
$$
N|\Omega(u_i)|<\frac{K|v_i|}{\sqrt{N}}\to 0,
$$
with all convergences being uniform when $v_i$ are on a compact. Also note that when $u_i=\chi_i+p+\sqrt{N}^{-1}v_i$ we have
\begin{multline*}
N\Re\left[h_{\lambda{(N)}_i/N}(u_i)- h_{\lambda{(N)}_i/N}(\chi_i+p)\right]=-\left(\frac{1}{2(1-r_i)^2}+\frac{\lambda_i^{(N)}}{2Nr_i^2}\right)|v_i|^2+N\Re\left[\Omega(u_i)\right]\\
<-2d_3 |v_i|^2+N\Re\left[\Omega(u_i)\right]
\end{multline*}
for some $d_3>0$. Reducing $\varepsilon$ we can assume $K\varepsilon<d_3$, so $N|\Omega(u_i)|<d_3|v_i|^2$ for $v_i\in [-\i\varepsilon\sqrt{N},\i\varepsilon\sqrt{N}]$ and
$$
N\Re\left[h_{\lambda{(N)}_i/N}(u_i)- h_{\lambda{(N)}_i/N}(\chi_i+p)\right]<-d_3|v_i|^2.
$$
Summing it up, after the change of variables we get
$$
\sqrt{N}^{d}\int_{\chi_1+p-\i\varepsilon}^{\chi_1+p+\i\varepsilon}\frac{du_1}{2\pi \i}\cdots\int^{\chi_k+p+\i\varepsilon}_{\chi_k+p-\i\varepsilon}\frac{du_k}{2\pi \i}I^{(N)}(\bu)=\int_{-\i\sqrt{N}\varepsilon}^{\i\sqrt{N}\varepsilon}\frac{dv_1}{2\pi \i}\cdots\int^{\i\sqrt{N}\varepsilon}_{-\i\sqrt{N}\varepsilon}\frac{dv_k}{2\pi \i}\sqrt{N}^{d-k}I^{(N)}(\bu),
$$
where $\sqrt{N}^{d-k}I^{(N)}(\bu)$ is uniformly integrable on $[-\i\sqrt{N}\varepsilon,\i\sqrt{N}\varepsilon]^k$ and $\sqrt{N}^{d-k}I^{(N)}(\bu)$ converges uniformly when $v_a$ are on compact subsets of $\i\R$. This implies the convergence of the integral above to
$$
\prod_{i=1}^k\frac{f(\chi_i+p)}{\chi_i^{k-i+1}}\prod_{\substack{i<j:\\\chi_i>\chi_j}}(\chi_i-\chi_j)\int_{\i\R}\frac{dv_1}{2\pi \i}\cdots\int_{\i\R}\frac{dv_k}{2\pi \i}\exp\left(\sum_{i=1}^k-\sqrt{\frac{\alpha_i(1+\alpha_i)}{(\chi_i+p)^2}}x_iv_i+\frac{(1+\alpha_i)\alpha_i}{2(\chi_i+p)^2}v_i^2\right)\prod_{\substack{i<j:\\ \alpha_i=\alpha_j}}(v_i-v_j).
$$
The last integral factors into independent parts corresponding to each value of $\alpha_i$ and using Lemma \ref{gaussian} below we get the desired limit.
\end{proof}

\begin{lem}\label{gaussian} Let $\sigma>0$, $x_1, \dots, x_k\in\R$. Then
\begin{equation}\label{gaussianeq}
\int_{\i\R}\frac{dv_1}{2\pi\i}\dots\int_{\i\R}\frac{dv_k}{2\pi\i}\exp\left(\sum_{i=1}^k -\sigma x_iv_i+\frac{\sigma^2 v_i^2}{2}\right)\prod_{i<j}(v_i-v_j)=(2\pi)^{-\frac k2}\sigma^{-k(k+1)}\prod_{i=1}^ke^{-\frac{x_i^2}{2}}\prod_{i<j}(x_i-x_j).
\end{equation}
\end{lem}
\begin{proof}
After rescaling $v_i\mapsto \sigma^{-1}v_i$, the left-hand side of \eqref{gaussianeq} is equal
$$
\sigma^{-k(k+1)}\det \left[\int_{\i \R}e^{-x_iv_i+\frac{v_i^2}{2}}v_i^{k-j}\frac{dv_i}{2\pi\i}\right]_{1\leq i,j\leq k}.
$$
Due to the exponential decay, we can shift the contour of $v_i$ to $\i\R+ x_i$ and after the change of variables $v_i=\i t+x_i$ we get
$$
\sigma^{-k(k+1)}\det \left[\int_{\R}e^{ -\frac{x_i^2}{2} - \frac{t^2}{2}}(\i t+x_i)^{k-j}\frac{dt}{2\pi}\right]_{1\leq i,j\leq k}.
$$
Note that the $(i,j)$ entry in the matrix above is a polynomial of degree $k-j$ in $x_i$ and using column operations we can cancel out the lower degree terms, getting
$$
\sigma^{-k(k+1)}\prod_{i=1}^ke^{-\frac{x_i^2}{2}}\det \left[\int_{\R}e^{- \frac{t^2}{2}}x_i^{k-j}\frac{dt}{2\pi}\right]_{1\leq i,j\leq k}.
$$
Using the expression for Vandermonde determinant again yield the right-hand side of \eqref{gaussianeq}.
\end{proof}


 \subsection{$G^{(0,-p)}_\lambda(1^N)$ with $k$ columns} Now we consider the functions $G_{\lambda}^{(0,-p)}(1^N)$ as $N\to\infty$ when the number of columns of $\lambda$ is finite. This computation is similar to the case with finitely many rows.

Again, fix $k\in\mathbb Z_{>0}$ and a sequence $\beta_1, \beta_2, \dots, \beta_k$ such that
$$
1>\beta_1\geq \beta_2\geq \dots\geq \beta_k>0.
$$
We consider partition sequences $\lambda(N)$ such that $\lambda(N)_1\leq k$ and
$$
N^{-\frac12}\frac{\lambda(N)'_i-\beta_i N}{\sqrt{\beta_i(1-\beta_i)}}\to y_i\,
$$
where $\by=(y_1, \dots, y_k)\in \R^k$. 
Let $\chi_i=\frac{\beta_i-p}{1-\beta_i}$ and $\tilde \chi_i=\max(0, \chi_i)$. We again assume that $\chi_i\neq 0$ for all $i$ and use $\tilde k$ to denote the number of $i$ such that $\chi_i>0$ (equivalent to $\beta_i>p$), $m(\beta)$ to denote the number of times $\beta$ appears in $(\beta_1, \dots, \beta_k)$ and $\bz=(z_1, \dots, z_n)\in\C^n$ be a fixed vector satisfying $|z_i|<1$. Set
$$
d=\tilde k+\sum_{\beta>p}\binom{m(\beta)}{2}=\sum_{\beta>p}\binom{m(\beta)+1}{2}.
$$

\begin{theo}\label{steeptheo2} With the notation above:

(1) The limit
\begin{equation}\label{indepeqlim2}
\lim_{N\to \infty}G^{(0,-p)}_{\lambda(N)}(\bz, 1^{N-n})N^{\frac d2}\prod_{i=1}^k\frac{(\tilde \chi_i+p)^{\lambda(N)'_i}}{(1+\tilde \chi_i)^N}
\end{equation}
does not depend on $\lambda(N)'_i$ with $\beta_i<p$. When $\beta_i<p$ for all $i$ the limit above is $1$.

(2) When $\beta_1,\dots, \beta_k>p$:
\begin{equation}\label{steeptheoeq2}
G^{(0,-p)}_{\lambda(N)}(\bz, 1^{N-n})N^{\frac d2}\prod_{i=1}^k\frac{(\tilde \chi_i+p)^{\lambda(N)'_i}}{(1+\tilde \chi_i)^N}\to
Z\ (2\pi)^{-\frac{k}{2}}\prod_{i=1}^{k}\frac{e^{-\frac{y_i^2}{2}}}{\sqrt{\beta_i(1-\beta_i)}}\prod_{\substack{i<j\\\beta_i=\beta_j}}\frac{y_i-y_j}{\sqrt{\beta_i(1-\beta_i)}}\prod_{i=1}^n\Phi^{\varnothing, \calB}(z_i),
\end{equation}
where
$$
 Z=\frac{\prod_{\chi} (\chi+p)^{\binom{m(\chi)+1}{2}}\prod_{\substack {i<j:\\ \chi_i>\chi_j}}(\chi_i-\chi_j) }{\prod_{i=1}^k\chi_i^{k-i+1}},\qquad \Phi^{\varnothing, \calB}(z)=\prod_{i=1}^k\left(1+\frac{\beta_i-p}{1-p}(z-1)\right)  .
$$
Moreover, this convergence is uniform over sequences $(\lambda(N))_N$ such that $\by\in C$ for a compact subset $C\subset\R^k$ and $\lambda(N)'_i-N\beta_i-\sqrt{N\beta_i(1-\beta_i)}y_i$ are bounded by a uniform constant $M$.

\end{theo}

This is proved by the steepest descent analysis of the integral expression from Corollary \ref{Gint2}. Fix the analytic branch of $\ln(u)$ on $\C\backslash\R_{\leq0}$ with real values on $\R_{>0}$. Let
$$
h_t(u)=\ln(1-u)-t\ln(p-u),\qquad f(u)=\prod_{i=1}^n \frac{1-z_i u}{1-u}
$$
$$
R_k(u_1, \dots, u_k)=\prod_{i<j}(u_i-u_j)\prod_{i=1}^k\frac{1}{u_i^{k-i+1}}.
$$
Then
\begin{multline}
\label{int2}
\prod_{i=1}^k\frac{(\tilde \chi_i+p)^{\lambda(N)'_i}}{(1+\tilde \chi_i)^N}G_{\lambda(N)}^{(0,-p)}(\bz, 1^{N-n})\\
=\oint_{\mathcal{C}}\frac{du_1}{2\pi \i}\cdots\oint_{\mathcal{C}}\frac{du_k}{2\pi \i}R_k(\bu)\prod_{i=1}^k\exp\left(N \left(h_{\lambda(N)'_i/N}(u_i)-h_{\lambda(N)'_i/N}(-\tilde \chi_i)\right)\right) f(u_i),
\end{multline}
where $\calC$ is a positively oriented simple contour encircling $0$ and $p$. Let $C_{\beta}$ denote the positively-oriented circle centered at $p$ of radius $\frac{(1-p)\beta}{1-\beta}$. Note that $C_\beta$ passes through $p-\frac{(1-p)\beta}{1-\beta}=\frac{p-\beta}{1-\beta}$, which is the critical point of $h_\beta(z)$.

\begin{lem} \label{descent-lem2} Let $\beta, \tilde\beta\in(0,1)$. Then $\Re\left[h_{\tilde\beta}\left(p-\frac{(1-p)\beta}{1-\beta}e^{i\theta}\right)\right]$ decreases for $\theta\in (0,\pi)$ and increases for $\theta\in(-\pi, 0)$, so $C_\beta$ is a steep descent contour for $\Re[h_{\tilde\beta}(z)]$.
\end{lem}
\begin{proof}
Set $r=\frac{(1-p)\beta}{1-\beta}$. It is enough to check that the derivative $\frac{d}{d\theta}\Re\left[h_{\tilde\beta}(p-re^{i\theta})\right]$ has the same sign as $-\sin(\theta)$. This follows from the computation below:
$$
\frac{d}{d\theta}\Re[h_\beta(p-re^{i\theta})]=\frac{d}{d\theta}\frac12\ln \left|1+\frac{\beta e^{i\theta}}{1-\beta}\right|^2=\frac{d}{d\theta}\frac12\ln \left(1+\frac{\beta^2}{(1-\beta)^2}-\frac{2\beta}{1-\beta}\cos(\theta)\right)=\frac{-\sin(\theta)}{\frac{\beta}{1-\beta}+\frac{1-\beta}{\beta}-2\cos(\theta)}.
$$
\end{proof}

\begin{lem}\label{Taylor2} (1) Assume that $\beta\in (0,1)$. We have $h'_\beta(u^{crt})=0$ where $u^{crt}=\frac{p-\beta}{1-\beta}=p-\frac{(1-p)\beta}{1-\beta}$.

(2)  Assume that $\beta\in (0,1)$. Then $h_\beta(\frac{p-\beta}{1-\beta})$ is the unique minimum of $h_\beta(x)$ on $(-\infty,p)$.

(3) Let $D$ be a convex compact set avoiding $p,1$. Then there exists a constant $K>0$ such that for any $z\in D$ and $\beta\in (0,1)$ satisfying $u^{crt}=\frac{p-\beta}{1-\beta}\in D$ we have
\begin{equation}
\label{Taylor2eq}
\left|h_\beta(u)-h_\beta(u^{crt})-\frac{(1-\beta)^3}{2\beta(1-p)^2}(u-u^{crt})^2\right|<K(u-u^{crt})^3.
\end{equation}
\end{lem}
\begin{proof}
Follows from
\begin{equation*}
h'_\beta(u)=\frac{u-\frac{p-\beta}{1-\beta}}{(1-u)(p-u)},\qquad h''_\beta\left(\frac{p-\beta}{1-\beta}\right)=\frac{(1-\beta)^3}{\beta(1-p)^2}.\qedhere
\end{equation*}
\end{proof}

\begin{proof}[Proof of Theorem \ref{steeptheo2}] Our argument follows the same steps as the proof of Theorem \ref{steeptheo1}. Let
$$
I^{(N)}(\bu):=R_k(\bu)\prod_{i=1}^k \exp\left(N h_{\lambda(N)'_i/N}(u_i)-N h_{\lambda(N)'_i/N}(-\tilde \chi_i)\right)f(u_i),
$$
so we study the limit
\begin{equation}\label{studiedlimit2}
\lim_{n\to\infty} N^{\frac d2}  \oint_{\calC}\frac{du_1}{2\pi \i}\cdots\int_{\calC}\frac{du_k}{2\pi \i}I^{(N)}(\bu).
\end{equation}
In the first step we want to remove columns with $\beta_i<p$. Let $\beta_k<p$ and deform the integration contours to $C_1, \dots, C_k$ where $C_i=C_{\beta_i}$ if $\beta_i>p$ or $i=k$ and $C_i$ is the circle around $p$ of radius $p+\delta$ otherwise.
\begin{equation}\label{rescancel2}
\oint_{\calC}\frac{du_1}{2\pi \i}\cdots\int_{\calC}\frac{du_k}{2\pi \i}I^{(N)}(\bu)=F+\oint_{\mathcal{C}_1}\frac{du_1}{2\pi \i}\cdots\oint_{\mathcal{C}_k}\frac{du_k}{2\pi \i}R_k(\bu)\prod_{i=1}^k\exp\left(N h_{\lambda(N)'_i/N}(u_i)-N h_{\lambda(N)'_i/N}(-\tilde \chi_i)\right) f(u_i),
\end{equation}
where $F$ comes from the reside at $u_k=0$. To show that the integral over $C_1,\dots, C_k$ decays exponentially we use Lemma \ref{descent-lem2}, getting
$$
\Re\left[\sum_{i=1}^kh_{\lambda(N)'_i/N}(u_i)-h_{\lambda(N)'_i/N}(-\tilde \chi_i)\right]\leq h_{\beta_k}(-\chi_k)-h_{\beta_k}(0)+\sum_{\beta_i<p, i\neq k}h_{\beta_i}(-\delta)-h_{\beta_i}(0)+O\left(\frac{1}{\sqrt{N}}\right).
$$
Using the second part of Lemma \ref{Taylor2}, $h_{\beta_k}(-\chi_k)-h_{\beta_k}(0)<0$, and picking $\delta>0$ small enough we can achieve, uniformly over $u_i\in C_i$,
$$
\Re\left[\sum_{i=1}^kh_{\lambda(N)'_i/N}(u_i)-h_{\lambda(N)'_i/N}(-\tilde \chi_i)\right]<-d
$$
for sufficiently large $N$ and fixed $d>0$. Since $f(u)$ and $R_k(\bu)$ are uniformly bounded we get exponential decay of the integral over $C_1, \dots, C_k$. So, the only asymptotically meaningful part in the right-hand side of \eqref{rescancel2} is $F$. Computing the residue at the simple pole at $u_k=0$ we get
$$
\oint_{\calC}\frac{du_1}{2\pi \i}\cdots\int_{\calC}\frac{du_{k-1}}{2\pi \i}R_{k-1}(u_1, \dots, u_{k-1})\prod_{i=1}^{k-1} \exp\left(N h_{\lambda(N)'_i/N}(u_i)-N h_{\lambda(N)'_i/N}(-\tilde \chi_i)\right)f(u_i),
$$
which is the integral we study with the $k$th column of $\lambda(N)$ removed. By induction the limit \eqref{indepeqlim2} does not depend on the columns of $\lambda(N)$ with $\beta_i<p$.

When $\beta_1\geq\dots\geq \beta_k>p$ we can deform the integration contours in \eqref{studiedlimit2} to $C_{\beta_1}, \dots, C_{\beta_k}$ for $u_1, \dots, u_k$ respectively without crossing any singularities. Then, repeating the second and the third steps from Theorem \ref{steeptheo1}, we can use the steep descent property from Lemma \ref{descent-lem2} and Taylor approximation from Lemma \eqref{Taylor2} to get
\begin{equation}
\label{concentrate2}
\left|\oint_{\mathcal{C}_{\beta_1}}\frac{du_1}{2\pi \i}\cdots\oint_{\mathcal{C}_{\beta_k}}\frac{du_k}{2\pi \i}I^{(N)}(\bu)-\int_{-\chi_1+\i\varepsilon}^{-\chi_1-\i\varepsilon}\frac{du_1}{2\pi \i}\cdots\int_{-\chi_k+\i\varepsilon}^{-\chi_k-\i\varepsilon}\frac{du_k}{2\pi \i}I^{(N)}(\bu)\right|<Ke^{-Nd}
\end{equation}
for sufficiently small $\varepsilon>0$ and $K,d>0$. Note that the integrals over $[-\chi_i+\i\varepsilon, -\chi_i-\i\varepsilon]$ are oriented downwards.

Finally we perform the change of variables $u_i=-\chi_i-\frac{v_i}{\sqrt{N}}$ in the integral
\begin{equation}\label{tmp222}
\int_{-\chi_1+\i\varepsilon}^{-\chi_1-\i\varepsilon}\frac{du_1}{2\pi \i}\cdots\int_{-\chi_k+\i\varepsilon}^{-\chi_k-\i\varepsilon}\frac{du_k}{2\pi \i}\exp\left(N\sum_{i=1}^k  h_{\lambda(N)_i'/N}(u_i)- h_{\lambda(N)'_i/N}(-\chi_i)\right)R_k(\bu)\prod_{i=1}^k f(u_i).
\end{equation}
We have as $N\to\infty$
$$
f(u_i)=f\left(-\chi_i-\frac{v_i}{\sqrt{N}}\right)\to f(-\chi_i),
$$
$$
\prod_{i=1}^k\frac{1}{u_i^{k-i+1}}=\prod_{i=1}^k\frac{1}{(-\chi_i-\sqrt{N}^{-1}v_i)^{k-i+1}}\to (-1)^{\binom{k}{2}+k}\prod_{i=1}^k\frac{1}{\chi_i^{k-i+1}},
$$
$$
N^{\frac12\sum_{\beta}\binom{m(\beta)}{2}} \prod_{i<j}(u_a-u_b) \to(-1)^{\binom{k}{2}} \prod_{\substack{i<j:\\ \beta_i=\beta_j}}(v_i-v_j)\prod_{\substack{i<j:\\\beta_i>\beta_j}}(\chi_i-\chi_j).
$$
Using Taylor expansions for $\ln(p-u)$, $\ln(1-u)$ and $\lambda_i'(N)=\beta_iN+y_i\sqrt{N\beta_i(1-\beta_i)}+o(\sqrt{N})$ we get
\begin{multline*}
Nh_{\lambda(N)_i'/N}(u_i)-Nh_{\lambda(N)_i'/N}(-\chi_i)=N(\ln(1-u_i)-\ln(1+\chi_i)) -\lambda(N)'_i(\ln(p-u_i)-\ln(\chi_i+p))\\
\to-\sqrt{\frac{\beta_i(1-\beta_i)}{(\chi_i+p)^2}}y_iv_i+\frac{\beta_i(1-\beta_i)}{2(\chi_i+p)^2}v_i^2.
\end{multline*}
The uniform bounds and convergences from Step 4 of Theorem \ref{steeptheo1} are verified in exactly the same way, so the integral \eqref{tmp222} converges to
$$
\prod_{i=1}^k\frac{f(-\chi_i)}{\chi_i^{k-i+1}}\prod_{\substack{i<j:\\\chi_i>\chi_j}}(\chi_i-\chi_j)\int_{\i\R}\frac{dv_1}{2\pi \i}\cdots\int_{\i\R}\frac{dv_k}{2\pi \i}\exp\left(\sum_{i=1}^n-\sqrt{\frac{\beta_i(1-\beta_i)}{(\chi_i+p)^2}}y_iv_i+\frac{\beta_i(1-\beta_i)}{2(\chi_i+p)^2}v_i^2\right)\prod_{\substack{i<j:\\ \beta_i=\beta_j}}(v_i-v_j).
$$
The proof is concluded by applying Lemma \ref{gaussian}.
\end{proof}


 \subsection{Asymptotic behavior of $g^{(p,0)}_\lambda$} Our final goal is to study $g^{(p,0)}_\lambda(\chi_1, \dots, \chi_k)$ for a fixed sequence of variables $\chi_1\geq \chi_2\ge\dots \geq \chi_k>0$. Recall that $g^{(p,0)}_\lambda(\chi_1,\dots, \chi_k)$ vanish unless $l(\lambda)\leq k$. Assume that $\lambda(N)$ is a sequence of partitions of length $k$ such that
 $$
 \lambda(N)_i=\alpha_i N+ t_i\sqrt{N} +O(1),
 $$
 where $\alpha_1\geq\dots\geq \alpha_k>0$ and $\bt=(t_1,\dots, t_k)\in\R^k$ do not depend on $N$. We assume that $\chi_1,\dots, \chi_k$ and $\alpha_1,\dots, \alpha_k$ have exactly the same ordering in the sense that both are decreasing and $\alpha_i=\alpha_j$ if and only if $\chi_i=\chi_j$. Also, for any $\chi> 0$ let $m(\chi)$ be the number of times $\chi$ appears in $(\chi_1,\dots, \chi_k)$.

 \begin{theo}\label{gtheo} With the notation above we have
 $$
 N^{-\sum_{\chi}\frac12\binom{m(\chi)}{2}} \frac{g_{\lambda(N)}^{(p,0)}(\chi_1,\dots, \chi_k)}{\prod_{i=1}^k(\chi_i+p)^{\lambda(N)_i}}\to Z^{-1}\frac{1}{\prod_{\chi>0}\prod_{i=1}^{m(\chi)}(i-1)!} \prod_{\substack{i<j\\ \chi_i=\chi_j}}(t_i-t_j),
 $$
 where
 $$
 Z=\frac{\prod_{\chi} (\chi+p)^{\binom{m(\chi)+1}{2}}\prod_{\substack {i<j\\ \chi_i>\chi_j}}(\chi_i-\chi_j) }{\prod_{i=1}^k\chi_i^{k-i+1}}
 $$
 Moreover, the convergence above is uniform over sequences $\lambda(N)$ such that $\bt\in C$ for a compact subset $C\subset\R^k$ and $\lambda(N)_i-\alpha_i N-t_i\sqrt{N}$ is bounded by a uniform constant $M$.
 \end{theo}
 \begin{proof} We use the determinantal formula from Proposition \ref{detg}. Note that for sufficiently large $N$ we have $\lambda(N)_i>0$ for all $i$, so
 $$
 g_{\lambda(N)}^{(p,0)}(x_1,\dots, x_k)=\frac{\det\left[x_i^{k-j+1}(x_i+p)^{\lambda(N)_j-1}\right]}{\prod_{i<j}(x_i-x_j)}.
 $$
 Rewrite it as
 $$
  g_{\lambda(N)}^{(p,0)}(x_1,\dots, x_k)\prod_{i=1}^k(x_i+p)^{1-\lambda(N)_i}x_i^{-k+i-1}=\frac{\det\left[x_i^{i-j}(x_i+p)^{\lambda(N)_j-\lambda(N)_i}\right]}{\prod_{i<j}(x_i-x_j)}.
 $$
 Now we want to substitute $x_i=\chi_i$. However, we need to be careful when $\chi_i=\chi_j$ since both sides of the fraction on the right-hand side vanish in this case. To get around it we use l'H\^opital's rule in the following way: for $i=1,\dots, k$ let $\rho(i)$ denote the number of $j$ such that $j<i, \chi_i=\chi_j$. Then we set $x_i=\chi_i$ one by one from $i=1$ to $i=k$. At the step $i$ during this substitution, both sides of the fraction have zero of order $\rho(i)$ at $x_i=\chi_i$, so we apply the derivative $\frac{\partial^{\rho(i)}}{(\partial x_i)^{\rho(i)}}$ to both sides. In the end we arrive at
 \begin{equation}\label{firststepg}
 g_{\lambda(N)}^{(p,0)}(\chi_1,\dots, \chi_k)\prod_{i=1}^k(\chi_i+p)^{1-\lambda(N)_i}\chi_i^{-k+i-1}=\frac{\det X(N)}{\prod_{i=1}^{k}(-1)^{\rho(i)}\rho(i)!\prod_{\substack {i<j\\ \chi_i>\chi_j}}(\chi_i-\chi_j)},
 \end{equation}
 where $X(N)$ is the $k\times k$ matrix
 $$
 X(N)_{ij}=\restr{\frac{d^{\rho(i)}}{dx_i^{\rho(i)}}\left(x_i^{i-j}(x_i+p)^{\lambda(N)_j-\lambda(N)_i}\right)}{x_i=\chi_i}.
 $$

So, we need to analyze $\det X(N)$ as $N\to\infty$. Let $\mathcal P$ be the partition of $\{1,\dots, k\}$ with $i,j$ in the same part if and only if $\chi_i=\chi_j$, this partition has the form $[1; l_1]\sqcup [l_1+1, l_2]\sqcup\dots\sqcup [l_{s-1}+1, l_s]$ for some integers $l_1,\dots, l_s$. Let $\mathfrak{S}_k$ denote the group of permutations of $k$ and $\mathfrak{S}_{\mathcal P}\subset \mathfrak{S}_k$ denote the subgroup preserving each part of $\mathcal P$. Rewrite $\det X(N)$ as
$$
\det X(N)=\sum_{\sigma\in \mathfrak{S}_k}(-1)^\sigma X(N)_\sigma, \qquad X(N)_\sigma=\prod_{i=1}^k X(N)_{i,\sigma(i)}.
$$
Our first goal is to show that only $\sigma\in \mathfrak{S}_{\mathcal P}$ contribute to $\lim_{N} \det X(N)$. We have
$$
X(N)_{ij}=\restr{\frac{d^{\rho(i)}}{dx_i^{\rho(i)}}\left(x_i^{i-j}(x_i+p)^{\lambda(N)_j-\lambda(N)_i}\right)}{x_i=\chi_i}=\sum_{\substack{a,b\geq 0\\a+b=\rho(i)}}c_{i,j}^{a,b}(N)\ \chi_i^{i-j-a}(\chi_i+p)^{\lambda(N)_j-\lambda(N)_i-b},
$$
where $c_{i,j}^{a,b}(N)$ depend on $N$ as polynomials in $\lambda(N)_j-\lambda(N)_i$ of degree $b$. In particular, since $\lambda_i(N)=\alpha_i N+O(\sqrt{N})$ we have
$$
|X(N)_{ij}|<K_{ij}e^{N(\alpha_j-\alpha_i)\ln(\chi_i+p)+\varepsilon_{ij}\sqrt{N}}.
$$
for some constants $K_{ij}, \varepsilon_{ij}$. Note that these constants can be chosen uniformly for all valid choices of $\lambda(N)$ in the second part of the theorem. Combining these upper bounds we get for some $K_\sigma, \varepsilon_\sigma$
$$
|X(N)_{\sigma}|<K_{\sigma}\exp\left(\varepsilon_{\sigma}\sqrt{N}+N\sum_{i=1}^{k}(\alpha_{\sigma(i)}-\alpha_i)\ln(\chi_i+p)\right).
$$

We claim that for any $\sigma\notin \mathfrak{S}_{\mathcal P}$ the following inequality holds:
\begin{equation}\label{lemineq}
\sum_{i=1}^{k}\ln(\chi_i+p)(\alpha_{\sigma(i)}-\alpha_i)<0.
\end{equation}
Indeed, note that
$$
\sum_{i=1}^{k}\ln(\chi_i+p)(\alpha_{\sigma(i)}-\alpha_i)=\sum_{i=1}^{k-1} (\ln(\chi_i+p) -\ln (\chi_{i+1}+p))(\alpha_{\sigma[1,i]}-\alpha_{[1,i]}),
$$
where for $A\subset[1;k]$ we set $\alpha_{A}=\sum_{i\in A}\alpha_i$ and $\alpha_{\sigma A}=\sum_{i\in A}\alpha_{\sigma(i)}$. Since $\alpha_i$ and $\chi_i$ decrease, we have $\ln (\chi_i+p) -\ln (\chi_{i+1}+p)\geq 0$ and $\alpha_{[1,i]}\geq \alpha_{\sigma[1,i]}$ for every $i$, so each term in the sum $\sum_{i=1}^{k-1} (\ln (\chi_i+p) -\ln (\chi_{i+1}+p))(\alpha_{\sigma[1,i]}-\alpha_{[1,i]})$ is non-positive. Moreover, this sum is $0$ only if $\alpha_{[1,l_i]}=\alpha_{\sigma[1, l_i]}$ for each $i=1,\dots, s$. In other words, for any part $P\in\mathcal P$ we must have $\alpha_P=\alpha_{\sigma P}$, and by monotonicity this is only possible when $\sigma\in S_{\mathcal P}$.

Using \eqref{lemineq}, for any $\sigma\notin S_{\mathcal P}$ we can find $d_\sigma>0$ such that
$$
|X(N)_{\sigma}|<K_{\sigma}\exp\left(-d_\sigma N \right).
$$
Hence
$$
\lim_N N^{-\sum_{\chi}\frac12\binom{m(\chi)}{2}}\det X(N)=\lim_NN^{-\sum_{\chi}\frac12\binom{m(\chi)}{2}}\sum_{\sigma\in \mathfrak{S}_{\mathcal P}} (-1)^{\sigma}X(N)_{\sigma}=\lim_N \prod_{P\in\mathcal P}N^{-\frac12\binom{|P|}{2}} \det[X(N)_{ij}]_{i,j\in P}.
$$

To finish the proof we compute the limit of $N^{-\frac12\binom{|P|}{2}}\det[X(N)_{ij}]_{i,j\in P} $ for each part of the partition $\mathcal P$. Fix $P\in\mathcal P$, and let $\alpha_i=\alpha$ and $\chi_i=\chi$ for all $i\in P$. Note that for any $i,j\in P$ we have $\lambda(N)_j-\lambda(N)_i=(t_j-t_i)\sqrt{N}+O(1)$. Hence $c_{ij}^{ab}(N)$ is of order at most $N^{\frac{b}{2}}$ and
$$
N^{-\frac{\rho(i)}{2}}X(N)_{ij}=N^{-\frac{\rho(i)}{2}}c_{i,j}^{0,\rho(i)}(N)\chi_i^{i-j}(\chi_i+p)^{\lambda(N)_j-\lambda(N)_i-\rho(i)}+o(1),
$$
in other words, the only asymptotically meaningful part of $N^{-\frac{\rho(i)}{2}}X(N)_{ij}$ comes from applying $\frac{d^{\rho(i)}}{(dx_i)^{\rho(i)}}$ only to $x_i^{\lambda(N)_j-\lambda(N)_i}$. Hence
$$
N^{-\frac{\rho(i)}{2}}X(N)_{ij}\to (t_j-t_i)^{\rho(i)}\chi_i^{i-j}(\chi_i+p)^{\lambda(N)_j-\lambda(N)_i-\rho(i)}
$$
and we get
$$
\sqrt{N}^{-\binom{|P|}{2}}\det[X(N)_{ij}]_{i,j\in P} =\det[N^{-\frac{\rho(i)}{2}}X(N)_{ij}]_{i,j\in P}\to \det\left[(t_j-t_i)^{\rho(i)}\chi^{i-j}(\chi+p)^{\lambda_j(N)-\lambda_i(N)-\rho(i)}\right]_{i,j\in P}
$$
Using basic row and column manipulations, we get
$$
\lim_N \sqrt{N}^{-\binom{|P|}{2}}\det[X(N)_{ij}]_{i,j\in P} = (\chi+p)^{-\binom{|P|}{2}}\det\left[(t_j-t_i)^{\rho(i)}\right]_{i,j\in P}.
$$
Note that $(t_j-t_i)^{\rho(i)}=f_i(t_j)$ for some monic polynomials $f_i(t)$ of degree $\rho(i)$, hence
$$
\det\left[(t_j-t_i)^{\rho(i)}\right]_{i,j\in P}=\det\left[f_i(t_j)\right]_{i,j\in P}=\det\left[t_j^{i-1}\right]_{i\in [1,|P|],j\in P}=\prod_{\substack{i,j\in P\\ i<j}}(t_j-t_i).
$$

To sum everything up, we get
$$
N^{-\sum_{\chi}\frac12\binom{m(\chi)}{2}}\det X(N)\to \prod_{P\in \mathcal P} (\chi_P+p)^{-\binom{|P|}{2}} (-1)^{\binom{|P|}{2}} \prod_{\substack{i<j\\ \chi_i=\chi_j}}(t_i-t_j).
$$
The proof is finished by applying this limit to \eqref{firststepg}.
 \end{proof}

\section{Finite GT-type coherent systems}\label{finitesect}

In this section we study in close detail the finite GT-type coherent systems $M^{\calA, \calB}_n$ from the previous section. First we establish the limit law for $M^{\calA, \calB}_n$ and then we prove that these measures are extremal when either $\calA$ or $\calB$ is empty. Throughout this section $p\in (0,1)$ and  $\calA=(\alpha_1,\dots, \alpha_k)$, $\calB=(\beta_1,\dots, \beta_l)$ denote sequences satisfying \eqref{ineq}. We also use $s$ to denote the number of $i$ such that $\beta_i=1$.

\subsection{Limit law}
To describe the behavior of $M^{\calA, \calB}_n$ we use the following notation. For a sequence $\calA$ of length $k$ let $\calP_\calA$ denote the set partition of $\{1,\dots, k\}$ such that $i$ and $j$ are in the same part if and only if $\alpha_i=\alpha_j$. Similarly let $\calP_{\calB}$ denote the analogous set partition for $\calB$. Recall that the density function of ordered eigenvalues of GUE random $n\times n$ matrices is given by
$$
\rho^{GUE}_n(x_1, \dots, x_n)=\frac{1}{(2\pi)^{\frac n2} \prod_{i=1}^ni!}\prod_{i=1}^n e^{-x_i^2/2}\prod_{i<j}(x_i-x_j)^2,
$$
where $x_1\geq x_2\geq \dots \geq x_n$. Taking a copy of such densities for each part of $\calP_\calA$ we define
$$
\rho^{GUE}_{\calA}(x_1, \dots, x_k)=\frac{1}{(2\pi)^{\frac k2} \prod_{P\in\calP_\calA}\prod_{i=1}^{|P|}(i-1)!}\prod_{i=1}^k e^{-x_i^2/2}\prod_{\substack{i<j\\ \alpha_i=\alpha_j}}(x_i-x_j)^2,
$$
which we treat as a probability density with respect to the Lebesgue measure $dx_1\dots dx_k$ on the space $\Delta_{\calA}\subset\R^k$ of points $\bx\in\R^k$ such that $x_i\geq x_j$ when $i<j$ and $\alpha_i=\alpha_j$. Similarly we define $\rho^{GUE}_{\calB}(y_{s+1}, \dots, y_l)$, however we ignore indices $i$ with $\beta_i=1$:
$$
\rho^{GUE}_{\calB}(y_{s+1}, \dots, y_l)=\frac{1}{(2\pi)^{\frac {l-s}2} \prod_{P\in\tilde\calP_\calB}\prod_{i=1}^{|P|}(i-1)!}\prod_{i=s+1}^l e^{-y_i^2/2}\prod_{\substack{i<j\\ \beta_i=\beta_j<1}}(y_i-y_j)^2,
$$
where $\tilde\calP_{\calB}$ is the partition of $\{s+1, \dots, l\}$ obtained by removing from $\calP_{\calB}$ the part $\{1,\dots, s\}$ corresponding to $\beta_i=1$. Finally, let $\bx^{GUE}_{\calA}$, $\by^{GUE}_{\calB}$ denote the random vectors distributed according to $\rho^{GUE}_{\calA}$ and $\rho^{GUE}_{\calB}$ respectively.

\begin{theo}\label{limittheo} Let $\lambda(n)$ denote the random partition distributed according to $M^{\calA,\calB}_n$. Then $\lambda(n)'_1=\dots=\lambda(n)'_s=n$ almost surely and as $n\to\infty$
\begin{equation}\label{limitrows}
\left(\frac{\lambda(n)_1-\alpha_1n}{\sqrt{n\alpha_1(1+\alpha_1)}},\dots, \frac{\lambda(n)_k-\alpha_kn}{\sqrt{n\alpha_k(1+\alpha_k)}} \right) \to \bx^{GUE}_{\calA},
\end{equation}
\begin{equation}\label{limitcolumns}
\left(\frac{\lambda(n)'_{s+1}-\beta_{s+1}n}{\sqrt{n\beta_{s+1}(1-\beta_{s+1})}},\dots, \frac{\lambda(n)'_l-\beta_ln}{\sqrt{n\beta_l(1-\beta_l)}} \right) \to \by^{GUE}_{\calB},
\end{equation}
where both convergences are in distribution.\footnote{The convergences \eqref{limitrows} and \eqref{limitcolumns} are considered separately, we do not claim anything about the joint law.} In particular, in probability, $\frac{\lambda(n)_i}{n}\to \alpha_i$ for $i\in [1,k]$ and $\frac{\lambda(n)'_i}{n}\to \beta_i$ for $i\in [1,l]$.
\end{theo}
\begin{proof} First we reduce the problem to the case $s=0$. Note from Proposition \ref{hook} that $\lambda(n)$ is supported on partitions $\lambda\in\G^p_n$ with $\lambda_n\geq s$, and the distribution of $\lambda(n)-s^n$ is $M_n^{\calA,\tilde\calB}$, where $\tilde\calB=(\beta_{s+1},\dots, \beta_l)$. In particular, $\lambda(n)'_1=\dots=\lambda(n)'_s=n$ almost surely and replacing $\lambda(n)$ by $\lambda(n)-s^n$  reduces the problem to the case $s=0$. Now we consider several cases.

\emph{Case 1: $\calB=\varnothing$.} From Proposition \ref{hook} we have
$$
M^{\calA,\varnothing}_n(\lambda)= G_{\lambda}^{(0,-p)}(1^n) g_{\lambda}^{(p,0)}(\chi_1, \dots, \chi_k)\prod_{i=1}^k \left(\frac{1+\frac{p}{1-p}}{1+\alpha_i}\right)^n,
$$
where $\chi_i=\frac{\alpha_i-p(1+\alpha_i)}{1+\alpha_i}=\frac{\alpha_i}{1+\alpha_i}-p$ and $\lambda$ has at most $k$ rows. Perform the change of variables
\begin{equation}\label{change1}
\lambda(\bx)_i=n\alpha_i+ x_i\sqrt{n\alpha_i(\alpha_i+1)} , \qquad i=1,\dots, k
\end{equation}
and consider the resulting induced measure $M^{\calA,\varnothing}_n(\lambda(\bx))$ on $\Delta_{\calA}$, which is supported on the lattice $\alpha_i n+ x_i\sqrt{n\alpha_i(\alpha_i+1)} \in\mathbb Z$. \footnote{Note that for fixed $n$ there exist $\lambda\in\G^p_n$ such that the corresponding $\bx$ is outside of $\Delta_\calA$, so the induced measure on $\Delta_{\calA}$ lacks $M^{\calA,\varnothing}_n(\lambda)$ for some partitions $\lambda\in\G^p_n$. This is not an issue since we show that the induced measure $M^{\calA,\varnothing}_n(\lambda(\bx))$ converges weakly to GUE distribution, so the total probability of such "bad" $\lambda$ goes to $0$.}
Then we need to prove the following weak convergence of measures on $\Delta_{\calA}$:
$$
M^{\calA,\varnothing}_n(\lambda(\bx))\to \rho^{GUE}_{\calA}(x_1, \dots, x_k)dx_1\dots dx_k.
$$
To do it it is enough to prove that as $n\to\infty$
\begin{equation}\label{Aconv}
n^{k/2}\prod_{i=1}^k\sqrt{\alpha_i(1+\alpha_i)} M^{\calA,\varnothing}_n(\floor{\lambda(\bx)})\to \rho^{GUE}_{\calA}(x_1, \dots, x_k),
\end{equation}
where the convergence is uniform on compact subsets of $\Delta_{\calA}$ and for $\bx\in\Delta_{\calA}$ we define $\floor{\lambda(\bx)}$ by taking floor of each $\lambda(\bx)_i$ in \eqref{change1} .  

To show \eqref{Aconv} we use asymptotic analysis of Grothendieck polynomials from Section \ref{asyman}. Theorem \ref{steeptheo1} implies that as $n\to\infty$
\begin{equation*}
G^{(0,-p)}_\lambda(1^n)\ n^{\frac{k}{2}+\frac12\sum_{P\in\calP_\calA}\binom{|P|}{2}}\prod_{i=1}^k\left(\frac{1-\chi_i-p}{1-p}\right)^{n}(\chi_i+p)^{\lambda_i}\to
Z (2\pi)^{-\frac{k}{2}}\prod_{i=1}^{k}\frac{e^{-\frac{x_i^2}{2}}}{\sqrt{\alpha_i(1+\alpha_i)}}\prod_{\substack{i<j\\\alpha_i=\alpha_j}}\frac{x_i-x_j}{\sqrt{\alpha_i(1+\alpha_i)}},
\end{equation*}
where $Z>0$ is a constant depending only on $\calA$ and $\lambda=\floor{\lambda(\bx)}$ implicitly depends on $n$. With the same $\lambda$ and $Z$, Theorem \ref{gtheo} shows that
 $$
 n^{-\sum_{P\in\calP_\calA}\frac12\binom{|P|}{2}} \frac{g_{\lambda}^{(p,0)}(\chi_1,\dots, \chi_k)}{\prod_{i=1}^k(\chi_i+p)^{\lambda_i}}\to Z^{-1}\frac{1}{\prod_{P\in\calP_\calA}\prod_{i=1}^{|P|}(i-1)!} \prod_{\substack{i<j\\ \chi_i=\chi_j}}\sqrt{\alpha_i(1+\alpha_i)}(x_i-x_j).
 $$
Moreover, both convergences are uniform when $\bx$ varies over compact subsets of $\Delta_{\calA}$. Taking the product of these results, we get
$$
n^{\frac{k}{2}} G^{(0,-p)}_\lambda(1^n)g_{\lambda}^{(p,0)}(\chi_1,\dots, \chi_k)\ \prod_{i=1}^k\left(\frac{1-\chi_i-p}{1-p}\right)^{n}\to\frac{\prod_{i=1}^{k}\frac{e^{-\frac{x_i^2}{2}}}{\sqrt{\alpha_i(1+\alpha_i)}}\prod_{\substack{i<j\\\alpha_i=\alpha_j}}(x_i-x_j)^2}{ (2\pi)^{\frac{k}{2}}\prod_{P\in\calP_\calA}\prod_{i=1}^{|P|}(i-1)!}.
$$
Since $\frac{1-\chi_i-p}{1-p}=\frac{1}{(1-p)(1+\alpha_i)}=\frac{1+\frac{p}{1-p}}{1+\alpha_i}$, we get
$$
n^{\frac{k}{2}}\prod_{i=1}^k\sqrt{\alpha_i(1+\alpha_i)} G^{(0,-p)}_\lambda(1^n)g_{\lambda}^{(p,0)}(\chi_1,\dots, \chi_k)\ \prod_{i=1}^k\left(\frac{1+\frac{p}{1-p}}{1+\alpha_i}\right)^{n}\to\rho^{GUE}_{\calA}(x_1, \dots, x_k),
$$
which is exactly what we needed to prove.

\emph{Case 2: $\calA=\varnothing$.} This case is analogous to the previous one. From Proposition \ref{hook} we have
$$
M^{\varnothing,\calB}_n(\lambda)= G_{\lambda}^{(0,-p)}(1^n) g_{\lambda'}^{(p,0)}(\chi_1, \dots, \chi_l)\prod_{i=1}^l\left(\frac{1-\beta_i}{1-p}\right)^n,
$$
where $\chi_i=\frac{\beta_i-p}{1-\beta_i}$ and $\lambda$ has at most $l$ columns. Now we use the change of variables
\begin{equation}\label{change2}
\lambda(\by)'_i=n\beta_i + y_i\sqrt{n\beta_i(1-\beta_i)}, \qquad i=1,\dots, l.
\end{equation}
Then we need to prove as $n\to\infty$
\begin{equation}\label{Bconv}
n^{l/2}\prod_{i=1}^l\sqrt{\beta_i(1-\beta_i)} M^{\varnothing,\calB}_n(\floor{\lambda(\by)})\to \rho^{GUE}_{\calB}(y_1,\dots, y_l),
\end{equation}
where the convergence is uniform on compact subsets of $\Delta_{\calB}$ and $\floor{\lambda(\by)}$, by a slight abuse of notation, is obtained by taking the floor of the columns $\lambda(\by)_i'$. Using the results from Section \ref{asyman}, Theorem \ref{steeptheo2} implies
$$
G^{(0,-p)}_\lambda(1^{n})n^{\frac l2+\frac12\sum_{P\in\calP_\calB}\binom{|P|}{2}}\prod_{i=1}^l\frac{(\chi_i+p)^{\lambda'_i}}{(1+\chi_i)^n}\to
Z\ (2\pi)^{-\frac{l}{2}}\prod_{i=1}^{l}\frac{e^{-\frac{y_i^2}{2}}}{\sqrt{\beta_i(1-\beta_i)}}\prod_{\substack{i<j\\\beta_i=\beta_j}}\frac{y_i-y_j}{\sqrt{\beta_i(1-\beta_i)}},
$$
and from Theorem \ref{gtheo}
 $$
 n^{-\sum_{P\in\calP_\calB}\frac12\binom{|P|}{2}} \frac{g_{\lambda'}^{(p,0)}(\chi_1,\dots, \chi_l)}{\prod_{i=1}^l(\chi_i+p)^{\lambda'_i}}\to Z^{-1}\frac{1}{\prod_{P\in\calP_\calB}\prod_{i=1}^{|P|}(i-1)!} \prod_{\substack{i<j\\ \chi_i=\chi_j}}\sqrt{\beta_i(1-\beta_i)}(y_i-y_j).
 $$
In both relations $\lambda=\floor{\lambda(\by)}$ is given by \eqref{change2}, $Z>0$ is a constant depending only on $\calB$, and the convergences are uniform over compact subsets of $\Delta_\calB$. Taking the product, we get
$$
n^{\frac l2}G^{(0,-p)}_\lambda(1^{n})g_{\lambda'}^{(p,0)}(\chi_1,\dots, \chi_l)\prod_{i=1}^l\frac{1}{(1+\chi_i)^n}\to\frac{\prod_{i=1}^{l}\frac{e^{-\frac{y_i^2}{2}}}{\sqrt{\beta_i(1-\beta_i)}}\prod_{\substack{i<j\\\beta_i=\beta_j}}(y_i-y_j)^2}{ (2\pi)^{\frac{l}{2}}\prod_{P\in\calP_\calB}\prod_{i=1}^{|P|}(i-1)!}.
$$
Since $1+\chi_i=\frac{1-p}{1-\beta_i}$ this is equivalent to \eqref{Bconv}.

\emph{General case.} Using Proposition \ref{branchingMPhi} we have
$$
M^{(\calA,\calB)}_n(\lambda)=\sum_{\mu\in\G^p_n}M^{(\varnothing,\calB)}_n(\lambda/\mu)M^{(\calA,\varnothing)}_n(\mu)=\sum_{\mu\in\G^p_n}M^{(\calA,\varnothing)}_n(\lambda/\mu)M^{(\varnothing,\calB)}_n(\mu).
$$
In other words, $\lambda$ distributed according to $M^{(\calA,\calB)}_n$ can be sampled in two ways. On one hand, we can obtain $\lambda$ by first sampling $\mu$ according to $M^{(\calA,\varnothing)}_n(\mu)$ and then getting $\lambda$ using $M^{(\varnothing,\calB)}_n(\lambda/\mu)$. From Proposition \ref{specialization} and the argument from Proposition \ref{hook}, $M^{(\varnothing,\calB)}_n(\lambda/\mu)$ vanishes unless $\lambda/\mu$ is a union of $l$ vertical strips, in particular $\lambda_i-\mu_i\leq l$ for all $i$. Hence in the limit
$$
\lim_{n\to\infty}\left(\frac{\lambda(n)_1-\alpha_1n}{\sqrt{n\alpha_1(1+\alpha_1)}},\dots, \frac{\lambda(n)_k-\alpha_kn}{\sqrt{n\alpha_k(1+\alpha_k)}} \right)
$$
we can replace $\lambda$ by $\mu$ and from the first case
$$
\left(\frac{\mu_1-\alpha_1n}{\sqrt{n\alpha_1(1+\alpha_1)}},\dots, \frac{\mu_k-\alpha_kn}{\sqrt{n\alpha_k(1+\alpha_k)}} \right) \to \bx^{GUE}_{\calA}.
$$
On the other hand we can sample $\lambda$ by first sampling $\mu$ according to $M^{(\varnothing,\calB)}_n(\mu)$ and then getting $\lambda$ using $M^{(\calA,\varnothing)}_n(\lambda/\mu)$. The latter vanishes unless $\lambda/\mu$ is a union of $k$ horizontal strips, hence from the second case we get \eqref{limitcolumns}.
\end{proof}

\subsection{Extremality of $M_n^{\calA,\varnothing}$ and $M_n^{\varnothing,\calB}$} The other goal of this section is to prove the following theorem.

\begin{theo} \label{Theorem_extreme} Let $\calA=(\alpha_1,\dots, \alpha_k)$ and $\calB=(\beta_1,\dots,\beta_l)$ be sequences satisfying \eqref{ineq}. Then the coherent systems $M_n^{\calA,\varnothing}$ and $M_n^{\varnothing,\calB}$ are extreme.
\end{theo}

The remainder of this section is dedicated to the proof of this theorem. Recall from Proposition \ref{centrlacoheq} that coherent systems on $\G^p$ are equivalent to central measures on the space of paths $\calT$. Our approach to extremality relies on the following properties of central measures:

\begin{prop} \label{extrprop}Let $M$ be a central measure on $\calT$.

(1) For $M$-almost every path $t\in\calT$ the limit
$$
\lim_{N\to\infty}\frac{\dim^p_n(\lambda)\dim^p_{n,{N}}(\lambda, t_{N})}{\dim^p_{N}(t_{N})}=\lim_{N\to\infty}p^{\downarrow}_{N,n}(t_{N}, \lambda)
$$
exists for all $n\geq 0, \lambda\in\G^p_n$. Here $p^{\downarrow}_{N,n}$ is defined by \eqref{downpdef}.

(2) Assume that for $M$-almost every path $t\in\calT$ we have
\begin{equation}\label{extrassumption}
\lim_{N\to\infty}p^{\downarrow}_{N,n}(t_{N}, \lambda)=M_n(\lambda), \qquad \forall n\geq 0, \lambda\in\G^p_n.
\end{equation}
Then $M$ is extreme.
\end{prop}
\begin{proof} Our proof follows the lines of \cite[Proposition 10.8]{Olsh01}.

Call two paths $t,t'\in\calT$ $N$-equivalent if $t_m=t'_m$ for $m\geq N$. Let $\xi_N(t)$ denote the $N$-equivalence class of $t$. We say that $t,t'\in\calT$ are $\infty$ equivalent if $t,t'$ are $N$-equivalent for some $N$. Let $\mathcal B_{-N}$ denote the $\sigma$-algebra of Borel sets $S$ satisfying $\xi_N(t)\subset S$ for $t\in S$, similarly we define $\mathcal B_{-\infty}$. We have
$$
\mathcal B_{-\infty}\subset \dots \mathcal B_{-2}\subset \mathcal B_{-1}, \qquad \bigcup_{N\geq 1}\mathcal B_{-N}=\mathcal B_{-\infty}.
$$

Let $\psi$ be a bounded Borel function on $\calT$, which we treat as a random variable with respect to a central measure $M$. Define $\psi_N=\mathbb E_M[\psi\mid \mathcal B_{-N}]$ and $\psi_{\infty}=\mathbb E_M[\psi\mid \mathcal B_{-\infty}]$ denote the conditional expectations. Then $\psi_{-N}$ form a reverse martingale. By the reverse martingale convergence theorem \cite[Theorem XI.15]{Doob} we get
$$
\psi_{N}\to \psi_{\infty} \qquad M-\text{almost\ everywhere}.
$$

Now let $n\geq 0$, $\lambda\in \G^p_n$ and consider
$$
\psi^{\lambda}(t)=\begin{cases} 1\qquad &t_n=\lambda,\\
0\qquad &\text{otherwise.}
\end{cases}
$$
Since $M$ is central, the conditional expectations $\psi^\mu_{N}$ are given by
$$
\psi^\mu_{N}(t)=\sum_{t'\in\xi_N(t)} \psi^\mu(t')\frac{w^p(t_{\leq N}')}{\dim^p_N(t_N)},
$$
where we set $t'_{\leq N}=(t_1',\dots t'_N)\in\calT_{\leq N}$. When $N\geq m$ only the terms with $t'_m=\mu$ matter and we get
$$
\psi^\mu_{N}(t)=\frac{1}{\dim^p_N(t_N)}\sum_{\substack{\varnothing=\tau_0\to\dots\to \tau_n=\lambda\\\lambda=\tau_{n}\to\dots\to \tau_N=t_N}}w^p(\tau)=\frac{\dim^p_n(\lambda)\dim^p_{n,N}(\lambda, t_N)}{\dim^p_N(t_N)}=p^{\downarrow}_{N,n}(t_{N}, \lambda).
$$

Now we are ready to prove both parts of the statement. The first part follows from the fact that $p^{\downarrow}_{N,n}(t_{N}, \lambda)=\psi^\lambda_{N}(t)$ converges to $\psi^\lambda_{\infty}(t)$ for $M$-almost every path $t$, and there are countably many choices of $n\geq 0, \lambda\in\G^p_n$.
For the second part assume that $M=w P+(1-w) Q$ for $w\in (0,1)$ and central measures $P,Q$. Repeat the construction above treating $\psi$ as a random variable with respect to $P$ instead of $M$ and taking the corresponding conditional expectations, we use $\tilde\psi^\lambda=\psi^\lambda, \tilde\psi_N^\lambda=\mathbb E_P[\psi^\lambda\mid\mathcal B_{-N}]$ and $\tilde\psi^\lambda_\infty=\mathbb E_P[\psi^\lambda\mid\mathcal B_{-\infty}]$ to denote the resulting functions. Then for any $n\geq 0,\lambda\in\G^p_n$ we have
$$
P_n(\lambda)=\mathbb E_P[\tilde \psi^{\lambda}]=\mathbb E_P[\tilde \psi^{\lambda}_\infty].
$$
Since $P$ is absolutely continuous with respect to $M$, by \eqref{extrassumption} we have $\tilde \psi^{\lambda}_\infty(t)=\lim_{N\to\infty}p^{\downarrow}_{N,n}(t_{N}, \lambda)=M_n(\lambda)$ for $P$-almost every path $t\in\calT$, so $P_n(\lambda)=\mathbb E_P[\tilde \psi^{\lambda}_\infty]=M_n(\lambda)$. Repeating this argument for all $n,\lambda$ we get $P=M$.
\end{proof}

We also need the following convergence property.
\begin{prop}\label{uniconv} Let $n\geq 0$, $r\in (0,1)$. Then the series
$$
\sum_{\lambda\in\G^p_k}c_\lambda \frac{G_\lambda^{(0,-p)}(z_1, \dots, z_n)}{G_\lambda^{(0,-p)}(1^n)}
$$
converges uniformly as a function of complex variables $\{c_\lambda\}_{\lambda}$, $z_1,\dots, z_n$ satisfying $|c_\lambda|\leq 1$, $|z_i|\leq r$ for all $\lambda, i$. In particular, the series is analytic in $z_1,\dots, z_k$ on the open unit disk for fixed $\{c_\lambda\}_\lambda$ satisfying $|c_\lambda|\leq 1$ for all $\lambda$.
\end{prop}
\begin{proof}
It is enough to give a converging uniform upper bound. Note from Proposition \ref{Gbranch} that for partitions $\mu\preceq \lambda$ and $z\in\mathbb C$ with $|z|\leq r$ we have
$$
|G_{\lambda/\mu}^{(0,-p)}(z)|=|z|^{|\lambda|-|\mu|}|1-pz|^{r(\mu/\overline{\mu})}\leq r^{|\lambda|-|\mu|}(1+pr)^{l(\lambda)}.
$$
Using the branching rule of Proposition \ref{Gbranch}, for any $\lambda\in\G^p_{k}$ we get
$$
|G_{\lambda}^{(0,-p)}(z_1,\dots, z_k)|\leq\sum_{\varnothing=\lambda^{(0)}\preceq\dots\preceq \lambda^{(k)}=\lambda}\ \prod_{i=1}^{k}|G_{\lambda^{(i)}/\lambda^{(i-1)}}^{(0,-p)}(z_i)|\leq  r^{|\lambda|}(1+pr)^{k^2}\dim_k(\lambda),
$$
where $\dim_k(\lambda)=\dim^0_k(\lambda)=s_\lambda(1^n)$ is the number of partition sequences $\varnothing=\lambda^{(0)}\preceq\dots\preceq \lambda^{(k)}=\lambda$. Similarly we have
$$
|G_{\lambda/\mu}^{(0,-p)}(1)|=(1-p)^{r(\mu/\overline{\mu})}\geq (1-p)^k,
$$
so $|G_{\lambda}^{(0,-p)}(1^k)|\geq (1-p)^{k^2}\dim_k(\lambda)$. Hence
$$
\left|c_\lambda \frac{G_\lambda^{(0,-p)}(z_1, \dots, z_k)}{G_\lambda^{(0,-p)}(1^k)}\right|\leq  r^{|\lambda|}\left(\frac{1+pr}{1-p}\right)^{k^2}.
$$
The uniform convergence now follows from the convergence of $\sum_{\lambda\in\G^p_k} r^{|\lambda|}=\prod_{i=1}^k\frac{1}{1-r^i}$.
\end{proof}

\begin{proof}[Proof of Theorem \ref{Theorem_extreme}] Our plan is to combine Theorem \ref{limittheo} with the second part of Proposition \ref{extrprop}. Let $M^{\calA,\calB}$ denote the central measure corresponding to $M_n^{\calA,\calB}$ and let $s$ be the number of $i$ such that $\beta_i=1$.

{\bfseries Step 1:} For $\delta>0$ define $\Delta_{\calA,\delta}\subset\Delta_{\calA}$ as a subset of points $\bx=(x_1,\dots x_k)\in\Delta_{\calA}$ satisfying $|x_i|<\delta^{-1}$ for all $i$ and $|x_i-x_j|>\delta$ for all $i<j$. For $n\geq 0$ let $A_{n,\delta}\subset\G_n^p$ denote the set of $\lambda$ such that $l(\lambda)\leq k$ and
$$
\left(\frac{\lambda_1-n\alpha_1}{\sqrt{n\alpha_1(1+\alpha_1)}}, \dots, \frac{\lambda_k-n\alpha_k}{\sqrt{n\alpha_k(1+\alpha_k)}}\right)\in\Delta_{\calA,\delta}.
$$
We claim that for $M^{\calA,\varnothing}$-almost every path $t\in\calT$ we can find $\delta>0$ and a subsequence $(t_{N_k})$ such that $t_{N_k}\in A_{N_k,\delta}$ for all $k$. Indeed, assume $t$ is distributed according to $M^{\calA,\varnothing}$. Then for any $\delta>0$, $N<k$ we have $\P(\forall n>N\ t_n\notin A_{n,\delta})\leq \P(t_k\notin A_{k,\delta})$. Taking $k\to\infty$ and using Theorem \ref{limittheo} we get $\P(\forall n>N\ t_n\notin A_{n,\delta})\leq \P(\bx_{\calA}^{GUE}\notin\Delta_{\calA,\delta})$ so for any $\delta>0$ we have
$$
\P(\exists N: \forall n>N\ t_n\notin A_{n,\delta})\leq \P(\bx_{\calA}^{GUE}\notin\Delta_{\calA,\delta}).
$$
Since $\P(\bx_{\calA}^{GUE}\in\Delta_{\calA,\delta})\to 0$ as $\delta\to 0$, almost surely we can find $\delta>0$ such $t_n\in A_{n,\delta}$ for infinitely many $n$.

Continuing with this setup, let $\delta>0$ and $(t_{N_m})$ be a subsequence such that $t_{N_m}\in A_{N_m,\delta}$ for all $k\geq 0$. Then $\frac{\lambda_i-N_m\alpha_i}{\sqrt{N_m\alpha_i(1+\alpha_i)}}$ is bounded for every $i,m$ and, taking a subsequence of $(t_{N_m})$ if necessary, we can assume that $x_i=\lim_{m\to\infty}\frac{(t_{N_m})_i-N_m\alpha_i}{\sqrt{N_m\alpha_i(1+\alpha_i)}}$ exists for $i=1,\dots, k$. Note that all $x_i$ must be distinct since $|x_i-x_j|\geq \delta$ when $i\neq j$. Let $z_1,\dots, z_n$ be complex variables satisfying $|z_i|<1$ and consider
$$
\lim_{m\to\infty}\frac{G^{(0,-p)}_{t_{N_m}}(z_1,\dots, z_n, 1^{N_m-n})}{G^{(0,-p)}_{t_{N_m}}(1^{N_m})}.
$$
From Theorem \ref{steeptheo1} we have
\begin{multline*}
\lim_{m} N_m^{r}\ G^{(0,-p)}_{t_{N_m}}(z_1,\dots, z_n, 1^{N_m-n}) \prod_{i=1}^k\left(\frac{1}{(1-p)(1+\alpha_i)}\right)^{N_m}\left(\frac{\alpha_i}{1+\alpha_i}\right)^{(t_{N_m})_i}\\
\to Z (2\pi)^{-\frac{k}{2}}\prod_{i=1}^{k}\frac{e^{-\frac{x_i^2}{2}}}{\sqrt{\alpha_i(1+\alpha_i)}}\prod_{\substack{i<j\\\alpha_i=\alpha_j}}\frac{x_i-x_j}{\sqrt{\alpha_i(1+\alpha_i)}}\prod_{i=1}^k\Phi^{\calA,\varnothing}(z_i),
\end{multline*}
where $r$ is the number of pairs $i\leq j$ such that $\alpha_i=\alpha_j$, $Z>0$ is a constant depending only on $\calA$ and $\Phi^{\calA,\varnothing}(z)$ is defined by \eqref{Mab-def}. Since $x_i-x_j\neq 0$ we take a ratio of two such limits getting
\begin{equation}\label{stepA}
\lim_{m}\frac{G^{(0,-p)}_{t_{N_m}}(z_1,\dots, z_n, 1^{N_m-n})}{G^{(0,-p)}_{t_{N_m}}(1^{N_m})}=\prod_{i=1}^k\frac{\Phi^{\calA,\varnothing}(z_i)}{\Phi^{\calA,\varnothing}(1)}=\prod_{i=1}^k\Phi^{\calA,\varnothing}(z_i).
\end{equation}
To sum it up, for $M^{\calA,\varnothing}$-almost every path $t\in\calT$ we can find a subsequence $(t_{N_m})$ such that \eqref{stepA} holds.

{\bfseries Step 2:} We can repeat the argument above for $M^{\varnothing, \calB}$, with the only difference coming from the case $\beta_i=1$. For $\delta>0$ define $\Delta_{\calB,\delta}$ consisting of $\by=(y_{s+1},\dots y_l)\in\Delta_{\calB}$ satisfying $|y_i|<\delta^{-1}$ for all $i>s$ and $|y_i-y_j|>\delta$ for all $s<i<j$. Let $B_{n,\delta}\subset\G_n^p$ denote the set of $\lambda$ such that $\lambda_1\leq l$, $\lambda'_1=\dots=\lambda_s'=n$ and
$$
\left(\frac{\lambda_{s+1}'-n\beta_{s+1}}{\sqrt{n\beta_{s+1}(1-\beta_{s+1})}}, \dots, \frac{\lambda'_l-n\beta_l}{\sqrt{n\beta_l(1-\beta_l)}}\right)\in\Delta_{\calB,\delta}.
$$
Then, using Theorem \ref{limittheo} in the same way as in Step 1, for $M^{\varnothing,\calB}$-almost every path $t$ we can find $\delta>0$ and a subsequence $(t_{N_m})$ such that $t_{N_m}\in B_{N_m,\delta}$ for all $m$. Taking a further subsequence, we can assume that the limits $y_i=\lim_{m\to\infty}\frac{(t_{N_m})'_i-N_m\beta_i}{\sqrt{N_m\beta_i(1-\beta_i)}}$ exist for $i=s+1,\dots, l$.

Let $z_1,\dots, z_n$ be complex variables satisfying $|z_i|<1$. Note that the first $s$ columns of $t_{N_m}$ are frozen and have length $N_m$, hence
$$
\frac{G^{(0,-p)}_{t_{N_m}}(z_1,\dots, z_n, 1^{N_m-n})}{G^{(0,-p)}_{t_{N_m}}(1^{N_m})}=\frac{z_1^sz_2^2\dots z_n^sG^{(0,-p)}_{\tilde t_{N_m}}(z_1,\dots, z_n, 1^{N_m-n})}{G^{(0,-p)}_{\tilde t_{N_m}}(1^{N_m})},
$$
where $\tilde t_N=t_N-s^N$ is obtained by removing the frozen columns. Now we can apply Theorem \ref{steeptheo2} getting
\begin{multline*}
\lim_{m} N_m^{r}\ G^{(0,-p)}_{\tilde t_{N_m}}(z_1,\dots, z_n, 1^{N_m-n}) \prod_{i=s+1}^l\left(\frac{1-\beta_i}{1-p}\right)^{N_m}\left(\frac{\beta_i(1-p)}{1-\beta_i}\right)^{(t_{N_m})'_i}\\
\to Z (2\pi)^{-\frac{l}{2}}\prod_{i=s+1}^{k}\frac{e^{-\frac{y_i^2}{2}}}{\sqrt{\beta_i(1-\beta_i)}}\prod_{\substack{i<j\\\beta_i=\beta_j\neq 1}}\frac{y_i-y_j}{\sqrt{\beta_i(1-\beta_i)}}\prod_{i=1}^k\Phi^{\varnothing,\tilde\calB}(z_i),
\end{multline*}
where $r$ is the number of pairs $i\leq j$ such that $\beta_i=\beta_j\neq 1$, $Z>0$ is a constant depending only on $\calB$ and $\tilde\calB=(\beta_{s+1},\dots, \beta_l)$. Hence we get
$$
\lim_{m\to\infty}\frac{G^{(0,-p)}_{t_{N_m}}(z_1,\dots, z_n, 1^{N_m-n})}{G^{(0,-p)}_{t_{N_m}}(1^{N_m})}=\prod_{i=1}^nz_i^s\lim_{m\to\infty}\frac{G^{(0,-p)}_{\tilde t_{N_m}}(z_1,\dots, z_n, 1^{N_m-n})}{G^{(0,-p)}_{\tilde t_{N_m}}(1^{N_m})}=\prod_{i=1}^nz_i^s\Phi^{\varnothing,\tilde\calB}(z_i)=\prod_{i=1}^n\Phi^{\varnothing,\calB}(z_i).
$$

{\bfseries Step 3:} From now on the argument is identical for both $M^{\calA,\varnothing}$ and $M^{\varnothing,\calB}$, so we use $M, \Phi$ to either denote $M^{\calA,\varnothing}, \Phi^{\calA,\varnothing}$ or $M^{\varnothing,\calB}, \Phi^{\varnothing,\calB}$.

From the previous steps and the first part of Proposition \ref{extrprop} for $M$-almost every path $t\in\calT$ we have the following two conditions
\begin{itemize}
\item For every $n\geq 0$ and $\lambda\in \G_n^p$ the limit $\lim_{N\to\infty}p^{\downarrow}_{N,n}(t_N,\lambda)$ exists.
\item There exists a subsequence $(t_{N_m})$ such that for any collection of complex parameters $z_1,\dots, z_n$ satisfying $|z_i|<1$ we have
\begin{equation}\label{conditionont}
\lim_{m\to\infty}\frac{G^{(0,-p)}_{t_{N_m}}(z_1,\dots, z_n, 1^{N_m-n})}{G^{(0,-p)}_{t_{N_m}}(1^{N_m})}=\prod_{i=1}^n\Phi(z_i).
\end{equation}
\end{itemize}
By the second part of Proposition \ref{extrprop}, to finish the proof it is enough to check that $\lim_{N\to\infty}p^{\downarrow}_{N,n}(t_N,\lambda)=M_n(\lambda)$ for $n\geq 0, \lambda\in\G^p_n$ and a path $t$ satisfying \eqref{conditionont}. So from now on we consider only such paths.

Fix $n\geq 0$. Then for $N>n$ and complex parameters $z_1,\dots, z_n$ we can use the branching rule from Proposition \ref{branching-rules} to get
\begin{multline}\label{tmpidentityprelimit}
\sum_{\lambda\in\G_n^p}p^{\downarrow}_{N,n}(t_N, \lambda)\frac{G_{\lambda}^{(0,-p)}(z_1,\dots, z_n)}{G_{\lambda}^{(0,-p)}(1^n)}=\frac{\sum_{\lambda\in\G_n^p}G_{t_N/\lambda}^{(0,-p)}(1^{N-n})G_{\lambda}^{(0,-p)}(z_1,\dots, z_n)}{G_{t_N}^{(0,-p)}(1^N)}\\=\frac{G_{t_N}^{(0,-p)}(z_1,\dots, z_n, 1^{N-n})}{G_{t_N}^{(0,-p)}(1^N)}.
\end{multline}
Note that the sums above are finite since only terms with $\lambda\subset t_N$ do not vanish. Now let us assume that $|z_1|,\dots, |z_n|<1$ and take the limit of both sides as $N\to\infty$. First we deal with the left-hand side of \eqref{tmpidentityprelimit}. For $\lambda\in\G_{n}^p$ let $c_\lambda=\lim_{N\to\infty} p^{\downarrow}_{N,n}(t_N, \lambda)$, which exists by our assumptions on $t$. By Proposition \ref{uniconv} we can exchange the limit and the summation and get
$$
\lim_{N\to\infty}\sum_{\lambda\in\G_n^p}p^{\downarrow}_{N,n}(t_N, \lambda)\frac{G_{\lambda}^{(0,-p)}(z_1,\dots, z_n)}{G_{\lambda}^{(0,-p)}(1^n)}=\sum_{\lambda\in\G_n^p}c_\lambda\frac{G_{\lambda}^{(0,-p)}(z_1,\dots, z_n)}{G_{\lambda}^{(0,-p)}(1^n)}.
$$
Consider the right-hand side of \eqref{tmpidentityprelimit}. From the left-hand side we already know the limit as $N\to\infty$ exists, so we can take it along the subsequence $(t_{N_m})$ from \eqref{conditionont}. Recalling the definition of $M_n(\lambda)$ from \eqref{defofM}, we get
$$
\sum_{\lambda\in\G_n^p}c_\lambda\frac{G_{\lambda}^{(0,-p)}(z_1,\dots, z_n)}{G_{\lambda}^{(0,-p)}(1^n)}=\prod_{i=1}^n\Phi(z_i)=\sum_{\lambda\in\G_n^p}M_n(\lambda)\frac{G_{\lambda}^{(0,-p)}(z_1,\dots, z_n)}{G_{\lambda}^{(0,-p)}(1^n)}.
$$
Both sums above are analytic functions in $z_1,\dots, z_n$ on the unit disk, so the identity above holds in the space of formal power series in $\bz$ as well. Since $G_{\lambda}^{(0,-p)}(z_1,\dots, z_n)$ are linearly independent we get $M_n(\lambda)=c_\lambda=\lim_{N\to\infty} p^{\downarrow}_{N,n}(t_N, \lambda)$ for every $\lambda\in\G^p_n$.
\end{proof}


 \section{Open questions and future directions} \label{Section_open_questions}

In this section we discuss conjectures and questions related to the coherent systems on $\G^p$.

Our first conjecture describes extreme coherent systems supported on partitions with finitely many rows or columns.

\begin{conj}\label{conj1} Let $\{M_n\}_n$ be an extreme coherent system on $\G^p$ supported on partitions contained in the hook with $k$ infinite rows and $l$ infinite columns, that is, $M_n(\lambda)=0$ unless $\lambda_{k+1}\leq l$. Then $\{M_n\}_n$ must be of the form $\{M^{\calA,\calB}_n\}_n$ for $\calA=(\alpha_1,\ \dots, \alpha_k)$, $\calB=(\beta_1,\ \dots, \beta_k)$ satisfying
$$
\alpha_1\geq\alpha_2\geq\dots\geq\alpha_k\geq \frac{p}{1-p},\qquad 1\geq\beta_1\geq\beta_2\geq\dots\geq\beta_k\geq p.
$$
\end{conj}

Note that when either $k=0$ or $l=0$ we can almost prove Conjecture \ref{conj1} using methods of this paper. By \cite[Proposition 10.8]{Olsh01} the converse to the second part of Proposition \ref{extrprop} is true: if $M$ is an extreme central measure then for $M$-almost every path $(t_N)$ we should have
\begin{equation}\label{limitcondition}
\lim_{N\to\infty} p^{\downarrow}_{N,n}(t_N, \lambda)=M_n(\lambda), \qquad \lambda\in \G^p_n.
\end{equation}
From Proposition \ref{uniconv} and \eqref{tmpidentityprelimit} this implies
$$
\lim_{N\to\infty} \frac{G_{t_N}^{(0,-p)}(z_1,\dots, z_n, 1^{N-n})}{G_{t_N}^{(0,-p)}(1^N)}=\sum_{\lambda\in\G^p_n}\lim_{N\to\infty} p^{\downarrow}_{N,n}(t_N, \lambda)\frac{G_{\lambda}^{(0,-p)}(z_1,\dots, z_n)}{G_{\lambda}^{(0,-p)}(1^n)}=\sum_{\lambda\in\G^p_n}M_n(\lambda)\frac{G_{\lambda}^{(0,-p)}(z_1,\dots, z_n)}{G_{\lambda}^{(0,-p)}(1^n)}.
$$
where the convergence is uniform over complex $z_1,\dots, z_n$ satisfying $|z_i|< r<1$. So to prove Conjecture \ref{conj1} when $l=0$ it is enough to show that for any sequence $\lambda(N)$ of partitions of length $\leq k$ such that the limits $\lim\frac{\lambda(N)_i}{N}=\alpha_i$ exist we have
\begin{equation}
\label{conjlimit}
\lim_{N\to\infty} \frac{G_{\lambda(N)}^{(0,-p)}(z_1,\dots, z_n, 1^{N-n})}{G_{\lambda(N)}^{(0,-p)}(1^N)} =\begin{cases}\prod_{i=1}^n\Phi^{\tilde{\calA},\varnothing}(z_i)\qquad &\text{if}\ \alpha_1<\infty,\\1\qquad &\text{if}\ \alpha_1=\infty,\end{cases}
\end{equation}
where we allow the limits $\alpha_i$ to be $\infty$ and $\tilde{\calA}=(\tilde{\alpha_1},\dots, \tilde{\alpha_k})$ with $\tilde{\alpha_i}=\max(\alpha_i, \frac{p}{1-p})$. The asymptotic analysis in Section \ref{asyman} comes close to proving \eqref{conjlimit}, however due to technical limitations it does not cover the situations when $\alpha_1=\infty$, $\alpha_i=\frac{p}{1-p}$ or $\lambda(N)_i-\lambda(N)_{j}=\overline{o}(\sqrt{N})$ for a pair $i\neq j$. We believe that these technical limitations can be resolved and the steepest descent approach might be sufficient to prove \eqref{conjlimit} when either $k=0$ or $l=0$.

When $M_n$ is not supported on hook shapes our understanding of the situation is much less clear. We start from the simplest non-trivial example.

\begin{figure}[t]
  \includegraphics[width=0.47\textwidth]{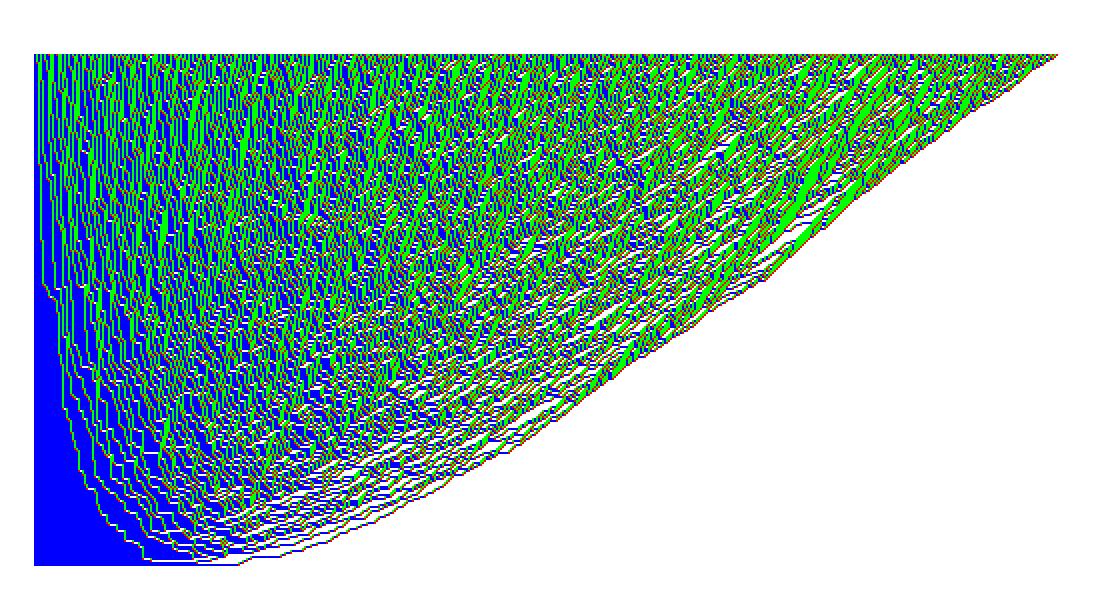}
   \hfill
  \includegraphics[width=0.47\textwidth]{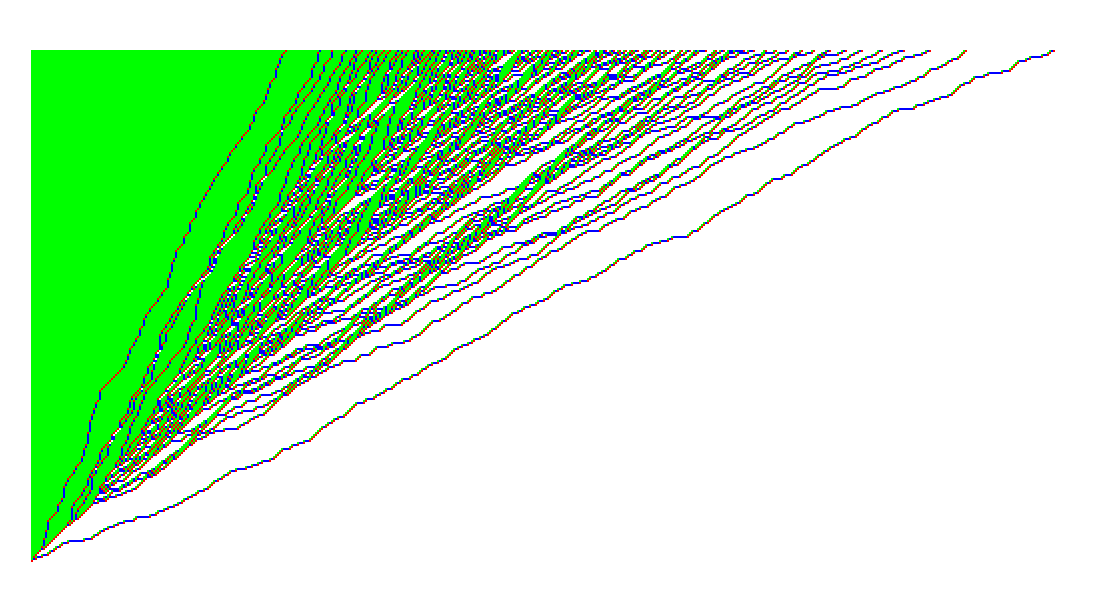}
\caption{A sample of the distribution $\P(\tau)=\frac{w^p(\tau)}{\dim^p_{N}(\lambda)}$ on paths $\tau=(\tau_0, \dots, \tau_{N})$ with $p=0.5$, $N=256$ and $\tau_{N}=\lambda(N)$. Left: $\lambda(N)_{i}=N-i$. Right $\lambda(N)_i= N\frac{p}{1-p}\left(1-\sqrt{\frac{i}{Np}}\right)^2$. In both cases we visualize the paths using the identification with the five-vertex model from Remark \ref{5-vertexrem}, with non-empty vertices of weights $p, 1-p, 1$ having blue, red and green colors respectively. In the left picture one can see a frozen region in the bottom-left corner. We expect this frozen region to grow linearly with $N$, so $\tau_1$ is going to grow linearly in this case and $p^{\downarrow}_{N,1}(\lambda(N), (i))\to0$ for any fixed $i$. The right picture corresponds to the limit shape of TASEP with geometric jumps, we expect this picture to converge to the system from Question \ref{TASEPq}.\label{simulationfig}}
\end{figure}

\begin{quest}\label{TASEPq} Let $\Phi(z)=\frac{1-p}{1-pz}$. Is the coherent system $\{M^{\Phi}_n\}_n$ extreme?
\end{quest}

Note that by \eqref{TASEPdegen} the coherent system in Question \ref{TASEPq} corresponds to TASEP with geometric jumps, which allows us to reformulate this question in terms of particle systems using \eqref{limitcondition}.

\begin{quest} Let $(Y(t))_{t\geq 0}$ denote TASEP with geometric jumps. Is it true that for fixed $n$ and $\lambda\in\G^p_n$ we have
$$
 \P(Y(n)=\lambda+\delta\mid Y(N))\to \P(Y(n)=\lambda+\delta)\qquad \text{almost\ surely\ as\ } N\to\infty?
$$
Here $\P(Y(n)=\lambda+\delta\mid Y(N))=\mathbb E[\1_{Y(n)=\lambda+\delta}\mid Y(N)]$ is treated as a random variable on the $\sigma$-algebra generated by $Y(N)$.
\end{quest}
We did not find such mixing questions for TASEP-like processes studied in the literature and we believe it could be an interesting problem. In the case of TASEP with geometric jumps we cautiously believe the answer to be positive, this guess is based on limited simulations of the process.

Our final open question is the description of the boundary of $\G^p$.

\begin{quest}\label{finalq}
Is it true that all extreme coherent systems on $\G^p$ have the form $\{M_n^\Phi\}_n$?
\end{quest}

From Proposition \eqref{extrprop} and \cite[Proposition 10.8]{Olsh01}, a closely related question is the description of paths $t=(t_N)_N\in\mathcal T$ such that the limit $\lim_{N\to\infty} p_{N,n}(t_N, \lambda)$ exists for every $\lambda\in\G^p_n$. This is also equivalent to existence of limits
$$
\lim_{N\to\infty} \frac{G_{t_N}^{(0,-p)}(z_1,\dots, z_n, 1^{N-n})}{G_{t_N}^{(0,-p)}(1^N)}
$$
for $|z_1|,\dots, |z_n|<1$. Unfortunately, we do not even have a good guess which paths $t$ lead to non-degenerate limits $\lim_{N\to\infty} p_{N,n}(t_N, \lambda)$, moreover, in a lot of cases these limits vanish like in Figure \ref{simulationfig}. To answer Question \ref{finalq} one likely either have to develop better asymptotic tools for studying Grothendieck polynomials or use a completely new approach to study the boundary of $\G^p$.

\bibliographystyle{alpha}
\bibliography{TASEPgraph.bib}

\end{document}